\documentclass[10pt]{article}

\usepackage[a4paper, margin=1.2in]{geometry}
\usepackage[title]{appendix}

\usepackage{amsmath} 
\usepackage{enumitem} 
\usepackage{amssymb, amsthm} 

\usepackage[dvipsnames]{xcolor}

         \newtheorem{theorem}{Theorem}[section]
	 \newtheorem{proposition}[theorem]{Proposition}
	 \newtheorem{lemma}[theorem]{Lemma}
         \newtheorem{corollary}[theorem]{Corollary}

\theoremstyle{definition}        
         \newtheorem{remark}[theorem]{Remark}

\numberwithin{equation}{section}

\newcommand \Abb {\mathbb{A}}
\newcommand \Cbb {\mathbb{C}}
\newcommand \Dbb {\mathbb{D}}
\newcommand \Nbb {\mathbb{N}}
\newcommand \Ibb {\mathbb{I}}
\newcommand \Lbb {\mathbb{L}}
\newcommand \Rbb {\mathbb{R}}
\newcommand \Zbb {\mathbb{Z}}
\newcommand \Tbb {\mathbb{T}}

\newcommand \be {\begin{equation}}
\newcommand \ee {\end{equation}} 

\newcommand\dif{h}

\def \t {\tilde} 

\def \v {\vec}

\newcommand \pa {\partial}

\newcommand \al {\alpha}
\newcommand \de {\delta}
\newcommand \ep {\epsilon}

\newcommand \ga {\gamma}
\newcommand \Ga {\Gamma}

\newcommand \na {\nabla}

\newcommand \si {\sigma}
\newcommand \ka {\kappa}
\newcommand \Om {\Omega}

\newcommand \Bcal {\mathcal{B}}
\newcommand \Ccal {\mathcal{C}}
 
\newcommand \Fcal {\mathcal{F}} 
\newcommand \Gcal {\mathcal{G}} 

\newcommand \Lcal {\mathcal{L}}
\newcommand \Scal {\mathcal{S}}

\newcommand \Ocal {\mathcal{O}}

\newcommand \Rcal {\mathcal{R}}
\newcommand \Pcal {\mathcal{P}}

\DeclareMathOperator\PV{PV}

\newcommand\low{{\mathrm{low}}}
\newcommand\high{{\mathrm{high}}}

\newcommand{\Ccont}{ C }

\renewcommand\tilde\widetilde

\title{Finite-time singularity formation\\ for angle-crested water waves}
\author{Diego C\'ordoba, Alberto Enciso and Nastasia Grubic}
\date{\normalsize{Instituto de Ciencias Matem\'aticas\\ Consejo Superior de
  Investigaciones Cient\'\i ficas\\ 28049 Madrid, Spain\\[1ex]
  E-mail: dcg@icmat.es, aenciso@icmat.es, nastasia.grubic@icmat.es}}

\begin{document}

\maketitle

\begin{abstract}
	We show that the water waves system is locally wellposed in weighted Sobolev spaces which allow for interfaces with corners. No symmetry assumptions are required. These singular points are not rigid: if the initial interface exhibits a corner, it remains a corner but generically its angle changes. Using a characterization of the asymptotic behavior of the fluid near a corner that follows from our a priori energy estimates, we show the existence of initial data in these spaces for which the fluid becomes singular in finite time.
\end{abstract}



\section{Introduction}

We consider the motion of an inviscid incompressible irrotational fluid in
the plane with a free boundary. A time-dependent interface
$$
\Ga(t) := \{z(\al,t) \,|\, \alpha\in \Rbb\}
$$ 
separates the plane into two open sets: the water region, which we denote by
$\Omega(t)$, and the vacuum region,
$\Rbb^2\backslash\overline{\Omega(t)}$. 
The
evolution of the fluid is described by the Euler equations,
\begin{subequations}\label{equationsI}
	\begin{align}
		\pa_t v + (v \cdot \nabla )v = - \nabla P - g \, e_2\quad &\text{in} \quad \Omega(t),\\
		\nabla\cdot v=0  \quad \mathrm{and} \quad \nabla^{\perp}\cdot v=0  \quad &\text{in} \quad \Omega(t), \\
		(\pa_t z- v) \cdot (\pa_\al z)^\perp = 0 \quad &\mathrm{on} \quad \Ga(t),\\
		p=0\quad &\mathrm{on} \quad \Ga(t).
	\end{align} 
\end{subequations}
Here $v$ and $p$ are the water velocity and pressure
on~$\Omega(t)$, $ e_2$ is the second vector of a Cartesian
basis and $g>0$ is the acceleration due to gravity. This system of PDEs
on~$\Rbb^2$, which are often referred to as the water wave equations,
can be formulated solely in terms of the interface curve, $z(\al,t)$,
and velocity of the fluid on the boundary.

The local wellposedness of the Cauchy problem for water waves was first established by Nalimov \cite{Na74}, Yoshihara \cite{Yo82,Yo83} and Craig \cite{Cr85} for small data and and by Wu \cite{Wu97,Wu99} for arbitrary data in Sobolev spaces. These results have been extended in various directions, see~\cite{ChLi00, Li05, La05, CoSh07, ZhZh08, CaCoFeGaGo13, AlBuZu14, KuTu14, HuIfTa16, HaIfTa17, AlBuZu18, Po19, Ai19, AiIfTa19, Ai20, SA23}. Local
wellposedness for initial data in low regularity Sobolev spaces was proven by
Alazard, Burq and Zuily~\cite{AlBuZu14,AlBuZu18} and subsequently refined by
Hunter, Ifrim and Tataru~\cite{HuIfTa16}. The lowest interface regularity they allow  is $C^{3/2}$.

Additionally, in all these works, a crucial assumption for the wellposedness of the problem in Sobolev spaces is the Rayleigh--Taylor sign condition, that is, the hypothesis that
\begin{align}\label{E.RTintro}
-\partial_n P \geq c >0 \quad \text{ on } \Gamma(t)\,,
\end{align}
where $n$ is the outward unit normal. The motivation for this condition is that, as the fluid is irrotational, the pressure satisfies the elliptic equation
\begin{equation}\label{E.subh}
	\Delta P = -|\nabla v|^2\leq0
\end{equation}
in the water region with the Dirichlet condition $P|_{\Ga(t)}=0$, so~\eqref{E.RTintro} follows from Hopf's boundary point lemma when $\Om(t)$ is regular enough to satisfy the interior ball condition. From the point of view of the energy estimates, $\partial_nP$ appears directly in the definition of the energy, and the sign of this function is directly related to the positivity of the energy.

A class of non-$C^1$ interfaces which has attracted considerable attention is that of piecewise smooth interfaces with angled crests. On one hand, the study of this kind
of solutions is motivated by experimental observations and hearkens back the XIX~century, when Stokes formally
constructed traveling wave solutions that featured sharp crests with
a $120^\circ$ corner. The existence of these traveling water waves was not rigorously established until the work of Amick, Fraenkel and
Toland~\cite{Amick} in the 1980s. On the other hand, the precise geometric control on the interface can be used to mitigate some of the analytic difficulties that appear due to the low regularity of the data (interface and velocity) and due to the failure of the Rayleigh--Taylor condition. In the context of the free boundary Euler equations with surface tension, other  kinds of singular interfaces were studied in~\cite{CEG,CEG2}.

Regarding the Cauchy problem for non-$C^1$ interfaces, Kinsey and Wu \cite{KiWu18} proved an a priori estimate for angled crested water waves, which Wu \cite{Wu19} subsequently used to prove local wellposedness in weighted Sobolev spaces that allow for interfaces with corners and cusps. Agrawal~\cite{SA23} proved estimates for the time of existence that are uniform as the gravity parameter $g\to0$. As showed in~\cite{Ag20}, the singular solutions constructed in these references are rigid, and in particular the angle of the corner does not change with time. 

In a recent paper~\cite{CoEnGr21}, we proved a local wellposedness result in another scale of weighted Sobolev spaces that allow for interfaces that have corners and cusps, with the property that the angle of the corner changes with time. Note that at the corners and cusps in \cite{Wu19, SA23, CoEnGr21}, the Rayleigh--Taylor condition is not satisfied, as ${\partial_n P}$ vanishes at those singularities.

A drawback of the vorticity-based approach we used to bypass the rigidity properties detailed in~\cite{Ag20} is that, in order to control the time evolution of the boundary vorticity, we have to consider solutions of Laplace equation for both the exterior and the interior domain. This is tricky because, depending on the size of the opening angle of the corner, solutions of Laplace equation can be very singular even for well-behaved boundary data. To bypass some of these problems, in~\cite{CoEnGr21} we had to impose fairly strong geometric assumptions, namely that the solutions were invariant under reflections across both Cartesian axes. 
A consequence of this is that, in particular, our approach in~\cite{CoEnGr21} forces us to disregard gravity, as the symmetry across the horizontal axis is not preserved whenever the gravity constant~$g$ is nonzero. This is in contrast to the approach developed to deal with rigid corners in~\cite{KiWu18}, which is based on mapping the singular water region conformally to the half-space and controlling the regularity of
the interface through weighted norms of the corresponding conformal map. In particular, only the interior domain (that is, the water region)  plays a role in the estimates. 

Our first objective in this paper is to prove the existence of non-rigid interfaces with corners in the case of water waves without any symmetry assumptions. We adopt the quasilinear reformulation of the problem in terms of interface parametrization and boundary velocity in the style of \cite{KiWu18, Wu19}, while using the weighted energy estimates developed in~\cite{CoEnGr21}. Our result can be informally stated as follows; a precise statement is presented as Theorem~\ref{thm:main} in Section~\ref{s.regularization}.

\begin{theorem}\label{T.main}
	The water wave system~\eqref{equationsI} is locally wellposed on a suitable scale of weighted Sobolev spaces that allows for interfaces with corners, without any symmetry requirements. If the initial interface exhibits a corner, it still has a corner for later times but generically its angle changes. Furthermore, a suitable analog of the Rayleigh--Taylor stability condition holds, but the the gradient of the pressure vanishes at the corner. 
\end{theorem}

A key contribution of our paper is that, as an essential part of the proof of the refined energy estimates that underlie Theorem~\ref{T.main}, we show that certain time-dependent quantities $(\nu_+,\nu_-,b_1)\in \Rbb\times\Rbb\times\mathbb{C}$ must satisfy the following system of ODEs:
\begin{gather}
\label{E.ODEintro}
	\frac{db_1}{dt} = \frac{|b_1|^2e^{-i(\nu_+ + \nu_-)}}{\cos(\pi+\nu_+-\nu_-)},\qquad
	\frac{d\nu_+}{dt} = \Re\left(i b_1e^{2i\nu_+}\right),\qquad \frac{d\nu_-}{dt} = \Re\left(ib_1 e^{2i\nu_-}\right).
\end{gather}
Geometrically, $b_1$ and $\nu_\pm$ respectively describe the asymptotic behavior of the gradient of the velocity and the jump in the tangent angle at the corner point, which in particular determine the opening angle (which we denote by $2\nu$) through the formula $2\nu= \pi +\nu_+-\nu_-$.

It is a striking feature of the water wave system that the evolution of these quantities can be computed by solving an ODE. Roughly speaking, the reason for this is that these quantities (which should be thought of as time-dependent constants) arise as the quotient between the scale of Sobolev spaces~$H^k_\gamma(m)$ that we need to use to describe fluids with angled-crested interfaces and a related scale of Sobolev spaces~$\Lcal^k_{2,\gamma}(m)$ on which certain linear operators that appear in the analysis of water waves with corners are bounded. Proving energy estimates in this setting  requires to control those constants separately,  and the fact that the evolution of these constants is effectively decoupled is due to the structure of solutions to the Laplace equation on domains with corners with acute angle.
Details are given in Lemmas~\ref{lem:param} and~\ref{lem:si}. 

As a consequence of this fact, we conclude that there are initial data which blow up in finite time. Indeed, as strongly suggested by the quadratic nonlinearity, there are initial data for which the ODE system~\eqref{E.ODEintro} blows up. The gist of the argument is that there are Cauchy data in our weighted Sobolev spaces that allow the initial values of the associated quantities $(\nu_\pm,b_1)$ to be chosen in such a way that the ODE actually blows up, while our degenerate analog of the Rayleigh--Taylor condition remains bounded away from zero. The result can be informally stated as follows; a precise version of this result is presented in Theorem~\ref{T.sing2}.

\begin{theorem}\label{T.sing}
	There are initial data in the weighted Sobolev spaces of Theorem~\ref{T.sing} for which the water wave system~\eqref{equationsI} blows up in some finite time~$T_{\max}$. More precisely, for all times in certain interval $[0,T_{\max})$ the solution is in a suitable weighted space and the interface exhibits a corner, but as $t\to T_{\max}$ either the interface self-intersects away from the corner point or the weighted Sobolev norm of the solution tends to infinity.
\end{theorem}

A couple of comments are in order. First, the Sobolev spaces in Theorems~\ref{T.main} and~\ref{T.sing} permit to prescribe interfaces with an arbitrary number of corners. This number does not change during the evolution (at least for short times), and there is an ODE system of the form~\eqref{E.ODEintro} associated with each corner. Here we restrict ourselves to the case of just one corner point for the sake of concreteness. Second, for data in the spaces we consider in this paper, the velocity is of class $C^{1,\lambda}$~up to the boundary inside the water region, for some $\lambda>0$. Moreover, by choosing weighted Sobolev spaces of a sufficiently high order, we can can ensure that the interface and velocity in the water region are arbitrarily smooth away from the corner. This is particularly relevant in the context of Theorem~\ref{T.sing}. 

Concerning singularity formation for water waves, it is well known that the equations can develop splash singularities in finite time~\cite{CaCoFeGaGo13}. The two essential features of this scenario of singularity formation are that the solution remains smooth up to the singular time, and that the interface intersects itself without pinching the water region~\cite{26,23}. Singularities where one does not control the regularity of the solution up to the blowup time have been constructed dynamically by Agrawal~\cite{SA23}. The essence of his argument is that one can take initial data that feature a rigid cusp and a rigid corner, which move in opposite directions. Using novel uniform estimates, he can show that either the cusp and the corner collide in finite time or some norm blows up before that. Therefore, even though the implementation and the estimates required are completely different, the strategy behind his scenario is somewhat similar to that of the work of Kiselev, Ryzhik, Yao and Zlatos on singularities for the modified SQG patch equation in the presence of boundaries~\cite{Yao} and by Coutand for the two-phase vortex sheet problem with surface tension and a solid-fluid system~\cite{Cou}. 

Our scenario is based on a different strategy, which in particular works even with just one corner point on the interface. As mentioned above, the basic idea is that there are initial conditions for which the ODE system~\eqref{E.ODEintro} blows up in finite time. Specifically, there are initial conditions and some finite~$T_0>0$ such that $|b_1(t)|\to\infty$ and the corner angle $2\nu(t)$ tends to~$\frac\pi2$ as $t\to T_0$. Since $b_1(t)$ is, in complex notation, the limit $\frac{d v^*}{dz}(t,z)$ as $z$~approaches the corner point, it is not hard to prove that the gradient of the velocity will become unbounded at time~$T_0$ unless the weighted Sobolev norm blows up or the interface self-intersects away from the corner before that time. Heuristically, this scenario of singularity formation is based on the fact that harmonic functions satisfying Neumann boundary conditions on a conical region with angle greater than~$\frac\pi2$ are typically not~$C^2$ at the corner point. This is relevant to water waves because the velocity field~$v$ is given by the gradient of a harmonic function precisely of this kind (the so-called velocity potential), so the $C^2$-norm of the velocity potential corresponds to the $L^\infty$-norm of the gradient of the velocity, $\nabla v$. Therefore, intuitively speaking, the geometric idea behind our blowup result is to exploit that our singularities are not rigid to identify initial data for which the corner angle goes to~$\frac\pi2$ in finite time, which eventually implies that $|\nabla v|$ tends to infinity in finite time.

Let us now discuss the organization of the paper. In Section \ref{s.preliminaries}, we first rewrite the water wave system as a system of equations for the parametrization of the interface and for the velocity on the boundary, in the Lagrangian parametrization. Next, we introduce the scales of weighted Sobolev spaces that we used throughout the paper and prove estimates for the Cauchy singular integral with data on these spaces. In Section \ref{s.apriori} we prove the basic a priori energy estimate that lie at the heart of the local wellposed theorem, deriving the ODE system~\eqref{E.ODEintro} along the way. Section \ref{s.regularization} is devoted to the proof of Theorem~\ref{T.main}. Because of the singular weights that appear eveywhere in this problem, the proof involves two regularizations of the water wave equations using variable-step convolutions. The proof of Theorem~\ref{T.sing} is presented in Section~\ref{S.sing}. In the closing sections of the paper we present a number of technical results that are used repeatedly throughout the paper. Specifically, in Section~\ref{s.inverse} we prove estimates for harmonic functions and for the Dirichlet-to-Neumann map on domains with corners with sharp dependence on the weighted Sobolev norms of the boundary. There are also two appendices. Appendix~\ref{A.A} presents various auxiliary results about Hardy-type inequalities, commutators and variable-step convolutions, while Appendix~\ref{A.ODE} is devoted to the analysis of the ODE system~\eqref{E.ODEintro}.


\section{Preliminaries}
\label{s.preliminaries}

In this section we shall start by introducing some notation to describe water waves with angled interfaces and to recall some well known equations. Next we will define two scales of Sobolev spaces, $H^k_\ga(m)$ and~$\Lcal^k_{2,\gamma}(m)$, which we will use throughout the paper and then we move on to prove some estimates on these spaces for the Cauchy singular integral operator for curves with corners.

\subsection{Definitions and notation}
\label{SS.defs}

We consider the free boundary Euler equations for an incompressible irrotational fluid. The fluid velocity $v$ and the pressure $P$ satisfy the equations
\begin{subequations}\label{eqs:pre}
	\begin{align}
		\pa_t v + (v \cdot \nabla) v = -\nabla P - g  e_2,  \label{eq:dP} \\
		\nabla\cdot v=0,  \quad \nabla^{\perp}\cdot v=0,  \label{eq:Inc-Irr}
	\end{align} 
\end{subequations}
in the water domain $\Om\equiv \Om(t)$, where $ e_2 := (0,1)$. As the fluid density is~0 outside the water region, we impose the boundary condition
\be\label{ass:P}
P\big|_{\Ga} = 0.
\ee

We parametrize the interface $\Ga(t) := \pa\Om(t)$ by a time-dependent $2\pi$-periodic curve $z(\al,t)$ as
$$
\Ga = \{z(\al,t) \in \Rbb^2 \  |  \ \alpha\in [-\pi, \pi]\}.
$$
The time evolution of $z(\al, t)$ is governed by the kinematic boundary condition. That is, on the interface we have
\be
(\pa_t z- v) \cdot (\pa_\al z)^\perp = 0, \label{eq:evol_zt}
\ee
where $(\pa_\al z)^\perp$ denotes the vector perpendicular to the tangent $\pa_\al z$. As the condition \eqref{eq:evol_zt} prescribes~$z_t$ only in the normal direction, there is some residual freedom to choose the tangential component. We do so by considering the equations in Lagrangian parametrization, that is, by replacing~\eqref{eq:evol_zt} by the stronger condition 
\be\label{eq:lagrangian}
z_t = v.
\ee

Throughout, we identify a point $z = (z_1, z_2)\in \Rbb^2$ with its complex representation $z  = z_1 + iz_2\in \Cbb$, and we denote its complex conjugate by $z^* := z_1 - i z_2$. Note we are using~$z$ both for points in~$\Cbb$ and for the curve $z(\al,t)$, but the meaning of will always be clear from the context. Recall that, in complex notation, we can write the real scalar product $z\cdot w$ of two vectors $z, w \in\Rbb^2$ as
\be\label{eq:scalar}
z\cdot w = \Re \left(z^* w\right) = \Re\left(z w^*\right).
\ee
Although we will mostly employ complex notation, it will sometimes be convenient to write certain quantities in real notation (e.g., as a real scalar product). We will use both simultaneously without further comment.

In the rest of the paper, we will assume that 
 the water region~$\Omega$ is a bounded domain, as one would assume in the case e.g.\ of a water drop. This will enable us to make some arguments slightly cleaner.
	It is known that when the water domain is bounded, the gravity term can be transformed away, so there is no loss of generality in taking $g=0$ in the case of bounded domains. 

\begin{remark}\label{R.periodicOmega}
	For completeness, in Section~\ref{ss.unbounded} we will describe the minor modifications that readily enable us to handle periodic water domains, where one assumes that the system is invariant under $2\pi$-periodic translations in the horizontal direction. That is, in the periodic setting that we will consider in Section~\ref{ss.unbounded}, $\Om(t)$ is the unbounded domain given by the region below a periodic open curve (i.e., $z(\al+2\pi,t)=z(\al,t) + 2\pi$) and  $v$ and~$P$ are assumed to be $2\pi$-periodic and to satisfy the usual  conditions at infinity inside the water region (i.e, $v(z)\to0$ and $P(z)\to+\infty$ as $z\to-i\infty$). In this case, it is well known that one cannot transform the gravity away, and that the  fact that $g>0$ is important for the Rayleigh--Taylor condition. To make the presentation of Section~\ref{ss.unbounded} as elementary as possible, we will keep the gravity constant~$g$ everywhere in the paper, but we stress that readers only interested in  bounded water domains can assume $g=0$ throughout.
\end{remark}

In Lagrangian parametrization, the length $|z_\al|$ of the tangent vector $z_\al$ depends on both space and time. We therefore use the parametrization-independent derivative 
$$
\pa_s := \frac{1}{|z_\al|}\pa_\al
$$
in order to keep track of various factors of $|z_\al|$. We sometimes use subscripts to denote partial derivatives. For instance, in the case of the unit tangent vector, we write $z_s(\al, t) \equiv \pa_s z(\al, t) = \frac{z_\al(\al, t)}{|z_\al(\al, t)|}$.  

Note that the $\pa_s$-derivative of the unit tangent $z_s$ is directly proportional to the unit normal,
$
z_s^\perp = i z_s,
$
and vice-versa. More precisely, defining the tangent angle $\theta$ as
$$
z_s =: e^{i\theta},
$$
we have
$$
\quad \pa_s z_{s} = \theta_s z_s^\perp, \quad \pa_s z_{s}^\perp = -\theta_s z_s,
$$
and $\theta_s = z_{ss}\cdot z_s^\perp$ is the curvature of the interface.

For concreteness, we shall assume that the interface has exactly one corner point, although the arguments remain unchanged for an arbitrary (finite) number of corners, which is preserved by the evolution. We therefore assume the interface $\Ga$ is of class $\Ccont^{1,\lambda}$ for some $\lambda\in(0,1)$ everywhere, except at the corner point $z_*\equiv z_*(t) \in \Ga(t)$. This is the only point of discontinuity of the unit tangent vector. More precisely, the interface at time~$t$ is parametrized by a Lipschitz continuous function $z(\cdot, t)$ on a one-dimensional torus $\Tbb:=\Rbb/2\pi\Zbb$, whose derivative exists and is H\"older continuous everywhere except at some fixed point $\al_*\in[-\pi, \pi]$ (mod $2\pi$) corresponding to $z_*$. In general, we will refer to $z_*(t)$ (i.e., the tip of the curvilinear corner) as the singular point or the corner point, while all other points on $\Ga$ will be considered regular.

As $\Ga(t)$ is a closed curve, the choice of $\al_*$ is arbitrary, in the sense that we can (and will, in the following) parametrize the curve so that the corner point corresponds to~$\al_*=\pi$:
\[
z_*(t) = z(-\pi, t) = z(\pi, t).
\]

We will assume that the curve $z(\al,t)$ satisfies the arc-chord condition 
\be\label{arc-chord}
\Fcal(z)(t) := \sup_{(\al, \al')} \Fcal(z)(\al, \al', t) < \infty,
\ee
where
\be\label{eq:defFcal}
\Fcal(z)(\al, \al', t) :=\begin{cases}
	\frac{|e^{i\al}- e^{i\al'}|}{|z(\al, t)-z(\al', t)|} \quad &\al\neq \al'\\
	\frac{1}{|z_\al(\al, t)|}                       \quad &\al = \al'.
\end{cases}  
\ee
It is well known that this assumption excludes the existence of self-intersections and cusps, but it is compatible with the assumption that the interface has corner points.

Next, we recall how to rewrite Equations \eqref{eqs:pre} in terms of the interface parametrization $z(\al, t)$ and the boundary velocity $V(\al, t) := v(z(\al, t), t)$, which we define to be the (non-tangential) limit of $v(z, t)$ on the boundary, that is,
$$
V(\al, t) := \lim_{z\rightarrow z(\al, t)} v(z, t).
$$ 

Let us now recall that the incompressibility and irrotationality conditions \eqref{eq:Inc-Irr} correspond to the Cauchy--Riemann equations for the complex conjugate of the velocity, so $v^*: = v_1 - i v_2$ is holomorphic in $\Om$. Therefore, assuming that $v\in \Ccont(\overline{\Om})$, it is standard that
\be\label{eq:Cauchy-V}
V^* = \Ccal(z)  V^*
\ee
on any regular point of the boundary. Here, $\Ccal(z)$ denotes the Cauchy singular integral operator associated to the curve, defined as the principle value
\begin{equation}\label{E.defCcalV*}
	\Ccal(z)  f(\al, t) := \frac{1}{\pi i} \, \PV  \int_{-\pi}^\pi \frac{f(\al', t)}{z(\al, t)-z(\al', t)} \, z_{\al}(\al', t)\,d\al'.
\end{equation}

Note that $z_\al(\al',t)\, d\al'$ is just the (parametrization-invariant) arclength measure, and that we obviously have 
\be\label{eq:Cauchy-const}
\Ccal(z)  1 = 1\,,\qquad \Ccal(z) z(\al,t)=z(\al,t)
\ee
everywhere except at the corner tip. Conversely, with $1 < p < \infty$, it is well known that any complex-valued $V^*\in L^p(\Ga)$  for which \eqref{eq:Cauchy-V} holds almost everywhere on $\Ga$ can be extended to an analytic function on $\Om$. When there is no ambiguity, we drop the reference to the curve in $\Ccal(z)$ and denote the Cauchy integral simply by $\Ccal$. 

We claim that the complex derivative $\frac{dv^*}{dz}$ of the holomorphic function~$v^*$ satisfies the same equation. To see this, we start by noting that  
\be\label{eq:Val}
\pa_\al V(\al, t)^* = \pa_\al v(z(\al, t), t)^* = (z_\al \cdot \nabla) v(z(\al, t), t)^* = z_\al(\al, t) \frac{dv^*}{dz}(z(\al, t), t),
\ee
where in the last equality we use \eqref{eq:scalar} and the fact that $\frac{dv^*}{dz} = (\nabla v_1)^* = -i(\nabla v_2)^*$. Here we abuse the notation and identify the gradient $\nabla v_i$ with its complex representation $ \pa_x v_i + i \pa_y v_i$.

Now let us introduce the notation
\be\label{eq:DsV}
D_\al V^* := \frac{1}{z_\al} \pa_\al V^*,
\ee  
which we will sometimes refer to as the conformal derivative of $V^*$. It is parametrization independent, so we have $D_\al \equiv D_s$ with the obvious notation $D_s = \frac{1}{z_s}\pa_s$. Therefore, assuming~$v\in \Ccont^1(\overline{\Om})$, Equation~\eqref{eq:Cauchy-V} holds for $D_\al V^*$ as well, that is,
\begin{equation}\label{E.DalV*}
	D_\al V^*=	\Ccal(z) D_\al V^*
\end{equation}
as claimed. Observe we cannot argue similarly for higher derivatives of~$V^*$, since in the solutions we consider in this paper $v\in C^{1,\lambda}(\overline\Om)$ but it is not in~$C^2(\overline\Om)$, in general.

A similar calculation for the time derivative of $V$ yields  
\be\label{eq:Vt}
\pa_t V(\al, t) = (z_t \cdot \nabla) v(z(\al, t), t) + \pa_t v(z(\al, t), t).
\ee
Note that, in the Lagrangian parametrization, the right-hand side corresponds to the left-hand side of \eqref{eq:dP}, after letting $z$ approach $z(\al, t)$ from the inside of the fluid domain. Since the pressure is zero on the interface by~\eqref{ass:P}, the tangential component of the pressure gradient must be identically  zero and therefore
$$
\pa_t V = - (\nabla P \cdot z_s^\perp) z_s^\perp - g e_2.
$$

The evolution equations for $(z, V)$ can finally be written as 
\be\label{eqs:evolGa}
\aligned
z_t &= V\\
V_t + ig &= i\si z_s
\endaligned
\ee
under the constraint \eqref{eq:Cauchy-V} for the complex conjugate of $V$.	Here and in what follows, we are using the notation
$$
\si:=-\pa_n P=-\na P \cdot z_s^\perp
$$ 
for minus the normal derivative of the pressure, which is the term that will appear in the Rayleigh--Taylor condition later on. 

We can remove the constraint \eqref{eq:Cauchy-V} by noting that the components of $V$ are not independent. More precisely, we shall  use the Dirichlet-to-Neumann operator $\Gcal(z)$ associated to~$z$  to express the second component~$V_2$ in terms of~$V_1$. Recall that the Dirichlet-to-Neumann map of the domain~$\Om$ is defined as 
\be\label{eq:DvN}
\Gcal(z) f:= \pa_n w,
\ee 
where $w$ is the harmonic extension of $f$ to $\Om$, that is, the unique solution of 
$$
\Delta w = 0, \qquad w|_\Ga = f.
$$
Details are given Section~\ref{s.inverse}.

As $v^*$ is holomorphic and the domain~$\Om$ is simply connected, it is standard that we can obtain the conjugate function $-v_2$ of the first component~$v_1$ by complex integration. Since $\Gcal(z)V_1$ gives the normal derivative of the harmonic extension of~$V_1$, the boundary trace $-V_2$ of $-v_2$ is then, up to a constant, equal to $ \int \Gcal(z) V_1 \, ds$, where the integral denotes a primitive of this function with respect to an arclength parameter on the interface curve. Therefore, the boundary trace $V^*:=V_1-iV_2$ of $v^* := v_1 - iv_2$ can be written as
\be\label{eq:DtN-V}
V = V_1 + i \int \Gcal(z) V_1 \, ds_z,
\ee
and by the construction its complex-conjugate satisfies \eqref{eq:Cauchy-V}. When there is no ambiguity, we will drop the dependence on~$z$ notationally and denote the Dirichlet-to-Neumann operator simply by~$\Gcal$.


Next, let us write $\si$ in terms of of $(z,V)$  using the Dirichlet-to-Neumann operator $\Gcal(z)$. To do this, we note that, by taking a time derivative of \eqref{eq:Inc-Irr}, one readily sees that $\pa_t v^*$ is holomorphic in $\Om$. By the same argument as in \eqref{eq:Val}, assuming that $v_t\in \Ccont(\overline\Om)$, we can then write its trace  in terms of $(z,V)$  as  
$$
\pa_t v(z(\al, t), t)^* = \pa_t V(\al, t)^* - z_t(\al, t)D_\al V(\al, t)^*,
$$
where $D_\al V^*$ was defined in \eqref{eq:DsV}. Indeed, set  
$$
\Psi(\al, t) := \pa_t v(z(\al, t), t).
$$
Using \eqref{eqs:evolGa}, we  conclude 
\be\label{eq:Psi}
\aligned
\Psi \cdot z_s  &= -gz_{2s} - \frac{1}{2}\pa_s |V|^2,\\
\Psi \cdot z_s^\perp &=  \si -gz_{1s} - V_s \cdot V^\perp   .                                                                   
\endaligned
\ee
One should note that e.g., $V\cdot V_s = \Re(V \pa_s V^*) = \Re(z_s z_t D_sV^*)$ by \eqref{eq:scalar} and \eqref{eq:DsV}. Thus, knowledge of~$z$ and~$V$ yields the tangential component of $\Psi^*$ on the boundary, which in turn  determines the normal component.

Let $\psi$ be a function, harmonic in $\Om$, whose boundary trace $\psi|_{\Ga} = \int \Psi \cdot z_s\, ds$ is a primitive of $\Psi \cdot z_s$. By the previous formulas, we can take
\be\label{eq:psi}
\psi|_{\Ga}  = -gz_{2} - \frac{1}{2}|V|^2\,.
\ee
Let us now define $\Psi^* := \frac{d\psi}{dz}= \pa_x \psi - i\pa_y\psi$, and record here that
\begin{equation}\label{E.HPsi}
	\Ccal(z)\Psi^*=\Psi^*\,.
\end{equation}
In terms of the Dirichlet-to-Neumann operator, we have
$\Gcal(z)\psi = \pa_n\psi =  \Psi\cdot z_s^\perp$, so
\be\label{eq:sigma}
\si= \Gcal\psi + gz_{1s} + V_s\cdot V^\perp.
\ee

Finally, note that $-P$ is subharmonic in $\Om$ by~\eqref{E.subh}. Therefore, we can use Hopf's lemma to conclude that $\si$, the normal derivative of $-P$, must be strictly positive at any regular point $z\in\Ga$. In fact, at any point that satisfies the inner ball condition (that is, at any point of~$\Ga$ but the corner tip) one has
$$
\si(z, t) > 0, \qquad z\in \Ga \setminus\{z_*\}.
$$ 
As $z$ approaches the corner tip $z_*$, we will show (cf.~Lemma~\ref{lem:si}) that, under certain conditions on the initial boundary velocity, one has
$$ 
\si(z, t) \approx |z - z_*|.
$$
We will  refer to this bound as the Rayleigh--Taylor condition (for domains with corners). Here and in what follows, when comparing two non-negative quantities, we often write
$$
f \lesssim g \,,\qquad f \approx g
$$
to denote that $f\leq Cg$ and $f/C\leq g\leq Cf$, respectively, where $C$ is an absolute constant (or a constant depending only on controlled quantities). Even for quantities that are not necessarily positive, we will also use the big-O notation
\[
f = O(g)
\]
for the pointwise bound $|f|\lesssim |g|$.

\subsection{Weighted Sobolev spaces}
\label{ss.weightedSobolev}

Let us now define the function spaces that we will use in this paper. More information about these spaces can be found in~\cite{KMR,MazNaPla}. For the purposes of this subsection, we can notationally drop the dependence on~$t$.

For $\ga\in \Rbb$, we define the weighted Lebesgue space $\Lcal_{2, \ga}(\Ga, m)$ as
$$
\Lcal_{2, \ga}(\Ga, m) := \{ \phi: \Ga \rightarrow \Rbb \, \, | \, \, m^\ga \phi \in L^2(\Ga) \},
$$ 
endowed with the norm 
$$
\|\phi \|_{2, \ga}^2 := \int_{\Ga} m(z)^{2\ga} |\phi(z)|^2 ds_z.
$$
Unless explicitly stated otherwise, for these spaces we always consider the power weight 
$$
m(z) := |z - z_*|,
$$
which vanishes to first order at the corner point $z_*\in \Ga$.
Here we are using that~$\Ga$ is bounded; otherwise we would take a well-behaved weight that coincides with~$m$ near~$z_*$ and is constant outside a certain ball. For complex functions one uses essentially the same definition.

Similarly, we can define weighted Lebesgue spaces on the torus $\Tbb$, or more generally on some bounded interval $I\subseteq \Rbb$. As discussed in the previous subsection, we are particularly interested in taking the interval $I := (-\pi, \pi)$, with the singular point $z_*=z (\pi)=z(-\pi)$ corresponding to the endpoints of the interval, $\pa I :=\{-\pi, \pi\}$. The weight function in this setting is essentially the distance to $\pa I$ and can be chosen to be smooth, e.g.\ by setting
\be\label{eq:weight}
m(\al) := (\pi + \al)(\pi - \al).
\ee
For ease of notation, we usually drop either the space or the weight notationally and write $\Lcal_{2, \ga}(m)$ or $\Lcal_{2, \ga}(I)$. Note that for arc-chord curves (i.e. those satisfying \eqref{arc-chord}), we have $m(\al) \approx m(z(\al))$.

If $\ga \in\left(-\frac{1}{2},\frac{1}{2}\right)$,  $m^\ga$ is a Muckenhaupt weight, i.e., it satisfies the Muckenhaupt $A_2$-condition. It is therefore standard that $\Lcal_{2,\ga}(m)\subseteq L^p$ for some $p > 1$. For later use, we introduce the notation
\begin{equation}\label{E.deflambda}
	\lambda_\ga : = \frac{1}{2} - \ga
\end{equation}
and note that $\lambda_\ga\in (0, 1)$ whenever $\ga$ satisfies the Muckenhaupt condition.

Given $\ga\in\Rbb$, we introduce two families of weighted Sobolev spaces:
\begin{align}
\phi \in H^k_{\ga}(\Ga, m) \quad &\Leftrightarrow \quad \pa_s^j\phi\in \Lcal_{2, \ga}(\Ga, m) , \quad 0\leq j\leq k,\label{Hkgamma}\\ 
\phi \in \Lcal^k_{2, \ga}(\Ga, m) \quad &\Leftrightarrow \quad \pa_s^j\phi\in \Lcal_{2, \ga + j - k}(\Ga, m), \quad 0 \leq j\leq k. \label{Lkgamma}
\end{align}
The corresponding weighted Sobolev spaces on the torus $\Tbb$ or on a bounded interval $I\subseteq \Rbb$ (in which case the arc-length derivative $\pa_s$ are replaced by the usual derivative) are defined analogously.

One clearly has
\be\label{def:relsobolev}
\Lcal^k_{2, \ga}(m)  \subseteq H^k_{\ga}(m).
\ee
The question of when the converse inclusion holds, and how it can fail, is considerably more interesting, and very relevant for our purposes. When $(\ga - k) + \frac12 > 0$, Hardy's inequality (in the form given as Theorem~\ref{thm:hardy} and Lemma~\ref{lem:hardy} in Appendix~\ref{A.A}) ensures that these spaces coincide. Otherwise, the inclusion is proper. 

In the case $\gamma>-\frac12$, by integrating and using Hardy inequalities, it is easy to see that the quotient space $H^k_{\ga}(m)/ \Lcal^k_{2, \ga}(m)$ is spanned by finitely many polynomials $\Pcal_j^\pm(s)$ (where $s$ is any arclength parameter, with $s=0$ being the corner point), of order at most $k - 1$, multiplied by fixed but arbitrary cutoff functions $\chi^\pm$ that are equal to~1 in a one-sided neighborhood of the corner point (say on the right side for $\chi^+$ and on the left for $\chi^-$) and zero on the other side. The cutoffs are piecewise smooth, and smooth but at the corner. This is because, if $\phi\in\Lcal^k_{2,\ga}(m)$, necessarily $\pa_s^i\phi(z_*) = 0$ for all $i\leq k$ such that $(\ga - i) + \frac12 \leq 0$, as it can be seen by an easy contradiction argument. This simply captures the fact that, in general, a function $\phi\in H^k_{\ga}(m)$ can have a jump discontinuity at $z_*$. Details can be found e.g.\ in~\cite{KMR}. 

Let $l$ be the largest non-negative integer strictly smaller than $k$ such that $(\ga + l - k) < -\frac12$, if it exists. Then, the polynomials $\Pcal_j^\pm$ are of degree at most $l$, so a function $\phi\in H^k_{\ga}(m)$ can be written  in a unique way as
\[
\phi(s)= \tilde \phi(s)+ \sum_{i=1}^l s^i\,[a_i^+ \chi^+(s) +a_i^- \chi^-(s)]
\]
with $\tilde \phi \in L^k_{\ga}(m)$. For all $\ga > -\frac12$ we then define 
\be\label{otherHkgamma}
\|\phi\|_{H^k_\ga(m)} := \|\phi\|_{\Lcal^k_{2,\ga}(m)} + \sum_{i = 0}^l (|a^+_i| + |a^-_i|).  
\ee
Of course, if one decomposes $\phi= \tilde \phi_1+ \sum_j (\t a_j^+\Pcal_j^+\chi^++ \t a_j^-\Pcal_j^-\chi^-)$ instead, with $\tilde \phi_1 \in \Lcal^k_{2,\ga}(m)$, then the norm defined as
\begin{equation*}
\|\tilde\phi_1\|_{\Lcal^k_{2,\ga}(m)} + \sum_j (|\t a^+_j| + |\t a^-_j|)
\end{equation*}
is obviously equivalent.
If the curve has several corner points, it suffices to consider the polynomials associated with each corner.

When we consider products of functions from these spaces, we will frequently use the following elementary pointwise bound without further mention.

\begin{lemma}\label{lem:sobolev}
Assume $m'= O(1)$. If $\phi \in \Lcal^1_{2,\ga + 1}(m)$, then
$$
  |\phi| \leq  Cm^{-1 + \lambda_\ga} 
$$
with a constant that only depends on $\|\phi\|_{1,\ga+1}$.
\end{lemma}

\begin{proof}
As
$$
(m^{2(\ga + 1/2)} \phi^2)' = (2\ga + 1)m' m^{2\ga} \phi^2 + m^{2\ga + 1} \phi \phi' \in L^1
$$
by H\"older's inequality, the claim follows by integration.
\end{proof}

\begin{remark}\label{rem:ga}
	Let us briefly discuss an important special case which will be frequently used in the rest of the paper. Assume $\ga \in \left(-\frac12, \frac12\right]$. Then, $\phi\in\Lcal_{2,\ga - 1}(m)$ implies we must have $\phi = o(m)$ by a contradiction argument. More precisely, there cannot exist $c > 0$ such that $\phi \geq c$ a.e.~on a neighborhood of a singular point. On the other hand, by Lemma \ref{lem:sobolev}, we know that $\phi\in \Lcal^1_{2,\ga}(m)$ implies $\phi = O(m^{\lambda_\ga})$ with $\lambda_\ga \geq 0$, and we conclude that $\phi$ must be zero at the singular point (in the trace sense). In addition, any $\phi \in H^1_\ga(m)$ is actually bounded, since by definition $\phi = a_0^\pm + o(m)$ near the singular point. Note that $\ga = \frac12$ is the biggest (`worst') value of parameter for which $\phi \in H^1_\ga(m)$ is guaranteed to remain bounded as we approach the corner tip. 
\end{remark}

Just as in the analysis of water waves with smooth interfaces, we will also need fractional Sobolev spaces. To define them, let 
\[
\Lambda^{1/2}\phi := (-\pa_{x}^2)^{1/4}\phi
\]
denote the half-derivative (or fractional Laplacian of order $1/2$)  on the real line, which one can equivalently define by  
$$
\Lambda^{1/2}\phi\,(\al) = c_\Lambda \int \frac{\phi(\al) - \phi(\al')}{|\al - \al'|^{3/2}}\,d\al'
$$
for an explicit  constant $c_\Lambda$.

A weighted estimate for the Riesz integral (see Theorem~\ref{thm:rieszintegral}) guarantees that $\Lambda^{1/2}:H^1_{\ga + \frac{1}{2}}(m)\rightarrow  \Lcal_{2,\ga}(m)$ is continuous provided $\ga \in \left(-\frac{1}{2}, 0\right)$, that is, provided that both $\ga$ and $\ga + \frac{1}{2}$ satisfy the Muckenhaupt $A_2$-condition (cf.~Lemma~\ref{lem:halfDer}). Sobolev spaces of half-integer order can then be defined the obvious way.

If this condition on $\ga$ is violated, we proceed as follows. We fix a parameter $\lambda\in\Rbb$ such that $\ga - \lambda \in \left(-\frac{1}{2}, 0\right)$ and define
$$
\Lcal^{k+1/2}_{2, \ga}(m) := \{\phi \in \Lcal^k_{2,\ga-1/2}(m) \, \, | \, \, \Lambda^{1/2}(m^\lambda\pa_s^k\phi) \in \Lcal_{2, \ga - \lambda}(m)\},
$$
endowed with the norm 
$$
\|\phi\|^2_{\Lcal^{k+ 1/2}_{2,\ga}} := \|\phi\|^2_{\Lcal^{k}_{2,\ga-\frac{1}{2}}} + \|\Lambda^{1/2}(m^\lambda\phi)\|^2_{2, \ga - \lambda}.
$$
We define $H^{k+1/2}_{\ga}(m)$ similarly. It is not difficult to see this definition is independent of the exact value of $\lambda$ by the commutator estimates of Subsection~\ref{SS.La}. 

A closely related operator is the (periodic) Hilbert transform,
\be\label{eq:hilbert}
H\phi(\al) := \frac{1}{2\pi} \PV \int_{-\pi}^{\pi}  \cot\left(\frac{\al - \al'}{2}\right)\phi(\al')\, d\al'.
\ee
With $I:=(-\pi,\pi)$, it is well-known that $H$ is bounded as an operator $\Lcal_{2,\ga}(I)\rightarrow \Lcal_{2,\ga}(I)$ as long as $m^\ga$ is Muckenhaupt. Some useful results on derivatives and commutators with $H$ and $\Lambda^{1/2}$ are given in Section \ref{ss.commutator} in Appendix~\ref{A.A}.

\subsection{The Cauchy singular integral operator for curves with corners}
\label{ss.BR}

In this section, we give a quick overview of the properties of the Cauchy integral on weighted Sobolev spaces that we will need to prove our energy estimates. For ease of notation, in this subsection we notationally omit the time dependence. 

Throughout this section, the implicit constants appearing in the estimates only depend on $\|z\|_{H^3_{\beta + 1}} + \Fcal(z)^{-1}$, where the arc-chord parameter $\Fcal(z)$ is given by~\eqref{eq:defFcal}. Here we are omitting the factor $z_\al$ in~\eqref{E.defCcalV*}, which will not cause any trouble later on.

\begin{lemma}\label{lem:BRbasic}
	Let $\ga \in \left(-\frac{1}{2}, \frac{1}{2}\right)$, $\beta\in \left(-\frac{1}{2}, \frac{1}{2}\right]$ and let $z\in H^3_{\beta + 1}(I)$  satisfy the arc-chord condition \eqref{arc-chord}. Then, the operator 
	$$
	 \Ccal(z) f(\al) := \frac{1}{\pi i} \PV \int_{I} \frac{z_\al(\al')}{z(\al') - z(\al)} f(\al')d\al'
	$$ is bounded on the space of all (complex-valued) $f\in \Lcal_{2,\ga}(I)$:
	$$
	\|\Ccal(z) f\|_{2,\ga} \, \lesssim \, \|f\|_{2,\ga}.
	$$ 
\end{lemma}

\begin{proof} 
	We start by noting that it is obviously enough to consider the estimate in a small neighborhood of the corner tip~$z_*$. Outside this neighborhood, the estimate is classical because $z$ is in the usual Sobolev space $H^3$ away from the corner and because the arc-chord condition holds. 
	
	For notational simplicity, we consider the shifted interval $I := (0, 2\pi)$, with $z_*=z(0)=z(2\pi)$. We aim to show that Cauchy integral belongs to $\Lcal_{2,\ga}$ in the vicinity of the left endpoint, i.e., that as a function of~$\al$, 
	$$
	 \frac{1}{\pi i} \PV \int_{I} \frac{z_\al(\al')}{z(\al') - z(\al)} f(\al')d\al'\in \Lcal_{2,\ga}(I_\de),
	$$
	with $I_\de := (0,\de)$. 
	
	So let us fix $\al\in I_\de$. We decompose 
	\be\label{eq:partition}
	I = I_{2\de} \cup I\setminus I_{2\de} \qquad I_{2\de} := I_l(\al) \cup I_c(\al) \cup I_r(\al),
	\ee
	with the left, central and right intervals
	\be\label{eq:Is}
	I_l(\al):=\left(0,\frac{\al}{2}\right], \quad I_c(\al):= \left(\frac{\al}{2}, \frac{3\al}{2}\right), \quad I_r(\al):= \left[\frac{3\al}{2}, 2\de\right).
	\ee
	Since
	$$
	|\al - \al'| \gtrsim \al, \quad  \al' \in I_l(\al); \qquad |\al' - \al| \gtrsim \al', \quad  \al'\in I_r(\al) ,
	$$   
one can estimate
	$$
	\aligned
	\frac{1}{|z(\al) - z(\al')|} = \frac{|\al - \al'|}{|z(\al) - z(\al')|} \frac{1}{|\al - \al'|} \, \lesssim \, \Fcal(z) \begin{cases}		
		\frac{1}{\al}, &\quad \al'\in I_l(\al),\\
		\frac{1}{\al'}, &\quad \al'\in I_r(\al) .
	\end{cases}
	\endaligned
	$$
Therefore, as $m(\al)\approx \al$ on $I_{2\de}$, we have
\begin{align*}
	\left|\int_{I_l(\al)} \, \frac{z_\al(\al')}{z(\al) - z(\al')} f(\al')d\al'\right| \, \lesssim \, \frac{1}{\al }\int_{0}^{\al}|f(\al')|  \, d\al',\\
	\left|\int_{I_r(\al)} \, \frac{z_\al(\al)}{z(\al) - z(\al')}  f(\al')d\al'\right| \, \lesssim \, \int_{\al}^{2\de}\frac{|f(\al')|}{\al'} \, d\al'
\end{align*}
	Both are bounded in $\Lcal_{2,\ga}(I_\de)$ by the Hardy inequalities~\eqref{eq:HardyA} and~\eqref{eq:HardyB}, respectively, in Lemma~\ref{lem:hardy}. Note that $\|z_\al\|_{L^\infty} $ is finite by Lemma \ref{lem:sobolev}, since by assumption $z_\al\in H^2_{\beta + 1}$ (cf.~Remark \ref{rem:ga}).
	
	On the other hand, when $\al'\in I_c(\al)$ we can write 
	\be\label{eq:basic}
	\aligned
	\frac{z_{\al}(\al')}{z(\al) - z(\al')} &=  \left(\frac{z_{\al}(\al')}{z(\al) - z(\al')} - \frac{1}{\al - \al'}\right) + \frac{1}{\al - \al'} \\
	&= -\frac{\pa_{\al'} F(z)(\al, \al')}{F(z)(\al, \al')} + \frac{1}{\al - \al'},
	\endaligned
	\ee
	where we have defined
	\be\label{def:F}
	F(z)(\al, \al') := \frac{z(\al) - z(\al')}{\al - \al'}.
	\ee
	It is not difficult to see that
	$$
	\pa_{\al'}  F(z)(\al,\al') = -\frac{1}{(\al-\al')^2}\int_{\al'}^\al z_{\al\al}(\al'')\frac{(\al - \al'')}{2}d\al''= O(\alpha^{-1 + \lambda_\beta}), \quad \al'\in I_c(\al),
	$$
	where we have used that $\al\approx\al'$ on $I_c(\al)$ and that $z_{\al\al} = O(\alpha^{-1 + \lambda_\beta})$ by Lemma \ref{lem:sobolev}. Since $1/F(z)$ is bounded by the arc-chord condition, the corresponding integral over $I_c(\al)$ is bounded in $\Lcal_{2,\ga}((0,\de))$ as required. To spell out the details, the error term in~\eqref{eq:basic} can be estimated by a suitable Hardy inequality cf.~\eqref{eq:HardyA}, while leading term $1/(\al-\al')$ reduces to the Hilbert transform, which is bounded on Lebesgue spaces  because the weight with power~$\gamma$ satisfies the Muckenhaupt $A_2$-condition.
	
	Finally, assume $\al' \in I\setminus I_{2\de}$. Then $|\al - \al'| > \de$, and the kernel of the Cauchy integral is bounded in this region by the arc-chord condition, so the estimate is straightforward.\end{proof}

The Cauchy integral is naturally bounded on $\Lcal_{2,\ga}(m)$, with $\ga$ in the Muckenhaupt range $(-\frac12,\frac12)$. However, to handle derivatives in the energy estimates, we will need to consider the action of~$\Ccal(z)$ on functions $f\in\Lcal_{2, \ga + j}(m)$, with $j\geq 1$. Note these functions are are generally not locally integrable at the singular point. 
Conversely, if $f\in\Lcal_{2, \ga + j}(m)$ with $j < 0$, the Cauchy integral applied to $f$ generally results in a function $\Ccal f\in \Lcal_{2,\ga}(m)$, not in the strictly smaller subspace $\Lcal_{2, \ga + j}(m)\subset \Lcal_{2,\ga}(m)$. 

To keep track of these subtleties, which have a major impact in the energy estimates, for $j\in\Zbb$ we define following ``corrected'' Cauchy integral operators:
\be\label{def:BR+k}
\aligned
\Ccal_{(j)}f (\alpha):&=\frac{1}{(z(\al) - z_*)^j}\frac{1}{\pi i}\, \PV \int_I f(\al')(z(\al') - z_*)^j  \, \frac{z_\al(\al')}{z(\al') - z(\al)} d\al'\\
& = \frac{1}{(z - z_*)^j}\Ccal [(z(\cdot) - z_*)^jf](\alpha).
\endaligned
\ee
Obviously $\Ccal_{(0)} \equiv \Ccal$. With these corrections, an immediate consequence of Lemma~\ref{lem:BRbasic} is the following:

\begin{corollary}\label{lem:BRcorrections}
	Under the assumptions of Lemma \ref{lem:BRbasic}, the operator
	$$
	\Ccal_{(j)}: \Lcal_{2,\ga + j}(m) \rightarrow \Lcal_{2,\ga + j}(m),
	$$
	is bounded, for any $j\in \Zbb$.
\end{corollary}

We are now interested in relating the Cauchy integral to this corrected operator, in order to apply the above corollary. With $f\in \Lcal_{2,\ga - k}(m)$ and $k\geq1$, it is straightforward to write 
\be\label{def:BR-k}
\Ccal f =  \sum_{i = 0}^{k-1} c_i(f)(z_* - z(\al))^{i} + \Ccal_{(-k)}f, 
\ee
where we have defined the constants
\be\label{def:bi}
c_j(f):= \frac{1}{\pi i}\int_I f(\al')\, \frac{z_\al(\al')}{(z(\al') - z_*)^{j+1}} \,d\al'.
\ee
We thus conclude that the Cauchy integral operator actually maps $\Lcal_{2,\ga - k}(m)$ to the sum of a function in $\Lcal_{2,\ga - k}(m)$ and a complex polynomial of order at most $k - 1$.   


Our next goal if to fix the function~$f$ and consider the derivatives of $ \Ccal f(\al)$ with respect to~$\al$. A first observation is that the continuity of $f$ at the corner tip has a major effect on the differentiability of $\Ccal f$. In fact, if $f\in H^1_{\ga}(m)\cap \Ccont^1(\overline{I})$ with $|\ga| < \frac{1}{2}$, we have 
	\be\label{eq:BR-ipp}
	D_\al (\Ccal f) = \Ccal D_\al f + \frac{f(\pi) - f(-\pi)}{z-z_*},
	\ee
	where we recall that $D_\al f:=z_\al^{-1}\pa_\al f$.
Therefore, we cannot expect $\Ccal$ to be bounded as an operator $H^1_\ga(m)\rightarrow H^1_\ga(m)$. This is not surprising; as we expect that $\Ccal f = f$ when $f$ is the boundary value of a holomorphic function and the imaginary part of the complex logarithm $\log (z - z_*)$ has a jump discontinuity at $z_*$. Note, however that $\Ccal: \Lcal^1_{2,\ga}(m)\rightarrow H^1_\ga(m)$ is easily shown to be bounded.

If $\frac{1}{2} < \ga$, then $D_\al f$ is not locally integrable near the singular point. Using ideas similar to~\eqref{def:BR+k}, one can nevertheless define ``corrected'' operators that will often appear when we integrate by parts to establish the energy estimates:

\begin{lemma}\label{lem:BRders}
	Let $\ga \in \left(-\frac{1}{2}, \frac{1}{2}\right)$, $\beta\in \left(-\frac{1}{2}, \frac{1}{2}\right]$ and assume that the curve $z\in H^{k + 1}_{\beta + k - 1}(m)$, with $k\geq 1$, satisfies the arc-chord condition \eqref{arc-chord}. Then, 	$$
	\Ccal  : \Lcal^{k + 1}_{2,\ga + (k + 1)}(m) \rightarrow \Lcal^{k + 1}_{2,\ga + (k + 1)}(m)
	$$
	is bounded and we have
	$$
	D_\al^j (\Ccal f) = \Ccal_{(j)} D_\al^j f, \quad 1 \leq j \leq k + 1. 
	$$ 
\end{lemma}

\begin{proof}
	Since $\pa_\al^j f\in\Lcal_{2,\ga + j}(m)$ for $ 0 \leq j\leq k + 1 $, and $z \in H^{k + 1}_{\beta + k - 1}(m)$, we have	
	$$
	 D_\al^j f\in\Lcal_{2,\ga + j}(m), \quad 0 \leq j\leq k + 1.
	$$
Writing
	\be\label{eq:correction-ipp}
	\frac{1}{z_\al}\pa_\al  \left(\frac{z_{\al'}}{z(\al') - z(\al)}\right) = -\,\pa_{\al'} \left(\frac{1}{z(\al') - z(\al)} - \frac{1}{z_* - z(\al)}\right), 
	\ee
	we see using integration by parts and Corollary~\ref{lem:BRcorrections} that
	$$
	D_\al (\Ccal f)  = \Ccal_{(1)} D_\al f \in \Lcal_{2,\ga + 1}(m)\,,
	$$
	where we recall that $\Ccal_{(1)}$ was defined in \eqref{def:BR+k}. For higher order derivatives, an analogous argument shows	
	$$
	D_\al^j (\Ccal  f)  = \Ccal_{(j)} D^j_\al f, \qquad 1 \leq j \leq k + 1,
	$$
	so $\Ccal f \in \Lcal^j_{2,\ga + j}(m)$ as claimed.
\end{proof}

To finish this section, in the following lemma we aim to make precise the idea that the Hilbert transform~$H$ effectively captures the singular behavior of the Cauchy singular integral. This is standard, on unweighted Sobolev spaces, when the curve is smooth.

\begin{lemma}\label{lem:BRcancel}
	Under the assumptions of Lemma \ref{lem:BRders}, the operator 
	\be\label{eq:corner-der-on-ker}
	(i\Ccal + H): \, \Lcal_{2,  \ga}(m)\rightarrow H^{k - 1}_{\ga + k - 1}(m)
	\ee
	is bounded. 
\end{lemma} 

\begin{proof} 
	The result is well known but for the weights, so it suffices to prove the estimate for the difference
	$$
	\Rcal f(\al) := \frac{1}{\pi}\, \PV \int_{I} \bigg(\frac{z_{\al}(\al')}{z(\al') - z(\al)} - \frac{1}{\al' - \al}\bigg)f(\al')d\al' \in H^{k}_{\ga + k}(m)
	$$
	near the corner tip. The reasoning is very similar to that in the proof of Lemma \ref{lem:BRbasic}, so for convenience we shift the interval and consider $I:=(0,2\pi)$ with $z_*=z(0)=z(2\pi)$. 
	
	Using the notation of the proof of Lemma~\ref{lem:BRbasic} without further mention, let $\al\in I_\de$ be fixed for some small $\de>0$. When $\al'\in I\setminus I_{2\de}$, the kernel of $\Rcal f$ is essentially non-singular, so we we can take as many $\pa_\al$-derivatives as the regularity of the parametrization allows. Therefore, one can use the arc-chord condition and the corresponding lower bounds on $|\al - \al'|$ to obtain the required estimates, as in the proof of Lemma~\ref{lem:BRbasic}. 
	
	When $\al' \in I_{2\de}$, we can write
	\be\label{eq:hilbert-to-graph}
	\pa_\al\left(\frac{z_{\al}(\al')}{z(\al') - z(\al)} - \frac{1}{\al' - \al}\right) =  \frac{\pa_{\al'} F(z)(\al, \al')}{F(z)(\al, \al')},
	\ee
	where $F$ was defined in \eqref{eq:basic}-\eqref{def:F}. Given $f\in\Lcal_{2,\ga}(I_{2\de})$, we need to show that
	$$
	\al \mapsto \int_{0}^{2\de}  \frac{\pa_{\al'} F(z)(\al, \al')}{F(z)(\al, \al')}\, f(\al')d\al'  \in H^{k-1}_{\ga + (k - 1)}(I_\de).
	$$ 
	Thanks to the arc-chord condition and the assumption that we control $k + 1$ weighted derivatives of $z$, a minor variation of the proof of Lemma~\ref{lem:derFphi} shows that the claim holds for the $(k-1)$-th derivative of this integral. The lemma then follows. 
	\end{proof}


\section{The a priori energy estimate}
\label{s.apriori}

Our goal in this section is to prove the a priori energy estimate that lies at the heart of our local wellposedness result, Theorem~\ref{T.main}.

	Let us start by describing the setting in which this estimate holds. Throughout this section, both the parametrization of the interface and the boundary velocity are sufficiently smooth functions on $I\times [0, T]$, bounded in certain weighted Sobolev spaces which we now specify.

	We consider the interval $I:=(-\pi,\pi)$, fix some $\beta\in\left(-\frac{1}{2},\frac{1}{2}\right]$ and some integer $k\geq 2$. To describe the interface, we assume  
	\be\label{regularity-z-corner}	
	\theta(\cdot, t)\in H^{k+1}_{\beta + k}(I, m), \qquad \log|z_\al(\cdot, t)|\in H^{k + \frac{1}{2}}_{\beta + (k-\frac{1}{2})}(I, m).
	\ee
	The interface curve is then written as
	\be\label{eq:z}
	z(\al, t) = z_*(t) + \int_{-\pi}^\al |z_\al(\al', t)|e^{i\theta(\al', t)}d\al',
	\ee
	where the corner point $z_*(t):= z(\pi, t)=z(-\pi,t) $ is the point of discontinuity of the tangent vector. 
	
	To keep track of the angle of the corner, we denote the lateral limits of the tangent angle at the corner by
		\be\label{eq:angle}
	 \nu_+(t):=\lim_{\al \rightarrow -\pi^+}\theta(\al, t) , \qquad \nu_-(t):= \lim_{\al \rightarrow \pi^-}\theta(\al, t) ,
	\ee
	where the branch cut is chosen along the negative real axis. In particular, we have $\nu_\pm \in [-\pi, \pi]$ and the (interior) angle of the corner defined by arcs of the curve~$z(\al,t)$, which we denote by $2\nu(t)$ for convenience, is	
	$$
	2\nu(t) = \pi + \nu_+(t) - \nu_-(t).
	$$
	Throughout, we assume that
	$$
	2\nu(t) \in \left(0, \frac{\pi}{2}\right).
	$$
	
	Let us take a smooth cutoff $\chi$, which we can assume to be time-independent, and which is $1$ near the left endpoint and  $0$ near the right endpoint. Setting $\chi_+ := \chi$ and $\chi_- := 1 - \chi$, we can conveniently write 
	\be\label{eq:theta}
	\theta(\al, t) = \Theta(\al, t) +  \nu_+(t)\chi_+(\al) + \nu_-(t)\chi_-(\al).
	\ee
	where 
	\be\label{eq:Theta}
	\Theta\in \Lcal^{k + 1}_{2,\beta + k}(m)\,.
	\ee	
 	Recall that we have defined the $H^{k + 1}_{\beta + k}(m)$-norm of $\theta$ as (modulo cut-offs)
	\be\label{eq:norm-theta}
	\|\theta\|^2_{H^{k + 1}_{\beta + k}(m)} := |\nu_+|^2 + |\nu_-|^2 +  \|\Theta\|^2_{\Lcal^{k + 1}_{2, \beta + k}(m)}.
	\ee
		
	As we will see in a moment, the time derivatives of both $\theta$ and $\log |z_\al|$ are generally discontinuous at the corner tip, so in principle $\log |z_\al|$ will also have a jump discontinuity at $\al_*$. However, we will not explicitly need this fact to prove the energy estimates, so we will not discuss now the corresponding expansion for $\log|z_\al|$, which anyway follows along the same lines. 
	
	As for the boundary velocity $V$, we assume its complex conjugate $V^*$ has the form  
	\be\label{eq:v-rep}
	V^*(\al, t) = b_0(t) + b_1(t)(z(\al, t) - z_*(t)) + U^*(\al, t), \quad U^*\in \Lcal_{2, \beta + k-1}^{k+1}(m),
	\ee
	where $U^*$ satisfies
	\be\label{eq:CauchyU}
	U^* = \Ccal(z) U^*.
	\ee	
	As a consequence of this, it follows from~\eqref{eq:Cauchy-const} that $V^* = \Ccal(z)V^*$ everywhere except at the corner tip. Thus $V^*$ is the boundary value of a holomorphic function, as required in Section~\ref{s.preliminaries}.  
	
	Note that the complex numbers $b_0,b_1$ are
	$$
	b_0(t) := V^*(\pm \pi, t), \qquad b_1(t) := D_s V^*(\pm \pi, t).
	$$
	Hence both $V^*$ and $D_s V^*$ are continuous when crossing the singular point, and the corresponding conjugate velocity in the water region satisfies $v^*, \frac{dv^*}{dz}\in\Ccal(\overline{\Om})$. However, its tangential derivative~$\pa_s V^*$ is discontinuous at the singular point. In fact, for any $\beta\in(-\frac12,\frac12]$,
	\begin{equation}\label{E.Rk31}
		D_s^i\left(\Ccal(z) V^*\right) = \Ccal(z)(D_s^i V^*),\qquad i=1,2.
	\end{equation}
	To see this, it suffices to take up to two $D_s$-derivatives of the basic relation $V^* = \Ccal(z)V^*$ and note that the boundary terms could appear after integration vanish by Remark~\ref{rem:ga}.
	
	Just as in the case of~$\theta$, for $k\geq 1$ and $V$ defined in terms of $b_0, b_1$ and $U$ via \eqref{eq:v-rep}, we have
	\be\label{eq:norm-V}
	\|V\|^2_{H^{k + 1}_{\beta + (k - 1)}} = |b_0|^2 + |b_1|^2 + \|U\|^2_{\Lcal^{k + 1}_{2,\beta + k -1}}.
	\ee



\subsection{Estimates for the time derivatives in the Lagrangian parametrization}
\label{ss.PE}

Regarding the curve $z$ as a function of $z_*, \nu_\pm, \Theta, \log|z_\al|$ via~\eqref{eq:z}--\eqref{eq:theta} and the boundary trace of the velocity~$V$ as a function of $b_0, b_1, U$ via \eqref{eq:v-rep}, our goal now is to estimate the time derivatives for these quantities.
That is, under the assumption that~\eqref{eq:CauchyU} holds, we aim to compute the time derivatives of $z_*, \nu_\pm, b_0$, $b_1$ and to show that the assumptions~\eqref{eq:theta} and \eqref{eq:v-rep} are consistent. One should note that the latter is closely related to the behavior of $\si$ near the corner tip, and that the ODE system~\eqref{E.ODEintro} that we will use to prove the formation of singularities later on will make its appearance here for the first time, as we shall see in the key Lemmas~\ref{lem:param} and~\ref{lem:si}.

Throughout this section, we assume that
\be\label{eq:basic-energy}
\|\theta\|^2_{H^{k + 1}_{\beta + k}} +  \|\log |z_\al|\|^2_{H^{k +1/2}_{\beta + (k - 1/2)}} + \|V\|^2_{H^{k + 3/2}_{\beta + (k - 1/2)}} + \Fcal(z) < \infty.
\ee
It is worth mentioning, however, that when we define the  energy functional in Section \ref{ss.apriori-energy}, it will be convenient to replace the highest-order derivative of $V$ by a suitable derivative of a certain explicit function of $z$ and $V$, namely
\be\label{regularity-V}
\varphi_s := V_s \cdot z_s.
\ee
Note that here, unlike elsewhere in the paper, the subscript in $\varphi_s$ does not denote an arc-length derivative; it is simply a notationally convenient way to keep track of the fact that~$\varphi_s$ is on the same level of regularity as the partial derivative~$V_s$.

In Lemma \ref{lem:param} below, we shall show that the assumption~\eqref{eq:basic-energy} is actually equivalent to
\be\label{eq:basic-energy-phi}
\|\theta\|^2_{H^{k + 1}_{\beta + k}} +  \|\log |z_\al|\|^2_{H^{k + 1/2}_{\beta + (k - 1/2)}} + \|V\|^2_{H^{k}_{\beta + (k - 2)}} + \|\varphi_s\|^2_{H^{k + 1/2}_{\beta + (k - 1/2)}} + \Fcal(z) < \infty.
\ee
Recall that the norms of $V$ and $\theta$ have been given in \eqref{eq:norm-V} and \eqref{eq:norm-theta} respectively, and $\Fcal$ was defined in \eqref{arc-chord}. 

We next present several lemmas with estimates for the evolution of the various quantities under the water waves equations in Lagrangian parametrization, \eqref{eqs:evolGa}. In the proofs we will need a number of results from Section~\ref{s.inverse} and Appendix~\ref{A.A}. In all these lemmas, the various norms are controlled by at most the exponential of some polynomial of the energy~\eqref{eq:basic-energy-phi}; we will not repeat this fact in the statements.

\begin{lemma}\label{lem:param}
	Assume $(z, V)$ satisfy \eqref{eq:basic-energy-phi} for some $\beta\in\left(\frac12, \frac12\right]$ and $k\geq 2$. Then, the norm of the highest-order derivative of the boundary velocity, $\|\pa_s^{k + 1} V \|^2_{\Lcal^{1/2}_{\beta + k -1/2}(m)}$ can be controlled by~\eqref{eq:basic-energy-phi}. In particular, the energy functionals \eqref{eq:basic-energy} and \eqref{eq:basic-energy-phi} are equivalent. Moreover, 
	\begin{gather}
		(\log|z_\al|)_t = \varphi_s \in H^{k + 1/2}_{\beta + k - 1/2}(m),\qquad \Theta_t\in \Lcal^{k + 1/2}_{2, \beta + k - 1/2}(m),\notag\\[1mm]
		\frac{dz_*}{dt} = b_0^*, \qquad \frac{d\nu_+}{dt} = \Re\left(ib_1e^{2i\nu_+}\right), \qquad \frac{d\nu_-}{dt} = 			\Re\left(ib_1 e^{2i\nu_-}\right). \label{eq:nudt}
	\end{gather}
	Furthermore, when $\beta \in \left(\frac12, \frac12\right)$, we can write 
		\be\label{eq:th_ts}
		\theta_{ts} = H(\varphi_{ss})  + \Rcal(z,V), \quad  \Rcal\in \Lcal^{k}_{2, \beta + k}(m),
		\ee
	which, in case $\beta = \frac12$, should be modified to 
		\be\label{eq:th_ts12}
		\theta_{ts} = m^{-\lambda'}H(m^{\lambda'}\varphi_{ss})  + \Rcal(z,V), \quad  \Rcal\in \Lcal^{k}_{2,\beta + k}(m),
		\ee
	where $\lambda' \in (0,1)$ is arbitrary.
\end{lemma}

\begin{proof}
For simplicity let us take $k = 2$; the general case is completely analogous. By the definition of the parametrization-independent derivative $\pa_s$ and the tangent angle~$\theta$, we have $z_s = {z_\al}/{|z_\al|} = e^{i\theta}$, so
$$
 z_{st} = z_{ts} - \frac{|z_\al|_t}{|z_\al|} z_s.
$$
In Lagrangian parametrization, this implies
\be\label{eq:th_ts-aux}
 (\log|z_\al|)_t - i\theta_t = z_s (z_{ts})^* = z_s^2 D_s z_{t}^*=  z_s^2 D_s V^* \in H^1_{\beta}(m),
\ee
where we have used Lemma \ref{lem:sobolev}. Note that, by definition, 
\be\label{eq:eq-phis}
(\log|z_\al|)_t = \Re(z_s^2 D_sV^*) = V_s\cdot z_s = \varphi_s.   
\ee

Using \eqref{eq:v-rep}, we obtain
$$
 (\log|z_\al|)_t - i\theta_t = z_s^2 b_1 + z_s^2 D_s U^* = z_s^2 b_1 +  \Lcal^1_{2,\beta}(m).
$$ 
Here and in what follows, we simply write ``$+  \Lcal^1_{2,\beta}(m)$'', and similarly with other weighted spaces, when the identity holds up to the addition of a function in this space.
In the second equality, we have used that $U^*\in\Lcal^2_{2,\beta}(m)$ (which implies that $D_s U^*\in \Lcal^1_{2,\beta}(m)$) and $\theta_t\in H^1_\beta(m)$. By the properties of these spaces, both $\theta_t$ and $(\log|z_\al|)_t$ are thus bounded and well-behaved everywhere except possibly at the corner point. To compute the discontinuity, we simply approach the corner tip (or, equivalently, the endpoints of the interval $I=(-\pi,\pi)$) from the right and from the left, readily obtaining the equations for $\frac{d\nu_\pm}{dt}$ as stated in \eqref{eq:nudt}. Using \eqref{eq:theta}, one similarly gets the evolution equation for $\Theta_t$ and the estimate $\Theta_t\in\Lcal^1_{2,\beta}(m)$. Note that so far we have only used two derivatives of~$V$.

In order to prove \eqref{eq:th_ts}, we apply $D_s$ to both sides of \eqref{eq:th_ts-aux}, arriving at the equation
	\be\label{eq:complex-phi_s-th_t-der}
	\varphi_{ss} - i\theta_{ts} = z_s^3 D_s^2V^* + 2i\theta_s\left(\varphi_{s}-i\theta_t\right). 
	\ee
	Considering the real and imaginary part of this equation results in
	\be\label{eq:ztss}
	\theta_{ts} = \Re\left( i z_s^3D^2_sV^*\right) -2\theta_s\varphi_s,\qquad 	\varphi_{ss} = \Re\left(z_s^3D^2_sV^*\right) + 2\theta_s\theta_t.
	\ee
	
	Let us consider the case $\beta \in \left(\frac12, \frac12\right)$ first. Using~\eqref{E.Rk31}, we can rewrite the equation for~$\theta_{ts}$ in terms of the well-behaved operator $i\Ccal-H$ and commutators as
	$$
	\theta_{ts} = H\varphi_{ss} + \underbrace{\left[ \Re\left(z_s^3 (i\Ccal - H) D^2_sV^*\right) + \Re\left([z_s^3, H]D^2_sV^*\right) - 2\theta_s \varphi_s - 2H(\theta_s \theta_t)\right]}_{=:\Rcal(z,V)}.
	$$
	We thus conclude that all the terms belong to $\Lcal^1_{2,\beta + 1}(m)$: the first one by Lemma~\ref{lem:BRcancel}, the commutator term by Lemma~\ref{lem:derFphi}, and last two by Lemma~\ref{lem:trHilbert}. Here we have used that $D_s^2 V^*\in\Lcal_{2,\beta}(m)$ and $\theta_s \theta_t\in\Lcal^1_{2,\beta + 1}(m)$ under our regularity assumptions thanks to Lemma~\ref{lem:sobolev}. In particular,  $	\Rcal \in H^1_{\beta + 1}(m)$, as claimed.

	Next, we note that $\varphi_{ss}$ (or, equivalently, $\Re(z_s^3 D_s^2V^*$) is in $\Lcal^1_{2,\beta + 1}(m)$. By the preceding formulas for the derivatives of~$\theta_t$ and by Lemma~\ref{lem:trHilbert}, we therefore conclude that 
	$$
	\theta_{t} \in H^2_{\beta + 1}(m).
	$$
	In turn, this implies $z_t \in H^3_{\beta + 1}(m)$, which is equivalent to $V\in H^3_{\beta + 1}(m)$ as the equations are written in the Lagrangian parametrization. Thus, the norm $\|\pa_s^3 V\|^2_{\beta + 1}$ is controlled by~\eqref{eq:basic-energy-phi} and we have $\Theta_t \in \Lcal^2_{2,\beta + 1}(m)$.
	
	With the control on an additional derivative of $V$ that we have just established, we are ready to prove that we can take another derivative on the first summand in~$\Rcal$ (note that  all the other terms in~$\Rcal$ are now known to be in $\Lcal^2_{2,\beta + 2}(m)$). More precisely, it remains to show that  
	$$
	D^2_\al (i\Ccal - H)D_\al^2V^* \in \Lcal_{2, \beta + 2}(m),
	$$
	where we used that $D_s \equiv D_\al$. To simplify the notation, we temporarily assume $z_* = 0$ without loss of generality, so we can write $z$ instead of $z - z_*$. First let us compute 
	$$
	\aligned
	\pa_\al H(D^2_\al V^*) &= -\frac{z_\al}{z}HD^2_\al V^*  + \frac{1}{z}\pa_\al \left(zHD^2_\al V^*)\right)\\
	&=-\frac{z_\al}{z}HD^2_\al V^* + \frac{1}{z} H(\pa_\al (z D_\al^2 V^*)) + \frac{1}{z}\pa_\al [z, H]D^2_\al V^*\\
	&= \frac{1}{z}\left(z_\al H(z  D_\al^3 V^*) + [H, z_\al] (D_\al^2 V^* + zD_\al^3V^*) + \frac{1}{z}\pa_\al [z, H]D^2_\al V^*\right),
	\endaligned
	$$
	which implies 
	$$
	D_\al H(D^2_\al V^*) = \frac{1}{z}H(z  D_\al^3 V^*) + \Lcal^1_{2,\beta + 2}(m)
	$$
	by Lemmas \ref{lem:trHilbert} and \ref{lem:derFphi}. Using Equation~\eqref{eq:BR-ipp}, we now conclude that
	$$
	\aligned
	D_\al(i\Ccal - H) D^2_\al V^* &= \left(i\Ccal_{(1)}D_\al^3V^* - \frac{1}{z}H\left(zD_\al^3V^*\right)\right) + \Lcal^1_{2,\beta + 2}(m)\\
	&= \frac{1}{z}(i\Ccal - H) (zD_\al^3V^*)  + \Lcal^1_{2,\beta + 2}(m),
	\endaligned
	$$
	where we have used the definition of $\Ccal_{(1)}$, Equation~\eqref{def:BR+k}. Since $zD_\al^3V^*\in \Lcal_{2,\beta}(m)$ the remaining term belongs to $\Lcal^1_{2,\beta + 2}(m)$ by Lemmas \ref{lem:BRbasic} and \ref{lem:BRcancel}. 
	
	Note that in the argument we have not used any control over the highest-order half-derivative of $\varphi_s$. Therefore, in view of~\eqref{eq:th_ts}, the estimate for the additional highest-order half-derivative of~$V$ is straightforward.
	
	Now, consider the case $\beta = \frac12$. In this case, we cannot use Lemma~\ref{lem:BRcancel} (nor Lemmas~\ref{lem:trHilbert} and~\ref{lem:derFphi} for that matter) on $D_s^2 V^*$ directly, since the corresponding weight does not satisfy the Muckenhaupt $A_2$-condition. We shall therefore choose some $\lambda \in (0, 1)$ and we prove  \eqref{eq:th_ts12}, a modified version of \eqref{eq:th_ts}, as follows. 
	
	Let us write 
	$$
	\theta_{ts} = \Re\left( i z_s^3 z^{-\lambda'} (z^{\lambda'} D^2_sV^*)\right) -2\theta_s\varphi_s,
	$$
	where, for notational simplicity, we still assume $z_* = 0$ and write $z$ instead of $z - z_*$. The function $z^\lambda D_s^2 V^*$ is still the trace of a holomorphic function, but it now belongs to a weighted Sobolev space with a weight that does satisfy the Muckenhaupt $A_2$-condition. 
	
	Therefore, we can now proceed essentially as in the case $\beta \in\left(\frac12, \frac12\right)$. Indeed, we have
	$$
	\theta_{ts} = \Re\left(\frac{1}{z^{\lambda'} }\left(H(z_s^3 z^{\lambda'}  D_s^2 V^*) + z_s^3 (i\Ccal - H) (z^{\lambda'} D_s^2 V^*)  - [ H, z_s^3](z^{\lambda'} D_s^2 V^*) \right)\right) - 2\theta_s \varphi_s, 
	$$ 
	and we can further rewrite the first term as
	$$
	\Re\left(\frac{1}{z^{\lambda'}}H(z_s^3 z^{\lambda'} D_s^2 V^*)\right) = \frac{1}{m^{\lambda'}} H(m^{\lambda'} (\varphi_{ss} - 2 \theta_s \theta_t))  +  \frac{1}{m^{\lambda'}} \Re\left(\left[\frac{m^{\lambda'}}{z^{\lambda'}}, H\right]\left( z^3 z^{\lambda'} D_s^2 V^*\right)\right).
	$$
	Putting everything together, we arrive again at an equation of the form~\eqref{eq:th_ts12}, the only difference being that now one has
	\be\label{eq:R12}
	\aligned
	\Rcal(z, V) &:= \Re\left(\frac{z_s^3}{z^{\lambda'}} (i\Ccal - H) (z^{\lambda'} D_s^2 V^*) - \frac{1}{z^{\lambda'}}[ H, z_s^3](z^{\lambda'} D_s^2 V^*)\right) \\
	&- \Re\left(\frac{1}{m^{\lambda'}}\left[\frac{m^{\lambda'}}{z^{\lambda'}}, H\right]\left( z^3 z^{\lambda'} D_s^2 V^*\right)\right)- 2\theta_s \varphi_s - \frac{2}{m^{\lambda'}} H(m^{\lambda'} \theta_s \theta_t) 
	\endaligned
	\ee

	The higher-order regularity of $V$, as well as  the regularity of $\Rcal(z, V)$, now follows exactly as in the case $\beta \in\left(\frac12, \frac12\right)$. One simply has to recall that $z \approx m$ by the arc-chord condition and the regularity assumptions on the interface, and that we have the estimate $\pa_\al^j(m/z) = O(m^{-j})$ for $0\leq j \leq k$.
\end{proof}

 
 \begin{lemma}\label{lem:si} Let $\left|(\beta - 3) + \frac{1}{2} \right| < \frac{\pi}{2\nu}$ and let $(z, V)$ satisfy \eqref{eq:basic-energy-phi} for some $\beta\in\left(\frac12, \frac12\right]$ and $k\geq 2$. Then, the normal derivative of the pressure $\si$, cf.~\eqref{eq:sigma}, is in $ H^{k+1}_{\beta + (k-1)}(m)$. Moreover, 
 \begin{gather}
  	\label{eq:si-rep}
 	\si = |b_1|^2\, |z - z_*|\tan 2\nu  +  \Lcal^{k+1}_{2,\beta + (k-1)}(m),\qquad U^*_t \in \Lcal^{k + \frac{1}{2}}_{2, \beta + (k-\frac{3}{2})}(m),\\
 	\label{eq:db}
	\frac{db_0}{dt} = ig, \qquad \frac{db_1}{dt} = \frac{1}{\cos 2\nu}|b_1|^2e^{-i(\nu_+ + \nu_-)}, 
 \end{gather}
 	and the assumption \eqref{eq:v-rep} for the boundary velocity is consistent.
\end{lemma}

\begin{proof}	  
 Take $k = 2$ for concreteness, as the general case is analogous. Under our regularity assumptions, we can rewrite \eqref{eq:psi} as
   \be\label{eq:psi-asymptotic}
   \psi =- \Re\left(c_0 + c_1(z - z_*) + b^*_0 U^*\right) -  \frac{1}{2}|b_1|^2 \, |z - z_*|^2  + h, \quad h\in\Lcal^{3+\frac{1}{2}}_{2,\beta+\frac{1}{2}}(m),
   \ee
   where we have used \eqref{eq:v-rep} and we have set
   $$
   c_0 := \frac{1}{2}|b_0|^2, \qquad c_1 := b^*_0 b_1 - ig.
   $$
   
   Under the condition $\left|(\beta - 3) + \frac{1}{2} \right| < \frac{\pi}{2\nu}$, we therefore have
   $$
   	\Gcal \psi = \Im \left(c_1 z_{s} + b_0^* \pa_s U^* \right) + |b_1|^2\, |z - z_*|\tan 2\nu  +  \Lcal^{2+\frac{1}{2}}_{2,\beta + \frac{1}{2}}(m),
   $$
   by Proposition \ref{prop:dirichletInt} and Lemma~\ref{lem:special_sol}. Note that, to use Lemma~\ref{lem:special_sol}, we have set $c_+ = c_- := -\frac{1}{2}|b_1|^2$; we have also used that $\pa_n U^* = i\pa_sU^*$ by~\eqref{eq:CauchyU}. Using \eqref{eq:v-rep} to rewrite the last term in \eqref{eq:sigma}, we thus conclude
   $$
   \aligned
   \si &= \Gcal\psi - \Im(c_1 z_s + b_0^* \pa_s U^*) + \Lcal^{2+\frac{1}{2}}_{2,\beta + \frac{1}{2}}(m) \\
   	   &= |b_1|^2\, |z - z_*|\tan 2\nu  +  \Lcal^{2+\frac{1}{2}}_{2,\beta + \frac{1}{2}}(m).
   \endaligned
   $$ 
   
We prove \eqref{eq:db} next. We have
$$
\aligned
\pa_t (D_s V^*) &= - \frac{z_{ts}}{z^2_s} \pa_s V^* + D_s V^*_t= -\frac{1}{z_s^2} |\pa_sV|^2 +  D_s V^*_t.
\endaligned
$$
Using the Euler equation~\eqref{eqs:evolGa} to express $V^*_t$ in terms of $\si$, the asymptotic formula \eqref{eq:v-rep} for $V$ together with the asymptotic formula~\eqref{eq:si-rep} for $\si$ that we have just derived yield
$$
\aligned
&b_{0t} = ig\\
&b_{1t} + \pa_t \left(D_s U^*\right) = \pa_t \left(D_s V^*\right) = -\frac{1}{z_s^2}|b_1|^2\left(1 + i\pa_s|z - z_*| \tan 2\nu\right) + \Lcal^{1+\frac{1}{2}}_{2,\beta + \frac{1}{2}}(m).
\endaligned
$$
If we then set $b_{1t}$ to be the zero-order term on the right-hand side, provided the limit exists, and include everything else in the time-derivative of $D_s U^*$, we conclude that the ansatz~\eqref{eq:v-rep} is consistent. 

Our next goal is to analyze the one-sided limits of these quantities as $z \rightarrow z_*$. To this end, as in Subsection~\ref{ss.cov}, consider a small neighborhood~$B$ of the corner point~$z_*$ and denote by $\Ga^\pm$ (with $\Ga \cap B = \Ga^+ \cup \Ga^-$ and $\Ga^+ \cap \Ga^- = \{z_*\}$) the arcs that form the (curvilinear) corner $\Ga\cap B$. Without loss of generality, we can assume that these arcs can be parametrized as graphs over the horizontal axis. If the corresponding parametrized curves are then denoted by $z^\pm\equiv z^\pm(x)$ and for convenience we reverse the orientation on the lower arc $z_-$, we can write
\[
z_\pm  = z_*+ (1 \pm i \tan \nu_\pm)x + O(|x|^{1 + \lambda_\beta}),
\]
with $\lambda_\beta$ defined by~\eqref{E.deflambda}. In the limit $x\rightarrow 0^\pm$, we obtain 
$$
\pa_s |z_\pm - z_*| = \pm \frac{1}{|z'_\pm|}\frac{d}{dx} |z_\pm - z_*| \rightarrow \pm 1,
$$ 
and therefore
$$
e^{- 2i\nu_\pm}(1 \pm i \tan 2\nu) = -\frac{1}{\cos 2\nu}e^{- 2i\nu_\pm \pm 2i\nu} = -\frac{1}{\cos 2\nu}e^{-i(\nu_+ + \nu_-)}
$$
using the definition $2\nu := \pi + \nu_+ - \nu_-$.

It remains to show that $\pa_s^3\si\in\Lcal_{2,\beta + 1}(m)$. Note that for this we only use $U\in \Lcal^{3}_{2,\beta + 1}(m)$, not the additional control on half a derivative. Just as in the proof of relations \eqref{eq:th_ts} and \eqref{eq:th_ts12} in Lemma \ref{lem:param}, the proof will rely on repeated applications of Lemma \ref{lem:BRcancel} together with  Lemmas \ref{lem:trHilbert} and \ref{lem:derFphi}. 

These results apply directly when the weight is Muckenhaupt, that is, in the case $\beta\in\left(\frac12, \frac12\right)$. However, in the case $\beta=\frac12$, the argument remains valid with the only proviso that one must replace the Hilbert transform $H$ and the Cauchy integral $\Ccal$ by the operators
\be\label{E.cacambio}
\frac{1}{m^{\lambda'}} H m^{\lambda'} \qquad\text{and} \qquad \frac{1}{(z - z_*)^{{\lambda'}}}\Ccal (z - z_*)^{\lambda'}
\ee 
respectively. Here ${\lambda'} \in (0,1)$ is any fixed parameter. With this replacement, and up to the appearance of additional harmless commutator terms (e.g., with $(m/z)^{\lambda'}$) that can then be estimated using Lemma \ref{lem:derFphi}, the proof in the case $\beta = \frac12$ goes just as  in the case $\beta<\frac12$. Therefore, we can focus on proving the aforementioned estimate in the case $\beta \in(- \frac12,\frac12)$.

First, note that, by the above, we already know that $\Psi^* = V_t^* - z_t D_s V^*\in H^2_{\beta + 1}(m)$, see also \eqref{eq:Psi}--\eqref{eq:psi}. In complex notation, we can write 
$$
\aligned
\pa_s^2 \si &=\pa_s^2 \left(\Psi \cdot z_s^\perp + V_s\cdot V^\perp + gz_{1s}\right) - H\pa_s^2 \left(\Psi \cdot z_s + V_s\cdot V + gz_{2s}\right) \\
&= \Re\left((iI - H)\pa_s^2\left(\Psi^*z_s + V \pa_sV^* - igz_s\right) \right),
\endaligned
$$
where $I$ is the identity operator and $H$ is the Hilbert transform. Furthermore, setting $\t V := V - b_0^*$ and $c_1 := b_0^*b_1 - i g$, we can write 
\begin{gather}
	V\pa_s V^* = z_s VD_sV^*  = z_s\t V D_sV^* + z_s b_0^* (b_1  + D_s U^*),\\ 
\label{eq:Phi}
\Psi_1 := \Psi^* + c_1 + b_0^*D_sU^* = b_{1t}(z - z_*) + \Lcal^2_{2,\beta}(m), \qquad D_s U^*\in \Lcal^2_{2,\beta + 1}(m).
\end{gather}

Since both $\Psi^*$ and $D_sU^*$ are boundary values of holomorphic functions by~\eqref{E.HPsi} and~\eqref{eq:CauchyU}, we can further write using the above formula for $\pa_s^2\si$
$$
\aligned
\pa_s^2\si 
&=  \Re\left((i I - H)\pa_s^2\left[z_s\left(\Psi_1 + \t V D_s U^* +  b_1 \t V\right)\right]\right)\\
&=\Re\left((iI - H)\left[z^3_s\left(D_s^2\Psi_1 + \t V D^3_s U^*\right) \right]  \right) +  \Lcal^1_{2,\beta + 1}(m).
\endaligned
$$
Here we have used the fact that  $\t V\in H^3_{\beta + 1}(m)$ together with Lemma~\ref{lem:trHilbert}. 

Also, as $D_s^2\Psi_1 = \Ccal D_s^2\Psi_1$, we now have  
$$
\aligned
iz^3_sD_s^2\Psi_1 - H\left(z^3_sD_s^2\Psi_1\right) &= z^3_s(i\Ccal - H)D_s^2\Psi_1 + [z^3_s, H]D_s^2\Psi_1\in \Lcal^1_{2,\beta + 1}(m)
\endaligned
$$
by Lemmas \ref{lem:BRcancel} and \ref{lem:derFphi}. Similarly, temporarily setting $z_* = 0$ for  ease of notation without any loss of generality, we can write 
$$
\aligned
iz^3_s \t V D^3_s U^* - H\left(z^3_s\t V D^3_s U^* \right) &= \frac{z^3_s\t V}{z} iz D^3_sU^* - H\left(\frac{z^3_s\t V}{z} z D_s^3 U^* \right) \\
&=\frac{z^3_s\t V}{z} (i\Ccal - H) (zD^3_sU^*) + \left[\frac{z^3_s\t V}{z}, H\right]\left( z D_s^3 U^* \right) \in \Lcal^1_{2,\beta + 1}(m)
\endaligned
$$
by Lemmas \ref{lem:BRcancel} and \ref{lem:derFphi}. Here we have used that $\t V = b^*_1 z^*  + \Lcal^3_{2,\beta + 1}(m)$, which implies that $j$-th space derivative of $\frac{z_s^3\t V}{z} = O(1)$  grows as $O(m^{-j})$ for $j\leq 2$.
\end{proof}

   \begin{lemma}\label{lem:si_t} Let $\left|(\beta - 3) + \frac{1}{2} \right| < \frac{\pi}{2\nu}$ and let $(z, V)$ satisfy \eqref{eq:basic-energy-phi} for some $\beta\in\left(\frac12, \frac12\right]$ and $k\geq 2$. Then, the time derivative of $\si$ is bounded as
 	$$
	\si_t = O(|z-z_*|).
	$$
   \end{lemma}
   
   \begin{proof} 
   Taking the time derivative of the second equation in \eqref{eqs:evolGa}, we can write
   	$$
   	V_{tt}\cdot z_s = -\si \theta_t, \qquad V_{tt}\cdot z_s^\perp = \si_t.   	
   	$$
   	Armed with these equations, the claim will essentially follow from Proposition~\ref{prop:dirichletInt} and Lemma~\ref{lem:special_sol}, once we have derived a formula  for $\si_t$ analogous to~\eqref{eq:sigma} (that is, once we have expressed $V_{tt}$ in terms of $z$, $V$ and $V_t$). Note that the results of Section~\ref{s.inverse} are stated for fractional Sobolev spaces, where one controls half a derivative more. We will not need the estimate for the additional half-derivative for the proof of this lemma.
   	   	
   	On the interface, recall that
   	$$
   	\Psi^* = V^*_t - V D_s V^*, \qquad  
   	\Psi^*_t = V^*_{tt} - V_tD_s V^* + \frac{1}{z_s^2}V|V_s|^2 - VD_sV^*_t,
   	$$
   	where, as discussed in the proof of the previous lemma, we already know that $\Psi^*$ and $D_sV^*$ are the boundary traces of holomorphic functions in $\Om$. On the other hand, from Equation \eqref{eq:Phi}, we know that $\Psi^* + c_1 \in \Lcal^1_{2,\beta}(m)$ with $c_1:= b_0^*b_1 - ig$. In particular, we have $D_s\Psi^* = \Ccal D_s\Psi^*$ by Equation~\eqref{eq:BR-ipp}. Therefore, by taking the time derivative of $\Psi^* = \Ccal\Psi^*$, we conclude that 
   	$$
   	\Psi^*_2 := \Psi^*_t - z_tD_s\Psi^*
   	$$
   	satisfies the same equation (i.e, $\Ccal\Psi_2^*=\Psi_2^*$). $\Psi_2^*$ is therefore the boundary trace of a holomorphic function. 
   	
   	Using now the evolution equation $z_t = V$, we can write
   	$$
   	\aligned
   	V^*_{tt} &= \Psi^*_t + V_tD_sV^* - \frac{1}{z_s^2}V|\pa_s V|^2 + VD_sV^*_t\\
   	   	&= \Psi^*_2 + VD_s\Psi^* + (\Psi + V^*(D_sV^*)^*)D_s V^* -  \frac{1}{z_s^2}V|\pa_sV|^2 + V D_s\Psi^* + V D_s(VD_s V^*)\\
   	&= \Psi^*_2 + 2VD_s\Psi^* + \Psi D_s V^* + V^*|D_s V^*|^2 + V^2D_s^2 V^*.
   	\endaligned
   	$$
   	From \eqref{eq:v-rep} and \eqref{eq:Phi}, it then follows that
$$
\aligned
2VD_s\Psi^* &= 2b_0^* D_s\Psi^* + 2b^*_1(z - z_*)^*(b_{1t} - b_0^* D^2_s U^*)  + \Lcal^1_{2,\beta - 1}(m)\\ 
\Psi D_s V^* &= -c_1^*b_1 -c_1^*D_sU^* - b_0b_1 (D_sU^*)^* - b_0 |D_sU^*|^2  + b_1b_{1t}^*(z-z_*)^* + \Lcal^1_{2,\beta - 1}(m)\\
V^*|D_s V^*|^2 &= b_0|D_s V^*|^2 + b_1|b_1|^2(z-z_*) + \Lcal^1_{2,\beta - 1}(m)\\
V^2D_s^2 V^* & = (b_0^*)^2 D_s^2U^* + 2 b^*_0b^*_1(z-z_*)^*D_s^2U^* + \Lcal^1_{2,\beta - 1}(m).
\endaligned
$$
Combining the first and the last equations, we obtain
$$
2VD_s\Psi^* + V^2D_s^2 V^* = 2b_0^* D_s\Psi^* + (b_0^*)^2 D^2_sU^* + 2b_{1t}b^*_1(z - z_*)^* + \Lcal^1_{2,\beta - 1}(m).
$$
The definition of $c_1$ and the second and third identities then yield
$$
\aligned
\Psi D_s V^* + V^*|D_s V^*|^2  =& -c_1^*(b_1 + D_sU^*) - b_0b_1 (D_sU^*)^* +  b_0(|b_1|^2 + 2\Re(b_1 (D_sU^*)^*)\\ 
&+ b_1|b_1|^2(z-z_*) + b_1b_{1t}^*(z-z_*)^*  + \Lcal^1_{2,\beta - 1}(m)  \\
=&ig(b_1 + D_sU^*) + b_1|b_1|^2(z-z_*) + b_1b_{1t}^*(z-z_*)^* + \Lcal^1_{2,\beta - 1}(m). 
\endaligned 
$$

Putting everything together, we thus conclude
$$
V^*_{tt} = \Psi^*_3 + (2b_{1t}b^*_1 + b_1b_{1t}^*) (z - z_*)^* +  \Lcal^1_{2,\beta - 1}(m),
$$
where in $\Psi^*_3$ we have collected all the terms which can be written as the boundary trace of a holomorphic function. Note that
$$
\aligned
\Psi_3\cdot z_s &= -\si \theta_t - 3\Re(b_1^* b_{1t})(z-z_*)\cdot z_s +  \Lcal^1_{2,\beta - 1}(m) \\
&= \si \Im (e^{2i\nu_\pm} b_1) - 3\Re(b_1^* b_{1t})(z-z_*)\cdot z_s + \Lcal^1_{2,\beta - 1}(m),
\endaligned
$$ 
where we have used that $(z-z_*)\cdot z_s^\perp\in \Lcal^2_{2,\beta}(m)$. Therefore, by Proposition~\ref{prop:dirichletInt} and Lemma~\ref{lem:special_sol}, there exists a harmonic extension of $\psi_3 = \int \Psi_3\cdot z_s\, ds$ (recall that this means that $\psi_3$ as a primitive of $\Psi_3\cdot z_s$). The action of the Dirichlet-to-Neumann map on~$\psi_3$ is bounded as $\Gcal\psi_3 = \Psi_3\cdot z_s^\perp = O(|z - z_*|)$, as functions in $\Lcal^1_{2,\beta - 1}(\Ga)$ are bounded by $|z - z_*|^{1 + \lambda_\beta}$ by Lemma \ref{lem:sobolev}. 
\end{proof}

\subsection{The basic energy estimate}\label{ss.apriori-energy}

We are now ready to prove the fundamental energy estimate that we will use to establish Theorem~\ref{T.main}. The energy we shall define is modeled on the usual energy functionals for smooth water waves, but involves a careful interplay between weights to effectively capture the evolution of non-rigid corners.

It is convenient to distinguish between the lower-order contributions to the energy,
\be\label{def:energy-low}
E^\low_{k,\beta}(t)^2 := \|\theta\|^2_{H^k_{\beta + (k - 1)}} +  \|\log |z_\al|\|^2_{H^k_{\beta + (k - 1)}} + \|V\|^2_{H^k_{\beta + (k - 2)}} + \|\varphi_s\|^2_{H^k_{\beta + (k - 1)}} + \Fcal(z), 
\ee
and the higher-order part
\be\label{def:energy-high}
E^\high_{k,\beta}(t)^2 := \|\sqrt{\si}\pa_s^{k+1}\theta\|^2_{2, \beta + k - 1/2} +  \left\|\frac{1}{m^{\lambda}\sqrt{|z_\al|}}\Lambda^{1/2}(m^\lambda\pa_s^k\varphi_s)\right\|^2_{2, \beta + k - \frac{1}{2}}.
\ee
The full energy functional is then defined as
$$
E_{k, \beta}(t)^2 = E^\low_{k,\beta}(t)^2 + E^\high_{k,\beta}(t)^2.
$$
Here $k\geq 2$ is an integer and, following the definition of our weighted Sobolev space of half-integer order in Section~\ref{s.preliminaries}, we fix some real $\lambda \equiv \lambda(k)$ in such a way that both $(\beta - \lambda) + (k - \frac{1}{2})$ and $(\beta - \lambda) + (k - 1)$ satisfy Muckenhaupt condition, that is
$$
(\beta - \lambda) + (k - 1) \in\left(-\frac{1}{2}, 0\right).
$$

Concerning the Rayleigh--Taylor condition, recall that $\si > 0$ away from the corner tip by the discussion in Section~\ref{s.preliminaries} and that Lemma~\ref{lem:si} implies 
\[
|z - z_*|^{-1}\si \approx 1
\]
on the whole interface provided that $b_1 \neq 0$ and that the opening angle $2\nu \in \left(0, \frac\pi2\right)$ (also recall that Lemma~\ref{lem:si} is only valid as long as $-\frac{\pi}{2\nu} + \frac52 < \beta \leq \frac12$). 

We are now ready to prove our basic estimate. Again, in the proof we will use a number of results from Section~\ref{s.inverse} and Appendix~\ref{A.A}. As a side remark, the exponential bound in \eqref{eq:energy} is not optimal, but an artifact of our definition of the energy as we chose to control $\log |z_\al|$ instead of $|z_\al|$. We could prove a polynomial bound by slightly modifying the energy. This is irrelevant for the purpose of proving our local wellposedness theorem, however.

\begin{lemma}\label{lem:apriori} Let $k\geq 2$ and let $(z,V)$ be a (sufficiently regular) solution of \eqref{eqs:evolGa}, where $V$ is given by \eqref{eq:v-rep}-\eqref{eq:CauchyU}. Suppose that the (degenerate)  Rayleigh--Taylor condition 
	$$
	\inf_{\al\in I}\frac{\si(\al, t)}{|z(\al, t) - z_*(t)|} > A
	$$ 	
holds.	Then we have the a priori energy estimate
	\be\label{eq:energy}
	\frac{d}{dt} E_{k, \beta}(t)^p \, \lesssim \, \frac{1}{A^p} \exp C E_{k, \beta}(t)^p
	\ee
	for some $p\in \Nbb$ and some constant $C>0$.
\end{lemma}

%
%
%
%
\begin{proof}
	We prove the claim for $k=2$; the proof for higher~$k$ is analogous. The first observation is that we only need to consider time derivatives of the highest-order terms in $\theta$ and $\varphi_{s}$, as the estimates on remaining terms follow from the lemmas proven in Subsection~\ref{ss.PE}. 
	
	We claim that		
	$$
	\frac{1}{2}\frac{d}{dt}\left( \big\|\sqrt{\sigma}\pa_s^3\theta\big\|^2_{2, \beta + 3/2} + 	 \big\|\Lambda^{1/2}\big(m^{\lambda}\pa_s^2\varphi_{s}\big)\big\|^2_{2, \beta -\lambda + 3/2} \right) 
	$$
	is bounded by the right hand side of \eqref{eq:energy}. 
	
	To see this, we first show that
	\be\label{en1}
	\frac{1}{2}\frac{d}{dt} \big\|\sqrt{\sigma}\pa_s^3\theta\big\|^2_{2, \beta + 3/2} = \int m^{2\beta + 3 -\lambda} \si \pa_s^3\theta \, \Lambda (m^\lambda\pa_s^2\varphi_{s}) d\al +  \Rcal_1(t),
	\ee
	where the remainder $\Rcal_1(t)$ is bounded by the right hand side of~\eqref{eq:energy}. To prove this, we start with the identity
	$$
	\frac{1}{2}\frac{d}{dt} \big\|\sqrt{\sigma}\pa_s^3\theta\big\|^2_{2, \beta + 3/2} = \int m^{2(\beta + 3/2)} \si \pa_s^3\theta \,  (\pa_s^3\theta)_{t}  \, ds + \int m^{2(\beta + 3/2)} \Big(\frac{\si_t}{\si} + \varphi_s\Big) \, \si|\pa_s^3\theta|^2 ds,
	$$
	where $ds = |z_\al|d\al$ is the line element. The last two terms clearly satisfy the desired estimate because $\varphi_s = \frac{|z_\al|_t}{|z_\al|} = O(1)$ and $\frac{\si_t}{\si} = O(1)$ by our assumptions and Lemmas~\ref{lem:si} and~\ref{lem:si_t}. 
	
	To estimate $(\pa_s^3\theta)_{t}$, we successively interchange $\pa_t$ with $\pa_s = \frac{1}{|z_\al|}\pa_\al$ and use the fact that $\varphi_s = \frac{|z_\al|_t}{|z_\al|}$ to write  
	$$
	\aligned
	(\pa_s^3\theta)_{t} &= \pa^3_s\theta_{t} - \pa_s^2(\varphi_s \theta_{s}) - \pa_s(\varphi_s \pa_s^2\theta) - \varphi_s \pa_s^3\theta \\
	&=\pa_s (m^{ - \lambda}H(m^{\lambda}\pa^2_s\varphi_{s})) + \Lcal_{2, \beta + 2}(m).
	\endaligned
	$$
	Here, we repeatedly used Lemma \ref{lem:sobolev} and our regularity assumptions on $\theta$ and $\varphi_s$, together with Lemmas \ref{lem:param}, \ref{lem:trHilbert} and \ref{lem:derFphi} to show that we can write 
	\be\label{eq:th_tsslambda}
	\pa^2_s\theta_{t} = m^{-\lambda}H(m^{\lambda}\pa^2_s\varphi_{s}) + \Lcal^1_{2, \beta + 2}(m)
	\ee
	in the allowed range of $\beta$'s. In fact, note that we an write
	\be\label{eq:th_tslambda}
	\pa_s\theta_{t} = m^{-\lambda + 1}H(m^{\lambda - 1}\pa_s\varphi_{s}) + \Lcal^2_{2, \beta + 2}(m).
	\ee
	In the case $\beta = \frac12$, this is precisely the relation \eqref{eq:th_ts12} with $\lambda' := \lambda - 1\in(\frac12, 1)$, while in the case $\beta \in(- \frac12,\frac12)$ this is the relation \eqref{eq:th_ts} up to a commutator with $m^{\lambda - 1}$, which can be readily estimated using Lemma~\ref{lem:derFphi}. Recall that, by construction, the weight with power $\beta - \lambda + 1$ satisfies the Muckenhaupt $A_2$-condition.
	
	Taking a derivative of the remaining term in \eqref{eq:th_tslambda}, and using Lemma \ref{lem:trHilbert}, we find
	\be\label{eq:HpaScomm}
	\aligned
	\pa_s \left(\frac{1}{m^{\lambda - 1}}H(m^{\lambda - 1}\pa_s \varphi_s)\right) =& \frac{1}{m^\lambda}H(m^\lambda  \pa^2_s\varphi_s) +  \frac{1}{m^{\lambda} |z_\al|}[H, |z_\al|] (m\pa_s(m^{\lambda - 1}\pa_s\varphi_s)) \\ 
	&+  \frac{\lambda - 1}{m^\lambda }[H, m_s] (m^{\lambda - 1}\pa_s\varphi_s) + \Lcal^1_{2,\beta + 2}\\ 
	&= \frac{1}{m^\lambda}H(m^\lambda  \pa^2_s\varphi_s) + \Lcal^1_{2,\beta + 2}(m).
	\endaligned
	\ee
	Here we have used Lemma \ref{lem:derFphi} and the fact that $|z_\al| \in H^2_{\beta + 1}(m)$. In particular, we have proved the identity \eqref{eq:th_tsslambda}. Recalling the identity $H\pa_\al  = \Lambda$, \eqref{en1} follows.


Now we shall show that
	\be\label{en2}
	\frac{d}{dt} \left\|\frac{1}{m^\lambda\sqrt{|z_\al|}}\Lambda^{1/2}\big(m^{\lambda}\pa_s^2\varphi_{s}\big)\right\|^2_{2, \beta + 3/2} = -\int m^{2\beta + 3 -\lambda} \si \pa_s^3\theta \, \Lambda (m^\lambda\pa_s^2\varphi_{s}) d\al  +  \Rcal_2(t),
	\ee
	where the remainder $\Rcal_2(t)$ is again bounded by the right hand side of~\eqref{eq:energy}. To keep the notation simple, we set
	\be\label{eq:beta'}
	\beta' := \beta - \lambda + 1
	\ee
	and note that $\lambda$ is chosen in such a way that both $\beta'$ and $\beta' + \frac{1}{2}$ satisfy the Muckenhaupt $A_2$-condition, i.e. $\beta' \in\left(-\frac{1}{2}, 0\right)$. 
	
	Taking the time derivative, we find
	$$
	\frac{1}{2}\frac{d}{dt} \left\|\frac{1}{\sqrt{|z_\al|}}\Lambda^{1/2}\big(m^{\lambda}\pa_s^2\varphi_{s}\big)\right\|^2_{2, \beta' + 1/2} = \int_{-\pi}^\pi m^{2\beta' + 1}  \Lambda^{1/2}\big( m^{\lambda}  (\pa_s^2\varphi_{s})_t\big) \Lambda^{1/2}\big(m^\lambda \pa_s^2\varphi_{s}\big)d\al.   
	$$
	Noting that 
	\be\label{eq:varphist}
	\varphi_{st} = -\varphi_s^2 + \theta_t^2 -\si \theta_s\in H^2_{\beta + 1}(m),
	\ee
	by Lemmas \ref{lem:param} and \ref{lem:si}, it is not hard to show that 
	\be\label{eq:varphissst}
	\aligned
	(\pa_s^2\varphi_{s})_t &=- 3(\pa_s\varphi_{s})^2 - 4\varphi_{s} \pa_s^2\varphi_{s} + \pa_s^2(- \si\theta_s + \theta_t^2) \\
	&= -4\varphi_{s} \pa_s^2\varphi_{s} - 2\theta_t m^{-\lambda}H(m^\lambda \pa^2_s\varphi_{s}) - \si \pa_s^3\theta + \Lcal^1_{2,\beta + 2}(m).
	\endaligned
	\ee
	Indeed, clearly $(\pa_s\varphi_{s})^2 \in \Lcal^1_{2,\beta + 2}(m)$ by Lemma \ref{lem:sobolev}, while \eqref{eq:th_tsslambda} and Lemma \ref{lem:param} ensure that
	$$
	\frac{1}{2}\pa_s^2(\theta_t^2) =  \theta_t \pa_s^2\theta_{t} + (\theta_{ts})^2 = - \theta_t m^{-\lambda}H(m^\lambda \pa^2_s\varphi_{s}) +\Lcal^1_{2,\beta + 2}(m).
	$$
	Finally, we know that $\si \in H^3_{\beta + 1}(m)$ by Lemma \ref{lem:si} and therefore 
	\be\label{aux:dersith}
	\pa_s^2(\si \theta_s)  - \si \pa_s^3\theta = 2 \si_s \theta_{ss} + \si_{ss}\theta_s \in \Lcal^1_{2,\beta + 2}(m).
	\ee
	Thus \eqref{eq:varphissst} follows. 
	
	Using the decomposition~\eqref{eq:varphissst}, we now claim 
	\be\label{eq:varphissst2}
	\Lambda^{1/2}(m^\lambda (\pa_s^2\varphi_{s})_t) = - \Lambda^{1/2}(m^\lambda\si \pa_s^3\theta )  + \Lcal_{2, \beta' + \frac{1}{2}}(m).  
	\ee
	To see this, note that when the half-derivative $\Lambda^{1/2}$ hits the terms in $m^\lambda(\pa_s^2\varphi_{s})_t$ that are in $\Lcal^1_{2,\beta' + 1}(m)$, one gets terms in $\Lcal_{2, \beta' + \frac{1}{2}}(m)$ by Lemma~\ref{lem:halfDer}. On the other hand, as $\varphi_s = O(1)$, $\pa_s\varphi_{s} = O(m^{-1  + \lambda_\beta})$ by Lemma \ref{lem:sobolev}, we have 
	$$
	\Lambda^{1/2}(m^\lambda \varphi_{s}\pa_s^2\varphi_{s}) = \varphi_{s}\Lambda^{1/2}(m^\lambda \pa_s^2\varphi_{s}) +  \left[ \Lambda^{1/2}, \varphi_s\right]\left(m^{\lambda}  \pa_s^2\varphi_{s}  \right) \in \Lcal_{2, \beta' + \frac12}(m),
	$$
	by Lemma \ref{lem:comm_gb}. The same estimates hold for $\theta_t$ (i.e. $\theta_t = O(1)$, $\pa_s\theta_{t} = O(m^{-1  + \lambda_\beta})$), and therefore
	$$
	\Lambda^{1/2}\left(\theta_t H(m^\lambda \pa^2_s\varphi_{s})\right) = \theta_t\Lambda^{1/2}H(m^{\lambda}\pa_s^2\varphi_{s}) + \left[\Lambda^{1/2}, \theta_t\right]H(m^\lambda \pa^2_s\varphi_{s}) \in \Lcal_{2, \beta' + \frac12}(m)
	$$
 	by Lemmas \ref{lem:comm_gb} and \ref{lem:trHilbert}. We have thus proven~\eqref{eq:varphissst2}.

	It remains to consider the term in \eqref{eq:varphissst2} that involves $\si$. We have
	$$
	-\int m^{2\beta' + 1}  \Lambda^{1/2}\left( m^{\lambda} \si \pa_s^3\theta  \right) \Lambda^{1/2}\left(m^\lambda \pa_s^2\varphi_{s}\right)d\al = \int (m^{\lambda} \si \pa_s^3\theta) \Lambda^{1/2}\big( m^{2\beta' + 1}  \Lambda^{1/2}\big(m^\lambda \pa_s^2\varphi_{s}\big)\big)d\al.
	$$
	Let us now write
	$$
	\Lambda^{1/2}\left( m^{2\beta' + 1}  \Lambda^{1/2}\left(m^\lambda \pa_s^2\varphi_{s}\right)\right) = -m^{2\beta' + 1}\Lambda \left(m^\lambda \pa_s^2\varphi_{s}\right) +\left[\Lambda^{1/2},  m^{2\beta' + 1}\right] \Lambda^{1/2}\left(m^\lambda \pa_s^2\varphi_{s}\right).
	$$
	Since by assumption $\Lambda^{1/2}(m^\lambda \pa_s^2\varphi_{s}) \in\Lcal_{2,\beta' + 1/2}(m)$, Lemma \ref{lem:comm_weight} implies
	$$
	m^{\lambda}\left[\Lambda^{1/2},  m^{2\beta' + 1}\right] \Lambda^{1/2}\left(m^\lambda \pa_s^2\varphi_{s}\right)\in \Lcal_{2, -(\beta+1)}(m).
	$$
	Therefore, using that $\si \pa_s^3\theta\in \Lcal_{2, \beta + 1}(m)$, we conclude   
	$$
	-\int  m^{2\beta' + 1}\Lambda^{1/2}\left( m^{\lambda} \si \pa_s^3\theta  \right) \Lambda^{1/2}\left(m^\lambda \pa_s^2\varphi_{s}\right)d\al = -\int  m^{2\beta -\lambda + 3}\si \pa_s^3\theta \, \Lambda(m^\lambda\pa_s^2\varphi_{s}) d\al + \, \text{`bounded'},
	$$
	where `bounded' stands for terms that are bounded by the right hand side of~\eqref{eq:energy}.
This proves~\eqref{en2}, so the lemma follows.  
\end{proof}


\section{The local existence theorem}
\label{s.regularization}

We are now ready to state and prove a precise version of our local wellposedness result, Theorem~\ref{T.main}. 

\subsection{Statement of the result}

In the setting we consider, we will describe the fluid by three real functions, $(\Theta,\log|z_\al|,U_1)$, and five parameters (two real and three complex): $(b_0,b_1,z_*,\nu_+,\nu_-)$. 

We  denote by ${\mathcal{B}}^k_{\beta}(m)$, with $\beta \in\left(-\frac{1}{2}, \frac{1}{2}\right]$ and $k\geq 2$, the Banach space of all   
\begin{gather*}
(z_*, \nu_+, \nu_-, \Theta, \log|z_\al|) \in \Cbb \times \Rbb \times \Rbb \times \Lcal^{k+1}_{2, \beta + k}(I, m) \times H^{k+1/2}_{\beta + k - 1/2}(I, m),  \\
	(b_0, b_1, U_1) \in \Cbb\times\Cbb\times \Lcal^{k + 3/2}_{2,\beta + k -1 /2}(I, m),
\end{gather*}
equipped with the natural norm. As we shall see next, the former 5-tuple characterizes the interface, and the latter triple characterizes the velocity on the interface.
(One can also consider the case $k=1$ using the same ideas, with some minor modifications, but we will not pursue this idea here.) 

Let us start with the interface. Given $(z_*, \nu_+, \nu_-, \Theta, \log|z_\al|)$ as above, we define the tangent angle and interface curve as
\be\label{z}
z := z_* + \int_{-\pi}^\al e^{\log|z_\al| + i\theta}d\al', \qquad \theta := \left(\nu_+\chi_+ + \nu_- \chi_-\right) + \Theta\,,
\ee
where $\chi_\pm$ are fixed cutoffs as in the introduction to Section \ref{s.apriori}.  We assume that the opening angle $2\nu$ and $\beta$ satisfy the compatibility condition 
\be\label{eq:nu}
2\nu \in \left(0, \frac{\pi}{2}\right), \qquad -\frac{\pi}{2\nu} + 3 < \beta + \frac{1}{2}
\ee
and that the curve~$z$ satisfies the arc-chord condition
\be\label{eq:Fcal}
\Fcal(z) < A
\ee
for some possibly large constant $A>0$. Let us recall that the arc-chord term~$\Fcal(z)$ is in fact included in the energy functional.

On the other hand, given $(b_0, b_1, U_1)$ as above, we define the complex conjugate of the  boundary trace of the velocity as 
\be\label{V}
V^* := b_0 + b_1(z - z_*) + U^*, \qquad  U := U_1 + i \int \Gcal(z)U_1 ds, 
\ee
where $\Gcal(z)$ is the Dirichlet-to-Neumann operator associated to $z$.

We assume that the Rayleigh--Taylor condition
\be\label{RT}
\inf_{\al\in I}\frac{\si}{|z - z_*|} > 1/A
\ee	
holds for some constant $A > 0$, where $\si$ (i.e., minus the normal derivative of the pressure) is given in terms of the above unknowns as
\be\label{eq:si-dvn}
\si(z, V) := \Gcal(z)\psi + gz_{1s} + V_s \cdot V^\perp, \qquad \psi := -gz_{2} - \frac{1}{2}|V|^2.
\ee
In view of the discussion in Section \ref{s.preliminaries} and of Lemma \ref{lem:si} (as elaborated in the proof of Equation~\eqref{E.RTfails} below), a Rayleigh--Taylor condition of the form \eqref{RT} holds provided that
		\be\label{RT-zero}
		|b_1(t)| > 1/A'>0
		\ee
(that is,  $\left|\frac{d\overline{v}}{dz}(z_*(t), t)\right| \neq 0$ is bounded away from zero) and the regularity of the interface and the arc-chord condition remain bounded. Therefore, our basic assumption concerning the Rayleigh--Taylor condition will be~\eqref{RT-zero}.

Conditions \eqref{eq:nu}, \eqref{eq:Fcal} and \eqref{RT} (which, as we have mentioned, can be replaced by~\eqref{RT-zero}) define an open set $\Ocal^k_\beta(m)\subseteq\Bcal^k_\beta(m)$. With some abuse of notation, in what follows we shall identify $(z,V)$ with the quantities $(\Theta,\log|z_\al|,U_1,b_0,b_1,z_*,\nu_+,\nu_-)$ we use to define them, and write $(z, V)\in\Ocal^k_\beta(m)$ whenever the $(z,V)$ defined by \eqref{z} and \eqref{V} are such that the conditions~\eqref{eq:nu}, \eqref{eq:Fcal} and \eqref{RT} are satisfied. The meaning of the assertion $(z,V)\in \Bcal^k_\beta(m)$ is defined analogously.

The local wellposedness result we establish, which was informally stated as Theorem~\ref{T.main} in the Introduction, is the following:

\begin{theorem}\label{thm:main} 
The water wave system,
\be\label{eqs:sys}
\aligned
z_t = V\qquad V_t + ig = iz_s \si 
\endaligned
\ee
is locally wellposed on $\Ocal^k_\beta(m)$. In particular, for each initial datum $(z^0, V^0)\in\Ocal^k_\beta(m)$, there exists some time $T>0$ such that~\eqref{eqs:sys} has a unique solution $(z, V)\in\Ccont([0, T], \Ocal^k_{\beta}(m))\cap \Ccont^1([0, T], \Bcal^{k-1}_{\beta}(m))$ with initial datum 
$$
z(\cdot, 0) = z^0, \quad V(\cdot, 0) = V^0.
$$
Furthermore, the evolution of the constants $(z_*,\nu_+,\nu_-,b_0,b_1)$ is given by the system of ODEs defined by~\eqref{eq:nudt} and~\eqref{eq:db}.
\end{theorem}

\begin{remark}\label{R.continuation}
For future reference, let us record that the continuation criterion that follows is that the solution exists as long as $(z,V)\in\Ocal ^k_{\beta}(m)$. Therefore, the solution exists at least until its $E_{k,\beta}$-energy (which we have shown to be equivalent to $\|(z,V)\|_{\Bcal^k_{\beta}(m)} + \|(z_t,V_t)\|_{\Bcal^{k-1}_{\beta}(m)} + \Fcal(z)$) becomes unbounded, the angle~$\nu$ stops satisfying the condition~\eqref{eq:nu}, or the arc-chord condition~\eqref{RT} fails (that is, the infimum in~$t$ of $\inf_\al|z(\al,t)-z_*(t)|\si(\al,t)$ is 0). 

\end{remark}

\begin{remark}
	As evolution of the angle of the corner is given by the ODE system~\eqref{eq:nudt}, the corners are not rigid. Also, note that the velocity field~$v(z,t)$ is of class $C^{1,\lambda}$ up to the boundary of the water region~$\Om(t)$ by Corollary~\ref{C.reg-v}, for some $\lambda>0$.
\end{remark}

Let us discuss the strategy of the proof of this result before getting bogged down in technicalities. The strategy to pass from the a priori energy estimates proven in Section~\ref{s.apriori} to a local wellposedness theorem is standard, see e.g.~\cite{MajBer}. There are many ways to regularize free boundary Euler equations; here, we adapt the approach in \cite{ambrose}, \cite{CCG} and \cite{CoEnGr21}. What makes this step nontrivial is the presence of singular weights, which we handle using functions that exactly vanish at the corner points. This property is certainly not preserved by standard mollifiers. We will strive to focus on the non-standard parts of the argument.

We regularize evolution equations in two steps. We first add a `viscosity'-term to the tangential direction of $\pa_t V$, which depends on a small parameter~$\ep$. Then, we introduce another regularization, depending now on another small parameter~$\delta$. The resulting regularized system, which we call the $(\ep, \de)$-system, is an ODE on an open set of a suitable Banach space. Therefore, we can use the abstract Picard theorem to find a sequence of solutions to the $(\ep, \de)$-system with a subsequence converging to a solution of the $\epsilon$-system as $\de\rightarrow 0^+ $.

\subsection{Regularization of the equations}
\label{ss.regul}

We shall next introduce the regularized system of equations that we will use to prove the local existence result.

We start by defining the Banach spaces ${\t\Bcal}^k_{\beta}(m)$ that we will use in this argument, where $\beta\in\left(-\frac{1}{2},\frac{1}{2}\right]$
 and $k\geq 2$. They are closely related to the above spaces $\Bcal^k_\beta(m)$ and consist of functions and parameters
\begin{gather*}
	(z_*, \nu_+, \nu_-, \Theta, \log|z_\al|) \in \Cbb \times \Rbb \times \Rbb \times \Lcal^{k+1}_{2, \beta + k}(I, m) \times H^{k+1}_{\beta + k}(I, m),  \\
(b_0, b_1, U_1) \in \Cbb\times\Cbb\times \Lcal^{k + 2}_{2,\beta + k}(I, m),
\end{gather*}
endowed with the natural norm.

We then define $z$ as a function of $(z_*, \nu_+, \nu_-, \Theta, \log|z_\al|)$ via \eqref{z} and $V$ as a function of $(b_0, b_1, U_1)$ via \eqref{V}. We denote by $\t\Ocal^k_\beta(m)\subseteq \t \Bcal^k_\beta(m)$ the open set  defined by the conditions \eqref{eq:nu}, \eqref{eq:Fcal} and \eqref{RT-zero} and, just as above, abuse the notation to write $(z, V)\in {\t\Ocal}^k_{\beta}(m)$.

The key difference between this space and $\Bcal^k_\beta(m)$ is that here  we control an extra half-derivative of $U_1$ and $\log|z_\al|$. By looking at the equations, it is clear why we take the additional half-derivative of~$U_1$, but one might believe the extra derivative of $\log|z_\al|$ is perhaps not required. The reason for which it is included (and necessary for our arguments to work) is the following. If we assume $U_1\in \Lcal^{k + 2}_{2,\beta + k}(m)$, we have $U_2\in \Lcal^{k + 2}_{2,\beta + k}(m)$ by Proposition \ref{prop:dirichletInt} (as long as the compatibility condition \eqref{eq:nu} for the opening angle and the parameter $\beta$ holds). Therefore,  $V \in \Lcal^{k + 2}_{2,\beta + k}(m)$, so $\varphi_s(z, V) = V_s \cdot z_s \in H^{k + 1}_{\beta + k}(m)$. As $\varphi_s = (\log|z_\al|)_t$, we really need to take an additional half-derivative of $\log|z_\al|$ as well.

We start by introducing the $\ep$-regularization of the system, for $\ep \geq 0$. Following \cite{CoEnGr21}, we keep the evolution equation for $z$ unchanged and we only modify the evolution equation for $V$, setting:
\be\label{eqs:sys-e}
\aligned
z_t &= V, \\
V_t + ig &= z_s\left(i \si +  \epsilon\, \pa_s\left(m^2\Phi_s\right)\right),
\endaligned
\ee
where we define
\[
\Phi_s := z_s\cdot U_s\in\Lcal^{k+1}_{2,\beta + k}(m).
\]
The reason for which we have used $\Phi_s$ in the regularization instead of  $\varphi_s = \Phi_s + \Re(b_1 z_s^2)$ is that $\pa_s (m^2\Phi_s) \in \Lcal^{k}_{2,\beta + k - 2}(m)$. Hence, this way we ensure are not introducing corrections asymptotic to $z - z_*$ near the corner point. Note that, keeping all other regularity assumptions fixed,  controlling $\Phi_s\in\Lcal^{k + 1}_{2,\beta + k}(m)$ is equivalent to controlling $\varphi_s\in H^{k + 1}_{\beta + k}(m)$.

As before, $V^*$ is by construction the boundary trace of a holomorphic function, which must therefore satisfy \eqref{eq:Cauchy-V}. Taking a time derivative of \eqref{eq:Cauchy-V}, it is not difficult to see that $V_t^* - z_t D_s V^*$ should also satisfy an equation of this form,
\be\label{E.Ccalmess} 
\Ccal(z) (V_t^* - z_t D_s V^*)=V_t^* - z_t D_s V^*\,.
\ee
Therefore, instead of taking the prescription~\eqref{eq:si-dvn}, define~$\si$ in this regularized setting by
\be\label{eq:si-eps}
\si(z, V, \epsilon) := \Gcal(z)(\psi + \psi^\ep) + gz_{1s} + V_s \cdot V^\perp ,
\ee
where $\psi$ and $\psi^\ep$ are
\be\label{eq:psi-eps}
\psi = - gz_{2} - \frac{1}{2}  |V|^2, \qquad \psi^\ep = \ep m^2 \Phi_{s}.
\ee
We will refer to \eqref{eqs:sys-e}, under these conditions, as the {\em $\epsilon$-system}\/. When $ \epsilon = 0$, we recover our original system, as discussed at the beginning of Section~\ref{s.regularization}. 

In the next lemma we show that the a priori energy estimate proven for the original system carries over, with  minor modifications, to the regularized setting:

\begin{lemma}\label{lem:apriori-e} Let $\ep>0$ be fixed and let $(z(\cdot, t),V(\cdot, t))\in\t \Ocal^k_\beta(m)$ be a sufficiently smooth solution of the $\epsilon$-system. Then the energy functional
\be\label{eq:energy-extended}
\t E_{k, \beta}(t)^2 := \|\log|z_\al|\|^2_{H^{k+1}_{\beta + k}} +  \|\theta\|^2_{H^{k+1}_{\beta + k}} + \|\Phi_s\|^2_{H^{k+1}_{\beta + k}} +  \|V_1\|^2_{H^{k}_{\beta + k - 2}}  < \infty,
\ee 
which is equivalent to the $\t \Bcal^k_\beta(m) $-norm of $(z(\cdot, t),V(\cdot, t))$, satisfies the a priori inequality
\be\label{eqs:tenergy}
\frac{d\t E_{k, \beta}(t)^p}{dt} \, \lesssim \,  \frac{1}{\epsilon} \exp  (C \t E_{k, \beta}(t)^p)
\ee
for some $p\in \Nbb$.
\end{lemma}

\begin{proof}
	As before, let us consider the case $k=2$, as the general case is similar. The results of Subsection~\ref{ss.PE}  carry over essentially word-for-word to the present situation, so we will mostly point out the  differences that arise. 
	
	First, as in Lemma \ref{lem:param}, when $\beta \in(- \frac12,\frac12)$ we can write 
$$
\theta_{ts} = H\varphi_{ss} + \Rcal(z, V),
$$
so \eqref{eq:th_ts} holds. In order to prove that $\Rcal(z, V)\in\Lcal^2_{2,\beta + 2}(m)$, we need to only use $V\in H^2_{\beta}(m)$ and Lemma~\ref{lem:BRcancel} (which, under current regularity assumptions, allows us to take two derivatives on the kernel of $(i\Ccal - H)D_s^2V^*$ directly, see the proof of Lemma ~\ref{lem:param}). Since we control an additional half-derivative of $\varphi_{s}$ (or, equivalently, of $\Phi_{s}$), we conclude $\theta_{t}\in H^{3}_{\beta + 2}(m)$. Since $V_s = z_s(\varphi_s + i \theta_t)$, we conclude that $V\in H^4_{\beta + 2}(m)$. In particular, the norms of the higher-order derivatives of~$V$ are controlled by some power of $\t E_{2,\beta}(t)$, and the energy functional $\t E_{2,\beta}$ is equivalent to the norm in the Banach space $\t\Bcal^k_\beta(m)$. 

In the case $\beta = \frac12$, we proceed analogously after replacing \eqref{eq:th_ts} by \eqref{eq:th_ts12} (with the corresponding term $\Rcal(z, V)$, as in \eqref{eq:R12}).   

On the other hand, by Proposition~\ref{prop:dirichletInt} and Lemma~\ref{lem:special_sol} we know that $\si \in H^{2}_{\beta}(m)$ is of the form~\eqref{eq:si-rep} on the entire allowed range of $\beta$'s. Here we are using that, by construction, $m^2\Phi_s\in\Lcal^3_{2,\beta}(m)$  (which is a subset of $ \Lcal^{3-\frac{1}{2}}_{2,\beta -\frac{1}{2}}(m)$ by Lemma~\ref{lem:halfDer} and the definition of the half-norm given in Section~\ref{ss.weightedSobolev}), so it does not contribute to the leading asymptotic behavior of~$\si$ (which is proportional to $|z-z_*|$) near the corner tip.

In fact, we can proceed exactly as in the proof of Lemma \ref{lem:si} to show that $\si = \Gcal(z)\psi^\ep + H^3_{\beta + 1}(m)$. Using the same strategy for $\Gcal(z) \psi^\ep$, when $\beta \in(- \frac12,\frac12)$ we infer that
\be\label{eq:psiep}
\aligned
\Gcal \psi^\ep &= \Gcal \psi^\ep - m^2 H\left(\frac{1}{m^2}\pa_s \psi^\ep \right) + \ep m^2 H\left(\frac{1}{m^2}\pa_s(m^2\Phi_s)\right), \quad \psi^\ep := \ep m^2\Phi_s \\
&= \ep m^2 H\pa_s\Phi_s + \Lcal^2_{2, \beta}(m)
\endaligned
\ee
Note that we have introduced suitable corrections to the Hilbert transform in order to ensure we have the desired behavior near the singular point. As $m^2 H\pa_s\Phi_s\in \Lcal^2_{2,\beta}(m)$, we conclude that
\be\label{eq:si-sp}
\si = |b_1|^2\, |z - z_*|\tan 2\nu + \epsilon m^2 H\pa_s\Phi_{s} + \Lcal^{2}_{2,\beta}(m).
\ee
As in the proofs of Lemmas \ref{lem:param} and \ref{lem:si}, when $\beta = \frac12$ we only need to replace the Hilbert transform by the operator $m^{-\lambda'} H m^{\lambda'}$ for some $\lambda'\in (0, 1)$ (see Equation~\eqref{E.cacambio}) to conclude 
\be\label{eq:si-sp12}
\si = |b_1|^2\, |z - z_*|\tan 2\nu + \epsilon m^{-\lambda' + 2} H(m^{\lambda'}\pa_s\Phi_{s}) + \Lcal^{2}_{2,\beta}(m).
\ee
Therefore $V_t\in H^2_\beta(m)$ and $U_t \in \Lcal^2_{2,\beta}(m)$, as before, for the entire range of weights $\beta\in(-\frac12,\frac12]$.

Finally,  consider the time derivative of $\varphi_s$:
\be\label{eq:varphist-e}
\varphi_{st} =- \varphi^2_s + \theta^2_t - \si\theta_s + \epsilon \pa_s^2(m^2 \Phi_{s}).
\ee
This implies that 
$$
\Phi_{st} =- \varphi^2_s + \theta_t (\theta_t - \Re(ib_1z_s^2)) - \si\theta_s + \epsilon \pa_s^2(m^2 \Phi_{s}) - \Re(b_{1t}z_s^2). 
$$
The only term which is nontrivial to analyze is the highest-order one, 
$$
\aligned
\frac{1}{2}\frac{d}{dt}\|\pa_s^3 \Phi_s\|^2_{2, \beta + 2} = -\int m^{2(\beta + 2)} &\pa_s^3 \Phi_s \pa^{3}_s\left(\si \theta_s\right) ds +  \epsilon\int m^{2(\beta + 2)} \pa_s^3 \Phi_s \pa^{5}_s(m^2 \Phi_{s}) ds \\
&+  \text{`bounded'}
\endaligned
$$  
where `bounded' stands for terms controlled by powers of $\t E_{2,\beta}(t)$.

To analyze the remaining terms, let us denote by  $I_1$ and $I_2$ the integrals on the right hand side of the last line. Integrating $I_2$ by parts yields
$$
I_2 = - \ep \int m^{2(\beta + 3)} |\pa_s^4 \Phi_s|^2ds + \ep \cdot \text{`bounded'}.
$$
For $I_1$, on the other hand, we have
\[
\aligned
I_1 &= \int \pa_s(m^{2(\beta + 2)} \pa_s^3 \Phi_s) \pa^{2}_s\left(\si\theta_s\right) ds \\
&=\int \left(m^{2(\beta + 2)}\pa_s^4 \Phi_s + 2(\beta + 2) m^{2\beta + 3}\pa_s m \,\pa_s^3 \Phi_s \right)\pa^{2}_s\left(\si\theta_s\right)ds.
\endaligned
\]
The integral over the second summand is bounded in terms of $\t E_{2,\beta}(t)$, while for the first one we have the estimate 
\be\label{eq:1/eps}
\int m^{2(\beta + 2)} \pa_s^4 \Phi_s \pa^{2}_s\left(\si \theta_s \right) ds \, \leq \, 
\epsilon \int m^{2(\beta + 3)}| \pa_s^4 \Phi_s |^2 ds 
 + \frac{4}{\epsilon} \int m^{2(\beta + 1)} |\pa^{2}_s(\si \theta_s)|^2 ds + \text{`bounded'}.
\ee
Thus the most singular terms from $I_1$ and $I_2$ cancel out, and the lemma follows.
\end{proof}

\begin{remark}
One should take note of the term in~\eqref{eq:1/eps} proportional to $1/\ep$, which does not play a major role here but will need to be dealt with carefully in order to pass to the limit $\ep=0$.
\end{remark}

In the next step, we introduce a second regularization, which depends on an independent small parameter~$\de$. The basic idea is to apply a variable-step convolution operator $A^\de$ to the right-hand side of the evolution equation for~$V$. Note that an ordinary (fixed-step) convolution operator would cause major trouble due to the presence of weights and by the special role played by the corner point.

So we take $\phi$ as a standard mollifier (i.e., a smooth, non-negative, even function $\phi:\Rbb \rightarrow \Rbb$ with $\mathrm{supp}(\phi) \subset [-1,1]$ and such that $\int_\Rbb \phi = 1$). Let  $\eta\in\Ccont^\infty(\overline{I})$ be strictly positive on $I:=(-\pi,\pi)$ with first order zeros at~$\pm\pi$ (i.e., $\eta(\pm\pi) = 0$ but $\eta'(\pm\pi) \neq 0$). This condition ensures that $\eta \approx m$. For sufficiently small $\de >0$ and $x\in I$, we define 
\be\label{E.defphide}
\phi_{\de \eta(x)}(y) := \frac{1}{\de \eta(x)}\phi\left(\frac{y}{\de \eta(x)}\right)
\ee
and set, for $x\in I$,
\be\label{E.defBde}
B_{\de}f(x) := (\phi_{\de \eta(x)} \ast f) (x) := \int_I \phi_{\de \eta(x)}(x - y)f(y)dy.
\ee
The adjoint is 
\be\label{E.defB*de}
B_{\de }^*f(y) :=  \int_I \phi_{\de \eta(x)}(x - y) f(x)dx,\quad y\in I.
\ee
Finally, the convolution operator $A^\de$ is defined as the composition of $B_\de$ and $B_\de^*$:
\be\label{E.defAde}
A_\de f := B^{*}_\de  B_\de f.
\ee

As we had anticipated, these operators are useful to regularize functions that live in the weighted Lebesgue spaces $\Lcal^k_{2,\ga}(I)$, since the weight function $m$ is equivalent to the distance to the boundary $\pa  I:=\{-\pi,\pi\}$. Note that the interval of integration always has positive distance to $\pa I$, which is crucial as we are ultimately interested in functions having non-integrable singularities at the boundary of $I$. These operators have the smoothing effect in $I$, they respect growth rates at~$z_*$, and they approximate the original function in $\Lcal^k_{2,\ga}(I)$ as $\de\to0^+$.  Note that taking a derivative on $\phi$ results in a factor of size $O((\de \eta)^{-1})$, which is why we need to ensure $\eta \approx m$ to be able to handle the corresponding error terms. One should observe that these operators do not generally fix boundary values of the original function (they would if  $\eta$ vanished to higher order at $\pm\pi$). Technical results about these operators are presented in Subection~\ref{ss.convolution} in  Appendix~\ref{A.A}. 

A technical point is that here it is more convenient to consider $D_s U^*$ (or, more precisely, the real part) as a basic unknown instead of the real part of~$U_1$, and then define $U^*$ by integration. However, a necessary condition for a function~$W$ to coincide with $ D_s U^*$ is   
\be\label{eq:W-int}
\int W z_s \, ds = 0.
\ee 
In addition, the real and imaginary parts of $W$ must be conjugate functions because $\Ccal(D_sU^*)=D_sU^*$, and those of $W_t - z_t D_s W$ too. 

Let us now prescribe the new regularized set of evolution equations for 
\[
(z_*, \nu_+, \nu_-, \Theta, \log|z_\al|,b_0, b_1, W).
\]
First, the evolution of the curve is essentially unchanged:
\begin{gather}
\label{eqs:sys-z}
\frac{dz_*}{dt} = b_0^*, \qquad \frac{d\nu_+}{dt} = \Re\left(i b_1e^{2i\nu_+}\right), \qquad \frac{d\nu_-}{dt} = \Re\left(ib_1 e^{2i\nu_-}\right),\\
 (\log |z_\al| - i\theta)_t = z_s^2 (b_1 + W).\notag
\end{gather}

The curve is then defined as 
\[
z(\al, t) = z_*(t) + \int_{-\pi}^\al e^{i\theta(\al', t)} \, |z_\al(\al',t)|\, d{\al'},
\]
just as in~\eqref{z}. This defines a closed curve provided
\be\label{eq:z-int}
\int_\Ga e^{i\theta} \, ds = 0.
\ee
This integral is not necessarily zero, but obviously
$$
\frac{d}{dt}\int_\Ga e^{i\theta} \, ds = 0 \quad \Leftrightarrow \quad \int_\Ga z_s W ds = 0,
$$
in which case the integral is independent of~$t$. As we can always add a time-independent function to the definition of $z$ without changing the evolution equations, we can therefore assume that the integral is actually zero.

We take the evolution of the velocity as
\begin{gather}\label{eqs:sys-V}
\frac{db_0}{dt} = ig, \qquad \frac{db_1}{dt} = \frac{1}{\cos2\nu}|b_1|^2e^{-i(\nu_+ + \nu_-)}   ,\\[1mm]
\label{eqs:sys-ed}W_t = -\frac{1}{z_s^2}|b_1 + W|^2 - \left(b_{1t} + D_s\left(\frac{i\si_0}{z_s}\right)\right) + D_s A_\de\left(\frac{1}{z_s}\left(-i\t\si  + \epsilon \pa_s(m^2 \Phi_s)\right)  \right) - \frac{i\Xi}{z_s},
\end{gather}
where we recall that $\Phi_s  := \Re(z_s^2 W)$. Furthermore, we have set 
$$
\t \si := \si - \si_0, \qquad \si_0 := |b_1|^2 |z - z_*|\tan 2\nu,
$$ 
and $\Xi$ is another correction term that we will define (and motivate) below.

We now define $U^*$ and $V^*$ in terms of $W$ as 
$$
V^* := b_0 + b_1(z - z_*) + U^*, \qquad U^*(\al, t) := \int_{-\pi}^\pi  W(\al', t) \, z_\al(\al',t)\,d\al'.
$$
Note that $U^* (\pi,t)=U^*(-\pi,t)= 0$ provided \eqref{eq:W-int} is satisfied. Again, the integral~\eqref{eq:W-int} is not necessarily zero, but we do have
$$
\frac{d}{dt}\int_\Ga z_s W ds = 0
$$
as long as $\int \Xi \, ds = 0$. As $\Xi$ will essentially arise as the normal derivative of a harmonic function, this condition will always hold by construction. In particular, we have 
$$
 D_s U^* = W, \qquad D_sV^* = b_1 + W,
$$ 
as required. One should note that Equation \eqref{eqs:sys-ed} is essentially the evolution for $(D_s V^*)_t$, where~$V_t$ is given by the second equation in \eqref{eqs:sys-e} with a regularized right-hand side and a correction term~$\Xi$.
Now that we have specified $V^*$ in terms of our ``basic'' unknowns, we can define $\si \equiv \si(z, V, \epsilon)$ as in the $\ep$-system, that is, by Equations~\eqref{eq:si-eps}-\eqref{eq:psi-eps}.

An issue that we must still take care of is the following. As a consequence of the presence of the regularization operator $A_\de$,  $W_t - z_t D_s W$ is not the boundary trace of a holomorphic function anymore, in contrast with~\eqref{E.Ccalmess}. This is why we introduce the aforementioned correction term $ \Xi \equiv \Xi(z, V, \epsilon, \delta)$. 

Specifically, let us set 
$$
\aligned
u := \Re\left\{- VD_sV^* - \left(b_{1t}(z - z_*) + \frac{i\si_0}{z_s}\right) + A_\de\left(\frac{1}{z_s}\left(-i\t\si + \epsilon (m^2 \Phi_s)_s\right)  \right) \right\},
\endaligned
$$
so that
\[
\pa_s u = \Re\left\{-z_s\left(\frac{1}{z_s^2}|D_s V^*|^2  + V D_s^2 V^* \right) -  \left(b_{1t}z_s + \pa_s\left(\frac{i\si_0}{z_s}\right)\right) + \pa_s A_\de\left(\frac{1}{z_s}\left(-i\t\si + \epsilon \pa_s(m^2 \Phi_s)\right)  \right) \right\} ,
\]
and define
\be\label{Xi}
\Xi := \Gcal u + \Im\left\{-\pa_s\left(VD_sV^* \right) - \left(b_{1t}z_s + \pa_s\left(\frac{i\si_0}{z_s}\right)\right) + \pa_s A_\de\left(\frac{1}{z_s}(-i\t \si + \epsilon (m^2 \Phi_s)_s)  \right) \right\}.
\ee
One should note that $u = - \Re(b_0^*D_sV^*) + \Lcal^{k + 1}_{2,\beta + k - 1}(m)$, since $A_\de$ and its derivatives preserve the growth rates at the singular point, as we show in~Subsection~\ref{ss.convolution}. Therefore, Proposition~\ref{prop:dirichletInt} ensures that 
\[
\Gcal u = \Im (z_s b_0^* D_s^2U^*) + \Lcal^{k -\frac{1}{2}}_{2,\beta + k - \frac{3}{2}}(m),
\] 
so in particular $\Xi\in \Lcal^{k - 1}_{2,\beta + k - 2}(m)$. We can further show that $\Xi\in \Lcal^{k + 1}_{2,\beta + k}(m)$. For this, we proceed as in the proof of Lemma~\ref{lem:si}; however, as we now control an additional half-derivative of $|z_\al|$, we can put two additional derivatives on the kernel of $i\Ccal - H$, cf.~Lemma~\ref{lem:BRcancel}. 

Observe that, when $\de = 0$, by the choice of $\si$ and by the uniqueness of solutions to the Laplace equation, we have $\Xi = 0$. By construction, we also have $\Theta_t\in \Lcal^{k + 1}_{2,\beta + k}(m)$ and $W_t \in \Lcal^{k + 1}_{2,\beta + k}(m)$.

We will refer to \eqref{eqs:sys-z}--\eqref{eqs:sys-ed} as the {\em $(\epsilon, \delta)$-system}\/. 
It defines an ODE on the open set $\t \Ocal^k_\beta(m) \subseteq \t \Bcal^k_\beta(m)$. Furthermore, it is easy to see that the right-hand of \eqref{eqs:sys-z}--\eqref{eqs:sys-ed} (the ``vector field'' defining the ODE on $\t \Ocal^k_\beta(m)$) is a Lipschitz function on this set. Perhaps the only point that is not completely elementary is the fact that the Dirichlet-to-Neumann operator is a Lipschitz function of~$z$; however, this follows from the results in Section~\ref{s.inverse} essentially as in the case of smooth curves (see e.g.~\cite{LannesBook}).

Therefore, by the abstract Picard theorem, for every $\de >0$ and every initial datum $(z^{\epsilon, \delta}_0, V^{\epsilon, \delta}_0)\in\t \Ocal^k_\beta(m)$, there exists some time $T^{\epsilon, \de}>0$ and a unique solution 
\[
(z^{\epsilon, \delta}, V^{\epsilon, \delta})\in \Ccont^1([0, T^{\epsilon, \de}], \, \t \Ocal^k_{\beta}(m))\]
of the $(\ep,\de)$-system \eqref{eqs:sys-z}-\eqref{eqs:sys-ed} with initial condition
\[
(z^{\epsilon, \delta}(0), V^{\epsilon, \delta}(0))=(z_0^{\epsilon, \delta}, V^{\epsilon, \delta}_0).
\]
Moreover, for any fixed $\ep, \de >0$, the solution can be extended until it leaves the open set $\t\Ocal^k_{\beta}(m)$. 

The key property we aim to show is that the time of existence can be chosen uniformly in $\de$:

\begin{lemma}\label{lem:Tepsilon}
Fix some $\epsilon>0$ and some initial datum $(z^{\epsilon}_0, V^{\epsilon}_0)\in\t \Ocal^k_\beta(m)$. Then there exists some $\de$-independent time $T^\epsilon>0$ such that the solution $(z^{\epsilon, \de}, V^{\epsilon, \de})$ to the $(\ep,\de)$-system with this initial datum satisfies
\[
(z^{\epsilon, \de}, V^{\epsilon, \de}) \in \Ccont^1([0,T^\epsilon], \t \Ocal^k_{\beta}(m))
\]
for all small enough $\de> 0$.
\end{lemma}

\begin{proof} 
As is well known, it is enough to prove that the estimate~\eqref{eqs:tenergy} for the $\t E_{k,\beta}(t)$-energy of~$(z^{\epsilon, \de}, V^{\epsilon, \de})$ holds uniformly for all small enough~$\de> 0$. Recall that we have shown $\t E_{k,\beta}(t)$ is equivalent to the $\t \Bcal^k_\beta(m)$-norm of $(z^{\epsilon, \de}(\cdot, t), V^{\epsilon, \de}(\cdot, t))$. This is because, if \eqref{eqs:tenergy} holds uniformly for sufficiently small $\de>0$, Picard's theorem yields a lower bound for the maximal time of existence of~$(z^{\epsilon, \de}, V^{\epsilon, \de})$ that  depends on $\epsilon$ and $\t E_{k,\beta}(0)$ but not on~$\de$. (This bound will not allow us to take the limit $\ep\to0^+$, however, as it scales like $1/\ep$.

Let us now prove an energy estimate for~$(z^{\epsilon, \de}, V^{\epsilon, \de})$ that is uniform in the small parameter~$\de>0$. For the ease of notation, in the rest of the proof we drop the $(\epsilon, \de)$ superscript everywhere. 

As before, for concreteness, let us stick to the case $k=2$, as the general case is analogous. A glance at the $(\ep,\de)$-system reveals that it is enough to estimate the terms with the highest-order derivatives of $\Phi_{s}$, as all the other terms are obviously bounded uniformly in~$\de$. So we only need to focus on this term.

We have
$$
\aligned
\Phi_{st} &= 2\theta_t^2 - \theta_t\Re(i b_1 z_s^2 ) + \Re(z_s^2 W_t)\\
&= \Re\left(z_s \pa_s A_\de\left(\frac{1}{z_s}(-i\t\si + \ep\pa_s(m^2 \Phi_{s}))  \right) - iz_s \Xi\right) + H^{3}_{\beta + 2}(m),
\endaligned
$$
where we henceforth write ``$+\, H^{3}_{\beta + 2}(m)$'', and similar expressions, for terms bounded uniformly in~$\de$ in these spaces. However, the most we can say about~$\Xi$, under our regularity assumptions and uniformly in~$\de$, is $\Xi\in \Lcal^{1}_{2,\beta}(m)$.

Let us present the proof in the case the $\beta \in(- \frac12,\frac12)$. In the case $\beta = \frac12$, the proof  is exactly the same, after modifying the operators $H$ and $\Ccal$ using the  $\lambda'$-power of the weight  as in Equation~\eqref{E.cacambio}. For instance, this simply amounts to replacing \eqref{eq:si-sp} by \eqref{eq:si-sp12} when estimating $\t \si$. The details are just as before.

The starting point of our uniform estimate is to write
$$
\aligned
\pa_s \left(\frac{1}{z_s}(-i\t\si + \ep\pa_s(m^2 \Phi_{s}))\right) &= \frac{1}{z_s}\left(-\t\si \theta_s -i\ep \theta_s (m^2\Phi_s)_s -i\t\si_s + \ep(m^2\Phi_{s})_{ss}\right)\\
&=:\frac{1}{z_s}\left(-i\ep(m^2H\Phi_{ss})_s + \ep(m^2\Phi_{ss})_{s} - \t\si \theta_s + \ep g_1 + g_2\right)  \\
&=:\frac{1}{z_s}\left(\ep \pa_s(m^2(I - iH)\Phi_{ss})  - \t\si \theta_s + \ep g_1 + g_2\right).
\endaligned
$$
Here we have used \eqref{eq:si-eps}--\eqref{eq:psi-eps}, $I$ is the identity operator, and from now on we denote by $g_1\in\Lcal^2_{2,\beta + 1}(m)$ and $g_2 \in \Lcal^{3}_{2,\beta + 2}(m)$ complex-valued functions, which may vary from line to line, that are bounded in these spaces uniformly in~$\de$. In our argument, they play the role of constants.
At this point, it is important to note that $\t \si \theta_s$ has the same regularity properties as $g_1$, but it does not carry a factor of $\ep$. 

For ease of notation, we temporarily set
$$
G := m^2 (I - iH)\Phi_{ss},
$$
where $G\in \Lcal^2_{2,\beta}(m)$ by Lemma \ref{lem:trHilbert}. Thus, we can write 
$$
\aligned
\pa_\al A_\de\left(\frac{1}{z_s}(-i\t\si + \ep\pa_s(m^2 \Phi_{s}))\right) =\ep A_\de \left(\frac{|z_\al|}{z_s}\pa_s G\right) +  \frac{\ep}{z_s} \t K_\de B_\de G - \frac{|z_\al|}{z_s} A_\de(\t\si \theta_s)  + \ep B_{\de}^*g_1 + g_2.
\endaligned
$$
where we have written $(B_\de^* f)' =: B_\de^* f' + \t K_\de f$ as in the proof of Lemma~\ref{lem:B_delta} and we have used that we control an additional derivative whenever we have a commutator with any of the above convolution operators. This follows by applying Lemma~\ref{lem:BdeH} to the commutators with  $|z_\al|/z_s$.  The derivative of this function is in $\Lcal^1_{2,\beta + 1}(m)$, aso it is bounded pointwise by $O(m^{\lambda_\beta - 1})$ by~Lemma \ref{lem:sobolev}. 

Using this formula, let us now write
$$
\aligned
-z_s &\left(D_s(VD_sV^*) + b_{1t} + D_s\left(\frac{i\si_0}{z_s}\right)\right) + \pa_s A_\de\left(\frac{1}{z_s}(-i\t\si + \ep\pa_s(m^2 \Phi_{s}))\right) \\
&= -z_s V D_s^2 U^* + \frac{\ep}{|z_\al|} A_\de\left(\frac{|z_\al|}{z_s}\pa_s G\right) - \frac{1}{z_s } A_\de(\t\si \theta_s) +  \frac{\ep}{z_\al} \t K_\de B_\de G  + \frac{\ep}{|z_\al|} B_{\de}^* g_1 + g_2.
\endaligned
$$
We now proceed as in the proof of Lemma \ref{lem:si}. That is, we add and subtract the ``corrected'' Hilbert transform $m H\left(m^{-1}  \pa_s u\right)$ (note the factors of~$m$ here) from Equation~\eqref{Xi} (see also~\eqref{eq:psiep}), and then we repeatedly use Lemmas~\ref{lem:BRcancel}, \ref{lem:trHilbert}, \ref{lem:derFphi} and~\ref{lem:BdeH} to conclude, after a tedious but straightforward calculation, that
$$
\aligned
\Xi \, =\, &  \ep\,\frac{m}{|z_s|}\Im A^\de\left[(I + iH) \left(\frac{|z_\al|}{z_s}\frac{1}{m}\pa_s G\right) \right] + \ep  \Im\left[\frac{1}{z_\al}  (I + iH) \left(\t K_\de B_\de G\right)\right] + \\
& + z_{2s} A_\de(\t\si \theta_s) - z_{1s} m A_\de H\left(\frac{1}{m}\t\si \theta_s\right) + \frac{\ep}{|z_\al|}B_{\de}^* g_1 + H^3_{\beta + 2}(m)\\
=\, &\ep\Im\left[\frac{1}{z_\al}  (I + iH) \left(\t K_\de B_\de G\right)\right] + z_{2s}  A_\de(\t\si \theta_s) - z_{1s}m A_\de H\left(\frac{1}{m}\t\si \theta_s\right) + \frac{\ep}{|z_\al|}  B_{\de}^* g_1 + H^3_{\beta + 2}(m).
\endaligned
$$ 
To do this, the only nontrivial observations one must use in the argument are that one can put up to two derivatives on the kernel of the commutator of $H$ with functions in $H^3_{\beta + 1}(m)$, and  that $H$ and the multiplication by $m$ and by~$m^{-1}$ essentially commute with $A_\de$ by Lemma~\ref{lem:BdeH}. The other manipulations are by now standard.

Putting everything together, we obtain
\be\label{eq:PHIst}
\aligned
	\Phi_{st} =& \, \ep\Re\left(\frac{z_s}{|z_\al|} A_\de\left(\frac{|z_\al|}{z_s}\pa_s G\right)\right) +\left[-z_{1s}^2A_\de (\t\si \theta_s) + z_{2s} z_{1s}m A_\de H\left(\frac{1}{m}\t\si \theta_s\right)\right]\\
	&+\frac{\ep}{|z_\al|}\left[\Re(\t K_\de  B_\de G) + z_{2s}\Im\left(\frac{1}{z_s}  (I - iH) \left(\t K_\de B_\de G\right)\right)\right]  + \ep \Re (f  B_{\de}^*g_1)+ H^3_{\beta + 2}(m)
\endaligned
\ee
for some $f\in H^3_{\beta + 2}(m)$. The first term in this formula is the most singular one, as we only control its $H^1_\beta(m)$-norm uniformly in~$\de$. We thus require to gain uniform control over two additional derivatives. The remaining terms (except for the last one, which is smoother) are uniformly in $H^2_{\beta + 1}(m)$, so we  only need to control one additional derivative uniformly.  

We are now ready to estimate the time derivative of $\|\pa_s^3 \Phi_s\|^2_{2, \beta + 2}$:
$$
\aligned
\frac{1}{2}\frac{d}{dt}\|\pa_s^3 \Phi_s\|^2_{2, \beta + 2} &= \int m^{2(\beta + 2)} \pa_s^3 \Phi_s \, (\pa_s^3\Phi_{st} - 3\Phi_s \pa_s^2\Phi_{st}) \,ds + \text{ `bounded'}\\
&=: I_1 + I_2 + \text{`bounded'}
\endaligned
$$
Here and in what follows, `bounded' stands for functions which can be bounded in terms of $\t E_{2,\beta}(t)$ uniformly in $\de$.

Let us consider the integral $I_2$ first. By the above formulas, we have
$$
\aligned
\pa_s^2\Phi_{st} =& \, \ep\Re\left(\frac{z_s}{|z_\al|} A_\de\left(\frac{|z_\al|}{z_s}\pa_s G\right)\right) + \Lcal_{2,\beta + 1}(m) \\
&= \ep \, \pa_s\Re\left(\frac{z_s}{|z_\al|^2} A_\de\left(\frac{|z_\al|^2}{z_s} \pa^2_s G\right)\right) + \Lcal_{2,\beta + 1}(m)\\
&=\ep \, \pa_s\Re\left(  A_\de \pa^2_s G  +  \frac{z_s}{|z_\al|^2}\left[A_\de,\frac{|z_\al|^2}{z_s}\right]  \pa_s^2 G \right) + \Lcal_{2,\beta + 1}(m)
\endaligned
$$
where in the second line we have used that if a derivative hits any factor but $G$, one gets a term bounded in $\Lcal_{2,\beta + 1}(m)$ by the results in Subsection~\ref{ss.convolution}. Recalling that $\pa_s^2G\in\Lcal_{2,\beta}(m)$, and that the derivative of the commutator term belongs to $\Lcal_{2,\beta + 1}(m)$, and using the formula as above $(B_\de^* f)' = B_\de^* f' + \t K_\de f$, we then conclude
$$
\pa_s^2\Phi_{st} = \frac{\ep}{|z_\al|}\,B_\de^*\left(\pa_\al B_\de (\pa_s^3\Phi_s)\right)+ \Lcal_{2,\beta + 2}(m).
$$ 
Therefore, integrating by parts in~$I_2$, we obtain
$$
\aligned
\frac{1}{3\ep}I_2 &= \int  \pa_\al B_\de\left(m^{2(\beta + 2)} \Phi_s \pa_s^3 \Phi_s  \right) B_\de\left(m^2\pa_s^3\Phi_s\right) d\al + \text{ `bounded'}\\
&=\int \frac{1}{2} m^{2(\beta + 1)}\Phi_s \pa_\al |B_\de(m^2\pa_s^3\Phi_s)|^2 +  \pa_\al [B_\de, m^{2(\beta + 1)} \Phi_s](m^2\pa_s^3 \Phi_s) B_\de\left(m^2\pa_s^3\Phi_s\right) d\al +  \text{ `bounded'}.
\endaligned
$$
As the commutator term belongs to $\Lcal_{2,-\beta}(m)$ by Lemma \ref{lem:BdeH}, and integration by parts shows that the first integral is also bounded in terms of $\t E_{2,\beta}(t)$, we conclude that $I_2$ is uniformly bounded too.

The analysis of the integral $I_1$ is more involved. We split it as
$$
I_1 = \ep J_1 + J_2 + \ep J_3 + \ep J_4 + \text{`bounded'},
$$
where the summands in this decomposition are directly given by the summands in~\eqref{eq:PHIst}. In particular, $J_1$ corresponds to the most singular part of $I_1$, while $J_2, J_3$ and $J_4$ are basically all of the same order. 

We claim there exists some $n\in\Nbb$ such that
\be\label{eq:I1}
I_1 \,\leq\, \ep\left(-1 + \frac{n}{N}\right)\int m^{2(\beta + 2)}  |\pa_\al B_\de\left( m^2\pa_s^3 \Phi_s  \right)|^2 d\al + \frac{cN}{\epsilon}  \int  m^{2(\beta + 2)} |\pa_s^2(\t\si\theta_s)|^2 d\al + \text{`bounded'},
\ee
where $N>n$ is an integer which can be chosen arbitrarily large (the only price to pay is to make the bounded term larger).

To see this, we consider $J_2$ first. This is where the term proportional to $1/\ep$ (which we will need to take care of appropriately to pass to the limit $\ep=0$ later on) comes from. We have
$$
\aligned
J_2=& \int m^{2(\beta + 2)} \pa_s^3 \Phi_s  \left(- z_{1s}^2\pa^{3}_s A_\de(\t\si\theta_s) + z_{1s}z_{2s}m \pa_s^3 A_\de H\left(\frac{1}{m}\t\si\theta_s\right) \right)ds + \text{`bounded'}  \\
=& \int m^{2(\beta + 2)} \pa_s^3 \Phi_s  \left(-z_{1s}^2 B^*_\de \pa_\al B^\de \pa_s^2(\t\si\theta_s) + z_{1s}z_{2s} m B_\de^* \pa_\al B_\de\pa_s^2 H\left(\frac{1}{m}\t\si\theta_s\right) \right)d\al + \text{`bounded'}  \\
=& \, J_{2,1} + J_{2,2} + \text{`bounded'},
\endaligned
$$
where we have used that we can interchange $\pa_s$-derivatives with $B_\de$ and $B_\de^*$ because the error terms and the commutators provide control over an additional derivative by Lemma~\ref{lem:BdeH}). 

As the analysis of both terms is completely analogous, let us just consider $J_{2,2}$. As 
\[
\pa_s^2 H (\t\si\theta_s) = \frac{1}{m} H(m\pa_s^2(\t\si\theta_s)) + \Lcal^1_{2,\beta + 2}(m)
\]
by Lemmas \ref{lem:trHilbert} and \ref{lem:derFphi}, we can write
$$
\aligned
J_{2,2}=& -\int \pa_\al B_{\de}\left(m^{2(\beta + 2)}z_{1s}z_{2s}\pa_s^3 \Phi_s \right)B_\de H \left(m\pa_s^2\left(\frac{
1}{m}\t\si\theta_s\right)\right)d\al + \text{`bounded'}\\
=& \int m^{2(\beta + 2)}z_{1s}z_{2s}\pa_s B_{\de} \left(m^2\pa_s^3 \Phi_s \right)B_\de H (\pa_s^2(\t\si\theta_s)) ds + \text{`bounded'}.
\endaligned
$$ 
As in the proof of Lemma \ref{lem:apriori-e}, we now use the elementary inequality 
\be\label{E.abep}
ab \leq  \frac\ep N a^2 + \frac{N}{4\ep }b^2
\ee
to write
$$
|J_{2}| \,\leq \, \frac{2\ep}{N}\int  m^{2(\beta + 2)} \left|\pa_s B_{\de} \left(m^2\pa_s^3 \Phi_s \right)\right|^2 ds + \frac{cN}{2\epsilon}  \int  m^{2(\beta + 2)} |\pa_s^2(\t\si\theta_s)|^2 d\al + \text{`bounded'},
$$
where $N$ is a sufficiently large integer and $c$ is a numerical constant that depends on the operator norms of $H$ and $B_\de$. 

To estimate $J_3$ and $J_4$, we proceed similarly. However, as they carry a factor of~$\ep$, no terms of order $1/\ep$ will appear. More precisely, one can show 
$$
|J_3 + J_4| \, \leq \, \frac{1}{N} \int  m^{2(\beta + 2)} \left|\pa_\al B_{\de} \left(m^2\pa_s^3 \Phi_s \right)\right|^2 d\al + \text{`bounded'}.
$$ 
As the estimates for all the terms are very similar, we will only discuss how to control the term (let us call it $J_{3,1}$) involving the Hilbert transform  in $J_3$. Arguing as before, we have
$$
\aligned
J_{3,1}=& \Re\int m^{2(\beta + 2)} \pa_s^3 \Phi_s \, \pa_s^3\left(\frac{z_{2s}}{z_\al}H\t K_\de B_\de G\right)ds + \text{`bounded'}  \\
=& \,\Re\int H\left(\frac{z_{2s}}{z_\al} m^{2\beta + 3} \pa_s^3\Phi_s\right)  m\pa_\al \t K_\de B_\de \pa_s^2 G d\al  + \text{`bounded'}  \\
=& \,\Re\int \t K^*_\de \left(m H\left(\frac{z_{2s}}{z_\al} m^{2\beta + 3} \pa_s^3\Phi_s\right)\right)\pa_\al B_\de \pa_s^2 G d\al  + \text{`bounded'},
\endaligned
$$
where $\t K^*_\de$ is the adjoint of $\t K_\de $. Since $\pa_s^2 G =  (I - iH)(m^2\pa_s^3\Phi_s) + \Lcal^1_{2,\beta + 1}(m)$, we can write 
$$
\aligned
\pa_\al B_\de \pa_s^2 G &= \pa_\al(I - i H) B_\de (m^2\pa_s^3\Phi_s) + \Lcal_{2,\beta + 1}(m)\\
&= \frac{1}{m}(I - i H) \left(m \pa_\al B_\de (m^2\pa_s^3\Phi_s)\right)
\endaligned
$$
and the uniform estimate follows just as above.

Finally, we consider $J_1$, which is the most singular part of $\pa_s^3\Phi_{st}$. We claim there exists some $n_1$, independent of~$N$, such that
$$
J_1 \,\leq\, \left(-1 + \frac{n_1}{N}\right) \int m^{2(\beta + 1)}  |\pa_\al B_\de\left( m^2\pa_s^3 \Phi_s  \right)|^2 d\al +  \text{`bounded'}.
$$
Indeed, we have 
$$
\aligned
J_1 =&\,\int m^{2(\beta + 2)} \pa_s^3 \Phi_s \pa_s^3 \Re\left(\frac{z_s}{|z_\al|} A_\de\left(\frac{|z_\al|}{z_s}\pa_s G\right)\right) ds \\
=&\  \Re\int m^{2(\beta + 2)} \pa_s^3 \Phi_s \,\frac{z_s}{|z_\al|^4}  B_\de^*\left(\pa^2_\al B_\de\left(\frac{|z_\al|^2}{z_s}\pa_s^2G\right)\right) ds + \frac{1}{N} \|\pa_\al B_{\de} (m^2\pa_s^3 \Phi_s)\|^2_{2,\beta + 2} + \text{`bounded'}.
\endaligned
$$
Here, the terms controlled by~$\|\pa_\al B_{\de} (m^2\pa_s^3 \Phi_s)\|^2_{2,\beta + 2}$ arise whenever a derivative hits a term other than $\pa_s G$ or from commutators with convolution operators, and the factor $1/N$ is introduced by means of the elementary inequality~\eqref{E.abep} as before (at the expense of making the bounded term larger with~$N$, but still independent of~$\de$). 

Denoting the remaining integral by $J_{1,1}$, we simply keep the negative term that arises and estimate the remaining terms as above (using, in particular, ~\eqref{E.abep}) to obtain
$$
\aligned
J_{1,1} = &\, -\Re\int  \pa_\al B_\de\left(m^{2(\beta + 2)} \pa_s^3 \Phi_s \frac{z_s}{|z_\al|^3} \right)\pa_\al B_\de\left(\frac{|z_\al|^2}{z_s}\pa_s^2 G\right) d\al  \\
=& \left(-1 + \frac{n_1}{N}\right) \int m^{2(\beta + 2)}  |\pa_\al B_\de\left( m^2\pa_s^3 \Phi_s  \right)|^2 d\al +  \text{`bounded'} 
\endaligned
$$
for some $n_1$. As $N$ can be chosen to be as large as necessary, we obtain the bound for $J_1$. The estimate~\eqref{eq:I1} follows, and thus the lemma.
\end{proof}


Armed with the uniform bound proven in Lemma~\ref{lem:Tepsilon}, standard arguments provide the following local wellposedness result for the $\ep$-system:

\begin{corollary}\label{C.Tepsilon}
	For any fixed $\ep>0$, the $\ep$-system is locally wellposed on~$\t \Ocal^k_\beta(m)$. In particular, for any initial datum $(z^{\epsilon}_0, V^{\epsilon}_0)\in\t \Ocal^k_\beta(m)$, there exists some $T^\epsilon>0$ such that there is a unique solution $(z^{\epsilon}, V^{\epsilon})\in \Ccont^1([0,T^\epsilon], \t \Ocal^k_{\beta}(m))$ to the $\ep$-system with this initial datum.
\end{corollary}

\subsection{Proof of Theorem~\ref{thm:main}}
\label{ss.existence}

Our goal in this subsection is to show how Corollary~\ref{C.Tepsilon} can be used to prove the local wellposedness result for  the water wave system presented in Theorem~\ref{thm:main}. The key step of the proof is to obtain a lower bound for the existence time $T_\ep$ which does not depend on~$\ep$.

For this, we start with the natural generalization $E^\epsilon_{k,\beta}(t)$ of the energy functional $E_{k,\beta}(t)$, which we will see is well suited to handle the  $\epsilon$-system:   
$$
E^\epsilon_{k,\beta}(t)^2 := E_{k,\beta}(t)^2+ \epsilon \, \left(\|\pa_s^{k+1}\Phi_s\|^2_{2,\beta + k} + \|\pa_s^{k+1}|z_\al|\|^2_{2,\beta + k}\right).
$$
Here $k\geq2$ and in the definition of $E_{k,\beta}(t)$ (Equation~\eqref{eq:energy})  we replace $\varphi_s$  by 
\[
\Phi_s := z_s \cdot U_s = \varphi_s - \Re(b_1 z_s^2).
\]
This is dictated by the particular choice of regularization in \eqref{eqs:sys-e}. In addition, we choose to control $|z_\al|$ directly instead of $\log|z_\al|$. One should note both prescriptions for the energy (i.e., with $\varphi_s$ or with $\Phi_s$) define equivalent norms.

The main result in this subsection is the following. It is standard that the local wellposedness result for water waves with corners, Theorem~\ref{thm:main} follows from this result and the energy estimates in Section~\ref{s.apriori}.
An important aspect of the proof is how we regularize the initial datum using the operators~$B_\de$ (roughly speaking, by means of variable-step convolutions over scales of size $\ep^{1/2}$) to make sure it is in the more regular space $\t \Ocal^k_\beta(m)$ used for the $\ep$-system.

\begin{lemma}\label{lem:existence} 
Let us fix some initial datum
\[
(z^0,V^0)\equiv (z^0_*, \nu^0_+, \nu^0_-, \Theta^0, \log|z^0_\al|,b^0_0, b^0_1, U_1^0)\in \Ocal^k_\beta(m).
\]
Let $T_\ep$ be the maximal time of existence of the solution~$(z^\epsilon, V^\epsilon)$ to the $\epsilon$-system \eqref{eqs:sys-e} with initial data
\begin{multline*}
	(z^{0,\ep},V^{0,\ep})\equiv(z^{0,\epsilon}_*, \nu^{0,\epsilon}_+, \nu^{0,\epsilon}_-, \Theta^{0,\epsilon}, \log|z^{0,\epsilon}_\al|,b^{0,\epsilon}_0, b^{0,\epsilon}_1, U_1^{0,\epsilon}):=\\
	(z^0_*, \nu^0_+, \nu^0_-, B_{\epsilon^{1/4}}(\Theta^0),B_{\epsilon^{1/4}}( \log|z^0_\al|),b^0_0, b^0_1, B_{\epsilon^{1/4}}(U_1^0)),
\end{multline*}
which is in $\t \Ocal^k_\beta(m)$ for small enough~$\ep$. 
Then there exists a time $T>0$, independent of $\epsilon$, such that $T_\ep>T$ for all small enough $\ep>0$.
\end{lemma}

\begin{proof}
As before, we will take $k=2$ for concreteness.
First, a short computation shows that the regularization $(z^{0,\ep},V^{0,\ep})$ of the initial datum over scales of order $\ep^{1/2}$ is such that the $E^\ep_{2,\beta}$-norm of~$(z^{0,\ep},V^{0,\ep})$ is bounded by the $E_{2,\beta}$-norm of~$(z^0,V^0)$ up to an $\ep$-independent multiplicative constant. Therefore, the bound will follow if one proves an energy inequality, similar to~\eqref{eq:energy}, for the energy~$E^\ep_{2,\beta}$ and the $\ep$-system.

Hence, our goal is to show how $\frac{d}{dt}E^\ep_{2,\beta}(t)^2$ can be estimated in terms of $E^\ep_{2,\beta}(t)$. Most of the arguments required to establish this estimate go just as before, so for the sake of brevity we will focus on the aspects of the proof that are genuinely new.

We start by noting that the time derivative of $\epsilon\|\pa^{3}_s \Phi_s\|^2_{2, \beta + 3}$, can be estimated just as in the proof of Lemma~\ref{lem:apriori-e}, as the additional $\ep$ we have included in the energy cancels out the factor $1/\ep$ in the right hand side of~\eqref{eq:1/eps}. In fact, a minor modification of the proof of Lemma \ref{lem:apriori-e} shows that
$$
\ep\frac{d}{dt}\|\pa_s^3 \Phi_s\|^2_{2, \beta + 2} \, \leq \, -\ep^2 \, \|\pa_s^4 \Phi_4\|^2_{2,\beta + 3} + \text{`bounded'}
$$
where `bounded' means controlled in terms of $E^\ep_{k, \beta}(t)$.

To estimate $\frac{d}{dt}E_{2,\beta}(t)^2$, we shall proceed basically as in Lemma \ref{lem:apriori} to show that 
$$
\frac{d}{dt}\left(\|\sqrt{\si}\,\pa_s^{3}\theta\|^2_{\beta + 3/2} + \|\Lambda^{1/2}(m^{\lambda}\pa^2_s\Phi_s)\|^2_{2, \beta' + 1/2}\right) \, \leq \, \ep^2 \, \|\pa_s^4 \Phi_4\|^2_{2,\beta + 3} + \text{`bounded'}
$$ 
We recall the notation $\beta' := \beta - \lambda + 1$ that we used, where $\lambda$ was chosen in such a way that $\beta'\in\left(-\frac12, 0\right)$. If one goes over the proof of Lemma~\ref{lem:apriori} in this setting, most of the estimates carry over word-by-word. However, some care is needed because we replaced $\varphi_s$ by $\Phi_s$ and because there are two additional $\ep$-dependent terms (and one of them actually yields the new term above). Let us discuss the details.

First, by definition, we have $\Phi_s = \varphi_s + \Lcal^3_{2,\beta + 2}(m)$, so we can write  
$$
\pa_s^2\theta_{t} = m^{-\lambda} H(m^{\lambda}\pa^2_s\Phi_{s}) + \Lcal^2_{2,\beta + 2}(m),
$$ 
for $\beta\in(-\frac12, \frac12]$ by~\eqref{eq:th_tsslambda} and the proofs of Lemmas \ref{lem:apriori} and \ref{lem:apriori-e} (using again~\eqref{E.cacambio}). 

Concerning the $\ep$-dependent terms, by~\eqref{eq:varphist-e} we have 
$$
\varphi_{st} =- \varphi^2_s + \theta^2_t - \si\theta_s + \ep\pa_s^2(m^2 \Phi_s),
$$
and 
$$
\pa_s^2(\si \theta_s) = \si \pa_s^3\theta + \ep m^{1-\lambda}\theta_s H(m^{1 + \lambda}\pa^3_s\Phi_{s}) + \Lcal_{2,\beta + 2}^1(m).
$$
Here we have used that $\si$, as defined in \eqref{eq:si-eps}, can be written as
$$
\pa_s^2 \si = \ep m^{1-\lambda} H(m^{1 + \lambda}\pa^3_s\Phi_{s}) + \Lcal^1_{2,\beta + 1}(m)
$$
for $\beta \in (-\frac12,\frac12]$. This essentially follows by taking derivatives of \eqref{eq:si-sp} when $\beta \in(- \frac12,\frac12)$ (or of \eqref{eq:si-sp12} with  $\lambda': = \lambda - 1$ when $\beta = \frac12$), and then repeatedly applying Lemmas \ref{lem:trHilbert} and \ref{lem:derFphi}. Recall that in the proof of Lemma \ref{lem:apriori} we used that $\pa_s^2(\si \theta_s) = \si \pa_s^3\theta + \Lcal_{2,\beta + 2}^1(m)$, which relied on the control of $\pa_s^3\si$. 

Some care is needed to estimate the aforementioned $\ep$-dependent terms, so we will discuss all the nontrivial parts of the argument. The integral corresponding to the additional term that comes from $\si\theta_s$ is
$$
\aligned
I_2 &= \epsilon\int m^{2\beta' + 1}\Lambda^{1/2}(m^{\lambda}\pa^2_s\Phi_s)\Lambda^{1/2}\left(m\theta_s H(m^{1 +\lambda}\pa_s^3\Phi_{s})\right)d\al\\
&= -\epsilon\int m \theta_s \Lambda^{1/2}\left(m^{2\beta'  + 1}\Lambda^{1/2}(m^{\lambda}\pa^2_s\Phi_s)\right)  H(m^{1+\lambda}\pa_s^3\Phi_{s})d\al \\
& = \epsilon\int m^{2\beta' + 2}  \,\theta_s \pa_\al H(m^{\lambda}\pa^2_s\Phi_s)  H(m^{1+\lambda}\pa_s^3\Phi_{s}) d\al + \text{`bounded'}.
\endaligned
$$ 
Here we have used the identity $\Lambda = \pa_\al H$ and Lemma \ref{lem:comm_weight}. From this expression, one easily sees that $I_2$ is controlled by $E^\ep_{2,\beta}(t)$ just as before. 

For the last term in $\pa_s^2\varphi_{st}$, we write
\be\label{est1}
\aligned
m^{\lambda}\pa_s^4(m^2 \Phi_{s}) &= m^\lambda \pa_s^2 \left(m^2\pa_s^2\Phi_s + 4 \pa_s m^2 \pa_s\Phi_s \right) + \Lcal^1_{2,\beta' +1}(m)\\
&=  m^{2+\lambda}\pa_s^4\Phi_s + g m^{1+\lambda}\pa^3_s\Phi_s +\Lcal^1_{2,\beta' +1}(m),
\endaligned
\ee
where $g \in H^2_{\beta + 1}(m)$ and $m^{1+\lambda} \pa_s^3\Phi_s$ and $m^{2+\lambda} \pa_s^4\Phi_s$ belong to $\Lcal_{2,\beta'}(m)$ by our regularity assumptions. 

Let us consider the most singular term first. Integrating by parts, we find
$$
\aligned
I_1 = & \,\epsilon\int m^{2\beta'  + 1}\Lambda^{1/2}(m^{\lambda}\pa^2_s\Phi_s)\Lambda^{1/2}(m^{2+\lambda}\pa_s^4 \Phi_4) d\al \\
=& \,\epsilon\int \left( m^{2\beta' + 1}\Lambda (m^{\lambda}\pa^2_s\Phi_s) - [\Lambda^{1/2}, m^{2\beta' + 1}] \Lambda^{1/2}(m^{\lambda}\pa^2_s\Phi_s)\right)(m^{2+\lambda}\pa_s^4 \Phi_s)\, d\al.
\endaligned
$$
Applying the pointwise inequality $ab \leq \frac{1}{\ep}a^2 + \ep b^2$, we conclude
$$
|I_1|  \, \leq \, \ep^2 \, \|\pa_s^4 \Phi_4\|^2_{2,\beta + 3} + \text{`bounded'}
$$
Finally, using Lemmas~\ref{lem:trHilbert}, \ref{lem:comm_gb} and \ref{lem:comm_weight}, we can similarly obtain a bound for 
 the integral involving the term $g m^{1+\lambda}\pa^3_s\Phi_s$ in \eqref{est1} which is quadratic in the norm of $\pa^3_s\Phi_s$, and therefore bounded by $E^\ep_{2,\beta}(t)$ modulo a multiplicative constant. The lemma is then proven.
\end{proof}

In order to finish the proof of Theorem \ref{thm:main}, it will be convenient to define another, higher-order, energy functional for the $\epsilon$-system, namely   
$$
\aligned
\t E^\epsilon_{k,\beta}(t)^2 := E^\ep_{k,\beta}(t)^2+ \ep E_{k + 1, \beta}^\high(t)^2\ + \epsilon^2 \, \left(\|\pa_s^{k+2}\Phi_s\|^2_{2,\beta + k + 1} + \|\pa_s^{k+2}|z_\al|\|^2_{2,\beta + k + 1}\right).
\endaligned
$$
Written out in terms of powers of $\ep$, this reads as
$$
\aligned
\t E^\ep_{k, \beta}(t)^2 =& E_{k,\beta}(t)^2 + \ep \left(E_{k + 1, \beta}^\high(t)^2 +\|\pa_s^{k+1}\Phi_s\|^2_{2,\beta + k} + \|\pa_s^{k+1}|z_\al|\|^2_{2,\beta + k}\right) \\
&+ \epsilon^2 \, \left(\|\pa_s^{k+2}\Phi_s\|^2_{2,\beta + k + 1} + \|\pa_s^{k+2}|z_\al|\|^2_{2,\beta + k + 1}\right).
\endaligned
$$

With the same argument we used to prove Lemmas \ref{lem:apriori}, \ref{lem:apriori-e} and \ref{lem:existence}, it is not difficult to verify that $\t E^\epsilon_{k,\beta}(t)$ also satisfies an a priori energy estimate that is uniform in $\ep$. This is because the term proportional to $\ep$ appears only linearly in the estimate for the time-derivative of $E^\high_{k + 1, \beta}(t)^2$, as the factor multiplying this term in the energy estimate only depends on the `lower-order' $\ep$-independent energy $E_{k, \beta}^2(t)$). To prove this, one only (but crucially) needs to observe that one can actually put two derivatives on $i\Ccal - H$, thanks to the additional regularity on the parametrization of the interface and Lemma~\ref{lem:BRcancel}. 

For concreteness, let us assume that $k = 2$, as the case of general~$k$ is analogous. Given an initial datum $(z^{0}, V^{0}) \in \Bcal^2_\beta(m)$, let $(z^{0, \ep}, V^{0, \ep})$ be the regularization of the initial datum as defined in Lemma \ref{lem:existence}. Note that the $\t E^\ep_{2,\beta}$-energy of $(z^{0,\ep},V^{0,\ep})$ is bounded by the $E_{2,\beta}$-energy of~$(z^0,V^0)$ up to an $\ep$-independent multiplicative constant. Therefore, the a priori estimate for $\t E^\ep_{2,\beta}$ ensures that there exists some $T>0$ and a a constant $M$ (depending on the $E_{2,\beta}$-norm of~$(z^0,V^0)$) such that 
$$
\t E^\epsilon_{2,\beta}(t) \,\leq\, M
$$
for all $t\in[0, T)$, uniformly in $\ep$. 

We claim that $(z^\ep, V^\ep)$ is a Cauchy sequence in $\Bcal^2_\beta(m)$. To show this,  consider the $\Bcal^2_\beta(m)$-norm of the difference $(z^\ep, V^\ep) - (z^{\ep'}, V^{\ep'})$. As usual, we may replace $\|\pa_s^3\theta^\ep - \pa_{s'}^3\theta^{\ep'}\|_{2,\beta+ 2}$ by
$$
\left\|\sqrt{\si^\ep} \pa_s^{3}\theta^\ep - \sqrt{\si^{\ep'}} \pa_{s'}^{3}\theta^{\ep'} \right\|_{2, \beta + 3/2}.
$$  
The resulting quantity, which we denote by $E^d_{\ep,\ep'}(t)$, is basically the  $E_{k, \beta}$-energy of the difference (cf.~\eqref{def:energy-low}--\eqref{def:energy-high}). The fact that this is equivalent to $\|(z^{\ep}, V^{\ep}) - (z^{ \ep'}, V^{\ep'})\|_{\Bcal^2_\beta(m)}$ is straightforward (just recall the definition of $\si$ and the fact that the Dirichlet-to-Neumann operator is Lipschitz). 

With the estimates that we now have at hand, it is not hard to see that the energy of the difference satisfies the a priori  estimate
\be\label{eq:cauchy-eps}
\frac{d}{dt} E^d_{\ep,\ep'}(t)^2 \, \leq \, C_1 E^d_{\ep,\ep'}(t)^2 + \max\{\sqrt{\ep}, \sqrt{\ep'}\}C_2 E^d_{\ep,\ep'}(t),
\ee
where $C_1$ and $C_2$ depend only on $M$. The proof goes just as that of Lemma \ref{lem:apriori}. The only interesting new terms are the additional $\ep$-dependent terms which appear in $\si$ and $\varphi_{st}$ (see the proof of Lemma~\ref{lem:existence}), but these only contribute  to the term proportional to  $\max\{\sqrt{\ep}, \sqrt{\ep'}\}$ and are thus controlled by $\sqrt{\ep}\t E^\ep_{2, \beta}$. 

Therefore, Gronwall's inequality ensures that
\be\label{eq:cauchy-eps}
E^d_{\ep,\ep'}(t) \, \lesssim \, E^d_{\ep,\ep'}(0) + \max\{\sqrt{\ep}, \sqrt{\ep'}\}.
\ee
uniformly for small (but uniform) $t$. Since by construction $(z^{0, \ep}, V^{0, \ep})$ is a Cauchy sequence and $E^d_{\ep, \ep'}(0)\approx \|(z^{0, \ep}, V^{0, \ep}) - (z^{0, \ep'}, V^{0, \ep'})\|_{\Bcal^2_\beta(m)}$, this completes the proof that $(z^\ep,V^\ep)$ is a Cauchy sequence. In turn, it is standard that this yields the existence of a solution to the original system (i.e., the water wave equations, corresponding to $\ep=0$). Furthermore, considering now the energy of the difference of two solutions, an energy estimate completely analogous to~\eqref{eq:cauchy-eps} shows the uniqueness and stability of solutions to the original system.

\subsection{The case of periodic domains}
\label{ss.unbounded}

Following Remark~\ref{R.periodicOmega}, in this subsection we will show how a minor modification of the strategy used to prove Theorem~\ref{thm:main} when the water region~$\Om$ is bounded can handle the case where $\Om$ is a periodic (unbounded) domain. That is, in this section we shall assume that~$\Om(t)$ is the planar region that lies below the interface curve~$z(\al,t)$, under the assumption that $\al\mapsto z(\al,t)-\al$ is $2\pi$-periodic, and of course the boundary velocity is assumed to be $2\pi$-periodic too. 

For consistency with previous sections, let us parametrize the curve so that the corner point is
   $$
   z_*(t) := z(-\pi, t).
   $$
   We find it convenient to denote by~$\Om_\pi$ one period (i.e., a fundamental cell) of the domain~$\Om$, whose upper boundary one can assume to be $\Gamma_\pi:=\{z(\al,t): -\pi\leq \al\leq\pi\}$.

As we shall see, all the essential ingredients of the proof of Theorem~\ref{thm:main} carry over directly to this case, and the modifications work just as in the usual case of smooth 2D gravity water waves (i.e., without any corners). 

We start by recalling the periodic Cauchy integral formula. As the domain~$\Om$ is unbounded, the Cauchy integral formula is not directly applicable. However, the function
\be\label{e.cot}
\frac{1}{2}\cot\frac{z}{2}  = \frac{1}{z} + h(z)
\ee 
is $2\pi$-periodic, with $h(z)$ holomorphic on $\Re z\in[-\pi,\pi]$, and moreover 
$
\cot\frac{z}{2} \rightarrow  \pm i
$ 
exponentially fast as $z \rightarrow \mp i\infty$. Hence, it is easy to see  that the boundary trace of a bounded holomorphic function on~$\Om$ satisfies~\eqref{E.DalV*} with the Cauchy singular integral operator~$\Ccal$ replaced by its periodic counterpart $\Ccal_{\mathrm{per}}$, defined as 
   \begin{equation}\label{E.defCcalV*-periodic}
   	\Ccal_{\mathrm{per}}(z)  f(\al, t) := \frac{1}{2\pi i} \, \PV  \int_{-\pi}^\pi f(\al', t) \left[\cot\left(\frac{z(\al, t)-z(\al', t)}{2}\right) + i \right] \, z_{\al}(\al',t)\, d\al'.
   \end{equation}
      
   With the periodic Cauchy integral in hand, the results of Section \ref{ss.BR} quickly generalize to the case of periodic domains. First, since \eqref{e.cot} holds, Lemmas \ref{lem:BRbasic} and \ref{lem:BRcancel} follow immediately. On the other hand, in order to define the periodic counterparts $\Ccal_{\mathrm{per}(j)}$  of the `corrected' Cauchy integrals, note that one can write the kernel as
   $$
   \cot\left(\frac{z - z'}{2}\right) + i = \frac{2ie^{-iz'}}{e^{-iz'} - e^{-iz}}.
   $$
  We can therefore define $\Ccal_{\mathrm{per}(j)}$ to be \eqref{def:BR+k} with $z - z_*$ replaced by $e^{-iz} - e^{-iz_*}$. Note that 
  $$
  |e^{-iz(\al, t)} - e^{-iz_*(t)}| \approx |z - z_*|
  $$
  in a neighborhood of the left corner tip, and analogously near the right corner tip at $z(\pi, t) = z_*(t) + 2\pi$ by periodicity, so straightforward analogs of Corollary \ref{lem:BRcorrections} and Lemma \ref{lem:BRders} hold.

   Another important ingredient is the periodic Dirichlet-to-Neumann operator, which we denote by $\Gcal_{\mathrm{per}}$. To analyze this linear map, observe that the function $ H(z) := e^{-iz} $ maps $\Omega_\pi$ conformally to $\tilde\Omega  \backslash (-\infty,0] $, where  $\tilde\Omega$ is a bounded domain with a corner. Furthermore, by standard results on removable singularities, any map which is continuous on~$\t\Om$ and conformal on $\tilde\Omega  \backslash (-\infty,0] $ can be extended as a conformal map on~$\t\Om$, and an analogous result holds for functions that are continuous on $\tilde\Omega  $ and harmonic on $\tilde\Omega  \backslash (-\infty,0] $. 
   
   This readily enables us to use our results about the Dirichlet-to-Neumann map on a domain with corners. Specifically, if we respectively denote by $\pa_n$ and $\pa_{\t n}$ the normal derivatives on $\Om$ and $\t \Om$, it is clear that they are related by
   \[
   \pa_n f = |H'|\, \big[\pa_{\t n} (f\circ H^{-1})\big]\circ H\,.
   \]
   The same relation must then hold if one replaces the normal derivatives $\pa_n$ and $\pa_{\t n}$ by the Dirichlet-to-Neumann maps of $\Om$ and $\t\Om$, $\Gcal_{\mathrm{per}}(z)$ and $\t\Gcal(H(z))$:
   \be\label{DtN-trafo}
   \Gcal_{\mathrm{per}}(z) f = |H'|\, \left[\t\Gcal(H(z)) (f\circ H^{-1})\right]\circ H\,.
   \ee
Therefore, we can use the results of Section \ref{s.inverse} on the Dirichlet-to-Neumann operator of a bounded domain with corners to handle the periodic case.
   
   We must still take care of one aspect that does not appear in the case of bounded domains: since $\Om$ is unbounded, one needs to prescribe the behavior at infinity the pressure and the velocity. Indeed, as mentioned in Remark~\ref{R.periodicOmega}, we must impose
\be\label{inf}
P(z) \rightarrow +\infty, \qquad v(z) \rightarrow 0
\ee
as $z \rightarrow -i\infty$. Here and in what follows, we omit the time dependence notationally whenever no confusion may arise.

To analyze the velocity, one can use the periodic Cauchy formula and the asymptotics of the kernel to see that the condition~\eqref{inf} for the velocity can be written in terms of the boundary velocity as
\be\label{eq:v-inf}
0=\lim_{z\rightarrow -i \infty} v^*(z) = \frac{1}{2\pi}\int_{\Ga_\pi} V^*(z)dz .
\ee  
However, the way we construct the boundary velocity~$V$ in this paper is by picking one component of the boundary velocity (or rather of its $D_s$-derivative), say $V_1$, as a basic unknown, and subsequently constructing the other component (namely, $V_2$) using the Dirichlet-to-Neumann map. Condition~\eqref{eq:v-inf} is then a constraint on $V_1$ which involves the a priori unknown quantity $\Gcal_{\mathrm{per}} V_1$. However, there is an easy way to incorporate this in the analysis: one can always  map $\Om$ conformally to the lower half-plane using the Riemann mapping theorem, and in this case the condition~\eqref{eq:v-inf} simply means that the basic unknown~$V_1$, composed with this conformal map, should have zero mean on~$\Gamma_\pi$. This is therefore a condition, depending only on~$V_1$ and on the interface curve~$z$, that one must impose on the admissible initial data to ensure that the velocity tends to zero at infinity. Since the time derivative of~\eqref{eq:v-inf} vanishes, this condition (and therefore the fact that $v(z)\to0$ as $z\to-i\infty$) is then preserved by the evolution.

Let us now discuss how one handles the condition~\eqref{inf} for the pressure. The evolution equations \eqref{eqs:evolGa} for $z$ and $V$ obviously remain the same. We also construct $\si = -\pa_n P$ just as before; that is, we look for a bounded function $\psi$, harmonic in $\Om$, such that 
$$
\psi|_{\Ga}  = - gz_2 - \frac{1}{2}|V|^2,
$$
and write $\si$ as
$$
\si= \Gcal_{\mathrm{per}}\psi + gz_{1s} + V_s\cdot V^\perp,
$$
cf.\ \eqref{eq:Psi}-\eqref{eq:sigma}. Now recall the well known formula 
$$
P(z) = - gz_2 - \frac{1}{2}|v(z)|^2 - \phi_t(z) , \quad z\in\Om ,
$$
where $\phi$ is the velocity potential (i.e. $v = \nabla \phi$). By construction, we know that $\phi_t = \psi$ is bounded, as is the velocity. Since $g>0$, we infer that the condition for the pressure in~\eqref{inf} must hold. A standard application of the Phragm\'en--Lindel\"of principle shows that $P > 0$, so that $\si(z, t) > 0$ at any regular point just as in the case of a bounded domain.

We can now move on to Section \ref{s.apriori}. We describe the interface $z$ in terms of the basic variables $z_*, \nu_\pm, \Theta$ and $\log|z_\al|$ exactly as in \eqref{regularity-z-corner}--\eqref{eq:norm-theta}. As for the boundary velocity, we modify the ansatz for the asymptotic expansion slightly to write 
\be\label{eq:V-rep-pi}
V^*(\al, t) = b_0(t) + b_1(t)(ie^{-iz(\al, t)} - ie^{-iz_*(t)}) + U^*(\al, t),
\ee
where the constraint for $U^*$ now involves the periodic Cauchy integral:
$$
U^* = \Ccal_{\mathrm{per}}(z) U^*.
$$	

In addition, as we have already mentioned, Equation \eqref{eq:v-inf} is a constraint on $\int U^*\, dz$ (which, as noted above, can be written using only~$U_1$ and~$z$, and now  $b_0, b_1$ as well). Otherwise, we simply follow \eqref{eq:v-rep}--\eqref{eq:norm-V}. We can then derive the ODE system for the coefficients $z_*, \nu_\pm, b_0$ and $b_1$ and prove an a priori energy estimate exactly as before. One only needs to replace every occurrence of $\Ccal$ and $(z - z_*)$ by $\Ccal_{\mathrm{per}}$ and $e^{-iz(\al, t)} - e^{-iz_*(t)}$, respectively, and note that results of Section~\ref{ss.BR}  hold in the periodic case as discussed above. 

The coefficients of the ODE system that one derives this way look very slightly different than in the periodic case  as a consequence of the fact that the above asymptotic expansions now involve the exponential map. Specifically, for $z_*$ and $\nu_\pm$ one finds
\be\label{E.odeper1}
\frac{dz_*}{dt} = b_0^*, \quad \frac{d\nu_+}{dt} = \Re\left(ib_1e^{- iz_* + 2i\nu_+}\right), \quad \frac{d\nu_-}{dt} = 			\Re\left(ib_1 e^{- iz_* + 2i\nu_-}\right).
\ee
To analyze $b_0$ and $b_1$, note that the expansion \eqref{eq:V-rep-pi} is particularly convenient because, going over to the bounded domain $\t \Om$ via the conformal map $H^{-1}$, we simply recover the asymptotic expansion \eqref{eq:v-rep} for the transformed boundary velocity $\t V := V \circ H^{-1}$. We can therefore directly use the results of Lemma \ref{lem:si} and the relation~\eqref{DtN-trafo} to find the leading order term in the expansion of $\si$. After a short calculation, we arrive at the equations
\be\label{E.odeper2}
\frac{db_0}{dt} = ig, \qquad \frac{d}{dt}(b_1e^{-iz_*})=  \frac{1}{\cos 2\nu}|b_1e^{-iz_{*}}|^2e^{-i(\nu_+ + \nu_-)}.
\ee

Note that the ODE system~\eqref{E.odeper1}-\eqref{E.odeper2} is just as in the case of bounded water domains (Equations~\eqref{eq:nudt} and~\eqref{eq:db}), with the caveat that the function $b_1$ must be replaced by $b_1e^{-iz_*}$ everywhere in the ODEs. This is merely because our choice of the parameter $b_1$ in the periodic case has a slightly different physical meaning as $b_1$ had in the case of bounded domains. Indeed, the coefficient~$b_1$ was introduced so that the gradient of the velocity at the corner tip is $b_1$ in the bounded case, but this quantity is given by $b_1 e^{-iz_*}$ in the periodic case. Therefore, this apparent difference simply reflects that, as one would expect, the leading order asymptotic behavior of the fluid near the corner does not depend on whether the fluid domain is bounded or periodic. 

The rest of proof of Theorem \ref{thm:main} carries over to the periodic way in a straightforward way. Besides, as we discussed right after~\eqref{eq:W-int} above, the constraint~\eqref{eq:v-inf} (which is preserved by the evolution) ensures that the solutions have the desired behavior at infinity. Incidentally, one should note that the proof also goes through if we relax \eqref{inf}-\eqref{eq:v-inf} to assume that the velocity is bounded at infinity but does not necessarily vanish in the limit $z \rightarrow -i\infty$; this corresponds to initial data that do not necessarily satisfy the additional constraint~\eqref{eq:v-inf}.

\section{Finite-time singularity formation}
\label{S.sing}

The fact that water waves with corners can develop singularities in finite time is a fairly straightforward consequence of the local wellposedness result established as Theorem~\ref{thm:main} and of the ODE system~\eqref{eqs:sys-z}--\eqref{eqs:sys-V}, which we  presented as~\eqref{E.ODEintro} in the Introduction. 

A precise statement of this fact, which accounts for Theorem~\ref{T.sing}, is as follows:

\begin{theorem}\label{T.sing2} 
	Set $\beta:=\frac12$. For any $k\geq2$, there exists an initial datum $(z^0, V^0)\in\Ocal^k_\beta(m)$ and some $T>0$ such that the unique solution $(z,V)$ to the initial value problem~\eqref{eqs:sys} with this initial datum is $\Ccont([0, T), \Ocal^k_{\beta}(m))\cap \Ccont^1([0, T), \Bcal^{k-1}_{\beta}(m))$, but
	\be\label{E.blowup}
	\lim_{t\to T^-}\left[ \|(z,V)\|_{\Bcal^k_{\beta}(m)} + \|(z_t,V_t)\|_{\Bcal^{k-1}_{\beta}(m)} + \Fcal(z)\right]=\infty\,.
	\ee
	This phenomenon is stable, as suitably small perturbations of this initial datum blow up as well.
\end{theorem}

\begin{proof}
	
	In Appendix~\ref{A.ODE} we show that there is an open set~$\Ocal$ of initial conditions $(\nu_+^0,\nu_-^0,b_1^0)\in\Rbb^2\times\Cbb$ for which the ODE system~\eqref{E.ODEintro} blows up in finite time. More precisely, for all $(\nu_+^0,\nu_-^0,b_1^0)\in\Ocal$ one has:
	\begin{enumerate}
		\item[i)] The maximal forward existence time of the ODE, $T_0\equiv T_0(\nu_+^0,\nu_-^0,b_1^0)>0$, is finite.
		\item[ii)] The functions $2\nu(t):=\pi +\nu_+(t)-\nu_-(t)$ and $|b_1(t)|$ are positive and increasing for $t\in(0,T_0)$.
		\item[iii)] As $t\to T_0^-$, $2\nu(t)\to\frac\pi 2$ and $|b_1(t)|\to\infty$.
	\end{enumerate}
	
	With $\beta:=\frac12$ and any fixed $k\geq2$, let us take some initial data 
	$$
	(z_*^0, \nu_+^0, \nu_-^0, \Theta^0, \log|z^0_\al|,b_0^0,b^0_1,U^0_1)\in \Ocal^k_\beta(m)
	$$
	for which $(\nu_+^0,\nu_-^0,b_1^0)$ are in~$\Ocal$, where we recall that $\Ocal^k_\beta(m)$ is the open set of admissible initial conditions that appears in our local existence result (Theorem~\ref{thm:main}). It is obvious that this can be done, as one only needs to ensure that the functions $(\Theta^0, \log|z^0_\al|,U^0_1)$ are regular enough and that the corresponding interface curve does not have any self-intersections.
	
	Theorem~\ref{thm:main} ensures that this initial datum determines a unique solution 
	$$
	(z,V)\equiv (z_*, \nu_+, \nu_-, \Theta, \log|z_\al|,b_0,b_1,U_1)\in\Ccont([0, T), \Ocal^k_{\beta}(m))\cap \Ccont^1([0, T), \Bcal^{k-1}_{\beta}(m)),
	$$ 
	with $(\nu_+,\nu_-,b_1)$ given by the corresponding solution to the ODE~\eqref{E.ODEintro}. By Remark~\ref{R.continuation}, the solution (with $(\nu_+,\nu_-,b_1)$ determined by~\eqref{E.ODEintro}) can be continued as long as the corner angle~$2\nu$ satisfies condition~\eqref{eq:nu}, the degenerate Rayleigh--Taylor condition 
	\begin{equation}\label{E.RTfails}
		\inf_{t\in(0, T)}\inf_{\al\in I}\frac{\si(\al,t)}{|z(\al,t)-z_*(t)|}>0
	\end{equation}
	holds, and the
	$E_{k,\beta}$-energy (i.e., the right hand side of~\eqref{E.blowup}) remains bounded.

	Let us now assume that the right hand side of~\eqref{E.blowup} remains bounded for~$T:=T_0$, since otherwise the theorem automatically follows. To show that this leads to a contradiction, we will use that if the continuation criterion of Remark~\ref{R.continuation} is satisfied up to time $T_0$, we immediately infer that the blowup condition~\eqref{E.blowup} holds for $T:=T_0$ because $\|(z,V)\|_{\Bcal^k_{\beta}(m)} \geq |b_1|\to\infty$ as $t\to T_0^-$. 
	
	Since $\beta=\frac12$, verifying the first condition of the continuation criterion is immediate. With $(\nu_+^0,\nu_-^0,b_1^0)\in\Ocal$ and $T_0$ as above, condition~\eqref{eq:nu} (which in the case $\beta=\frac12$ simply reads as $2\nu\in(0,\frac\pi2)$) holds for all $t\in(0,T_0)$ because $2\nu$ is positive, increasing and tends to~$\frac\pi 2$ as $t\to T_0^-$.
	
	Therefore, to show that the solution to the water wave system blows up, it only remains to show that the degenerate Rayleigh--Taylor condition cannot fail either. To see this, suppose that there is some time $T<T_0$ such that
	\begin{equation}\label{E.RTfails2}
		\liminf_{t\to T}\inf_{\al\in I}\frac{\si(\al,t)}{|z(\al,t)-z_*(t)|}=0.
	\end{equation}
	By Equation~\eqref{V} and Corollary~\ref{C.reg-v}, the velocity in the water region~$\Om\equiv \Om(t)$ is
	\[
	v(z,t)^*= b_0(t) + b_1(t)[z-z_*(t)]+ u(z,t)^*.
	\]
	Since $|b_1(t)|\geq |b_1^0|$ because $|b_1(t)|$ is increasing, $\frac{du^*}{dz}(z_*(t),t)=0$ by the regularity of the function~$U_1$ in~\eqref{V} and $\|u(\cdot,t)\|_{C^{1,\lambda}(\Om(t))}\lesssim E_{k,\beta}(t)$ is bounded uniformly for $0<t<T$ by hypothesis, it follows that there exists some uniform radius~$\de>0$ such that the
	the nonnegative function 
	\[
	F:=| \nabla v|^2\geq0
	\]
	is lower bounded as
	\be\label{E.lowerbd}
	F(z,t)\geq \frac12|b_1^0|,\qquad z\in \Om(t)\cap B_\de(z_*(t)),
	\ee
	where $B_\de(z)$ denotes the ball centered at~$z$ of radius~$\de$.
	
	By~\eqref{E.subh}, in~$\Om$ the pressure satisfies
	\be\label{E.DeF}
	-\Delta P=F,\qquad P|_{\Ga}=0.
	\ee
	As the geometry of~$\Om$ is controlled up to time~$T$, we claim that the lower bound~\eqref{E.lowerbd}  implies
	\be\label{E.unifsi}
	\si =-\pa_nP\geq c|z-z_*|
	\ee
	for some uniform $c>0$ and all $t<T$, which contradicts the assumption~\eqref{E.RTfails2}. 
	
	In fact, it is not hard to see that the estimate~\eqref{E.unifsi} follows from a calculation on the strip $\Pi_\nu$ very similar to the one in the proof of Lemma~\ref{lem:special_sol}. As we do not need to be so careful here about the sharp dependence of the estimates on the regularity of~$z$ because we have more regularity than strictly necessary for this part of the argument, for the benefit of the reader we provide an elementary alternative proof of this fact too. For this, we henceforth omit the time variable for the ease of notation.
	
	Rotating the axes and applying a smooth diffeomorphism (equal to the identity in a neighborhood of the corner point) to~$\Om$ if necessary, e.g., as in the proof of~\eqref{eq:poincare-weighted}, we can assume that for $t$ close to~$T$ the region $\Om$ sits on the right of the point~$z_*$, that is, $z_{1*}=\inf_{\Ga} z_1$. The conformal map $z\mapsto (z-z_*)^{\frac\pi{2\nu}}$ then maps $\Om$ to a $C^{1,\lambda}$~domain, which can be in turn conformally mapped to the disk~$\Dbb$~\cite{Pomme}. Let us denote by $h:\Dbb\to\Om$ the resulting conformal map. We can assume that $h(-1)=z_*$, that is, that $-1\in\Dbb$ is mapped to the corner. One then readily finds the asymptotic expressions
	\[
	|h(\zeta)|\approx|\zeta+1|^{\frac{2\nu}\pi},\qquad |h'(\zeta)|\approx |\zeta +1|^{\frac{2\nu}\pi-1}
	\]
	for $\zeta$ in a uniformly small neighborhood of~$-1$. Moreover, there is some uniform $\de'>0$ such that
	\[
	h^{-1}(\Om\cap B_\de(z_*))\subset \Dbb\cap B_{\de'}(-1).
	\]
	Then~\eqref{E.unifsi} is simply the pullback to~$\Om$ of the uniform lower bound for $\pa_n(P\circ h)$ that one gets by transforming~\eqref{E.DeF} to the disk via~$h$ and using Hopf's lemma.
	
	Therefore, under the above hypotheses, we conclude that the continuation criterion is satisfied up to the time~$T_0$. As we discussed above, this is a contradiction, so we conclude that the solution to the water wave system blows up in time at most~$T_0$.	The stability of the blowup follows directly from the fact that the ODE blows up for an open set of initial values. 
\end{proof}

\begin{remark}
	The ODE system~\eqref{E.ODEintro} does not depend on~$\beta$, so both this part of the argument and the discussion of the Rayleigh--Taylor condition hold for all weights $\beta\in (-\frac12,\frac12]$. However, the fact that~\eqref{eq:nu} holds for all $t<T_0$ strongly relies on the choice $\beta=\frac12$: indeed, this condition means that $2\nu\in(0,\pi/(\frac52-\beta))$, so it will fail for any fixed $\beta<\frac12$ because $2\nu\to\frac\pi2$ as $t\to T_0$.		
	If this continuation criterion fails, one cannot conclude that the solution to the water wave system blows up, as it could conceivably happen that the solution $(z_*, \nu_+, \nu_-, \Theta, \log|z_\al|,b_0,b_1,U_1)$ is well defined but the functions $(\nu_+,\nu_-,b_1)$ no longer coincide with the solution to the ODE~\eqref{E.ODEintro}. In view of Subsection~\ref{ss.unbounded}, it is clear that the proof works verbatim in the periodic case.
\end{remark}




\section{The Dirichlet-to-Neumann map for domains with corners}\label{s.inverse}

Our goal now is to prove several estimates involving the Dirichlet-to-Neumann map that are used throughout the paper.

In this section, we henceforth assume that~$\Om$ is a domain of the kind considered in the proof of the local wellposedness result, Theorem~\ref{thm:main}. That is, the boundary curve~$\Ga$ is as in this theorem, in  the sense that its $E_{k, \beta}$-norm (or rather the part of the energy involving only the curve~$z$) is finite, with $k\geq2$. All the implicit constants in the results we shall establish in this section are bounded in terms of said $E_{k, \beta}$-energy. Also, whenever we write~$k$ or~$\beta$ in this section, we always refer to the parameters that appear in the energy.

Without any loss of generality, in this section we shall assume that $z_*=0$, that is, the corner point is located at the origin. As before, the angle of the (curvilinear) corner is~$2\nu\in(0,\frac\pi2)$. Note that the boundary, which satisfies the arc-chord condition, is Lipschitz everywhere, and of class $\Ccont^{1 + k, \lambda}$ away from the corner point. We will notationally omit any dependence of the curve on~$t$, as time does not play a role in this section.

For clarity, we will divide this section in three subsections: first we shall state the main results we prove and discuss their interrelations, then we will introduce a key change of variables that lies at the heart of the proofs, and finally we will discuss how the main results of this sections are actually established.

\subsection{Statement of the results}\label{ss.statements}

The central result of this section is the following proposition, which defines the Dirichlet-to-Neumann map~$\Gcal(z)$. In what follows, we shall refer to the function~$u$ given by this proposition as the harmonic extension of~$\psi$ to~$\Om$.

\begin{proposition}\label{prop:dirichletInt} 
Take $\ga \in  \Rbb$ be such that $\left|\ga + \frac{1}{2}\right| <\frac{ \pi}{2\nu}$. For each $\psi\in \Lcal^{i+\frac{1}{2}}_{2, \ga + i + \frac{1}{2}}(\Ga)$ with $1 \leq i \leq k + 1$, there exists a unique solution~$u$ of the boundary value problem
	\be\label{eq:systemOmDirichlet}
	\Delta u = 0, \quad u|_\Ga = \psi
	\ee
	such that 
	\begin{equation}\label{E.defDN}
			\Gcal(z)\psi := \pa_n u
	\end{equation}
	is in $\Lcal^{i - \frac{1}{2}}_{2, \ga  + i + \frac{1}{2}}(\Ga)$. Furthermore, the Dirichlet-to-Neumann operator given by~\eqref{E.defDN} defines a continuous map
	$$
	\Gcal(z): \Lcal^{i+\frac{1}{2}}_{2, \ga  + i  + \frac{1}{2}}(\Ga) \rightarrow \Lcal^{i - \frac{1}{2}}_{2, \ga  + i + \frac{1}{2}}(\Ga).
	$$
\end{proposition}

As we shall see, an easy consequence of this result is the following. We state it so that we can directly apply it to estimate the velocity of the fluid inside the water region under the assumptions of our local wellposedness result, Theorem~\ref{thm:main}. One should note the specific form of~$\psi$ we are analyzing in detail corresponds exactly to the formula~\eqref{eq:v-rep} for the boundary trace of the velocity, since we are taking $z_*=0$.

\begin{corollary}\label{C.reg-v}
	If $\beta \leq \frac12$ is such that $|(\beta - 2) +\frac12|<\frac\pi{2\nu}$ and $\psi\in \Lcal^{k+\frac{3}{2}}_{2,\beta+k-\frac{1}{2}}(\Ga)$, then the harmonic extension $u$ of $\psi$ is bounded as $\|u\|_{C^{1,\lambda}(\Om)}\lesssim \|\psi\|_{\Lcal^{k+3/2}_{2,\beta+k-1/2}(\Ga)}$ for some $\lambda>0$. Furthermore, if $\psi$ is a complex-valued function on the boundary of the form	\[
	\psi(z(\al)):= b_0+ b_1 z(\al) + \t \psi(z(\al))\,,\qquad b_0, b_1\in\Cbb, \qquad  \t\psi\in \Lcal^{k+\frac{3}{2}}_{2,\beta+k-\frac{1}{2}}(\Ga),
	\]
	its conformal extension is of the form
	\[
	u(z)= b_0+ b_1z+ \t u(z)
	\]
	with $\|\t u\|_{C^{1,\lambda}(\Om)}\lesssim \|\t \psi\|_{\Lcal^{k+3/2}_{2,\beta+k- 1/2}(\Ga)}$
\end{corollary}

Although H\"older norms are convenient in a couple of applications of Corollary~\ref{C.reg-v}, we are mostly interested in mapping properties of the operators between weighted $L^2$-based Sobolev spaces of the kind we need to prove the energy estimates. In the domain~$\Om$, we shall denote the corresponding spaces by $\Lcal^l_{2,\ga}(\Om)$, where $\ga\in\Rbb$ and where~$l$ is a nonnegative integer. They consist of the functions on~$\Om$ for which the norm
\be\label{eq:sobolev-2d}
\|v \|^2_{\Lcal^l_{2,\ga}(\Om)} := \sum_{|\al| \leq l} \int_\Om (x^2+y^2)^{\ga + |\al| - l}|\pa_{x,y}^\al  v|^2\,dx \, dy 
\ee 
is finite, where the sum is over multiindices. It is known~\cite{Kufner} that the space of functions $\Ccal_c^\infty(\overline\Om\backslash\{0\})$ whose support has positive distance to the corner point is dense in $\Lcal^l_{2,\ga}(\Om)$. Furthermore, a boundary trace operator $\text{tr}:\Lcal^{l+1}_{2,\ga}(\Om) \rightarrow \Lcal^{l}_{2,\ga}(\Ga)$ and an extension operator can be defined as usual  (see e.g.~\cite{KMR}):

\begin{lemma}\label{lem:extension}
	Let $1\leq i\leq k + 1$. The trace is an onto, continuous map $\Lcal^{i + 1}_{2,\ga}(\Om)\rightarrow \Lcal^{i + \frac{1}{2}}_{2,\ga}(\Ga)$. Moreover, given $\psi\in \Lcal^{i + \frac{1}{2}}_{2, \ga}(\Ga)$, there exists $\Psi \in \Lcal^{i + 1}_{2, \ga}(\Om)$ such that $\Psi|_\Ga = \psi$ in the trace sense and $\|\Psi\|_{\Lcal^{i + 1}_{2, \ga}(\Om)} \lesssim \|\psi\|_{\Lcal^{i + \frac{1}{2}}_{2, \ga}(\Ga)}$.
\end{lemma}

To construct the harmonic extension, we will use a function~$\Psi$ as in Lemma~\ref{lem:extension} and then solve a problem with sources but with zero boundary conditions. The result we will use is the following:

\begin{proposition}\label{prop:maz}
	Let $\ga \in\Rbb$ such that $|1 - \ga| < \frac{\pi}{2\nu}$. Given $f\in \Lcal^{i}_{2, \ga + i}(\Om)$ with $0 \leq i \leq k$, the Dirichlet problem
	\be\label{eq:systemOmPoisson}
	\Delta v = f, \qquad v|_{\Ga} = 0	
	\ee	
	has a unique solution $v\in \Lcal^{i + 2}_{2,\ga + i}(\Om)$, which is bounded as
	\be\label{eq:estVnormal}
	\|u\|_{\Lcal^{i+2}_{2, \ga + i}(\Om)} \,\lesssim \, \|f\|_{\Lcal^{i}_{2, \ga  + i}(\Om)}.
	\ee 
\end{proposition}

The key point of this result is that one must ensure that the dependence of the various constants on the regularity of the interface can be controlled using only the weighted Sobolev norms of~$z$ that we control with the energy~$E_{k,\beta}$, instead of using pointwise norms as usual. When no weights are involved,  so the interface does not have corners, one can use~\cite{ChSh17}. When sufficiently many derivatives of the curve~$z$ are controlled, the result is proven in~\cite{MazNaPla}. However, for the energy estimates, it is key to have optimal dependence on the regularity of~$z$, as in the lemma.

Note that Proposition \ref{prop:dirichletInt} follows easily from the results stated above, whose proofs will be presented later in this section:

\begin{proof}[Proof of Proposition \ref{prop:dirichletInt}.] 
	Given $\psi\in \Lcal^{i + \frac{1}{2}}_{2, \ga + i + \frac{1}{2}}(\Ga)$, we use  Lemma~\ref{lem:extension} to pick an extension $\Psi\in \Lcal^{i + 1}_{2,\ga + i + \frac{1}{2}}(\Om)$ of comparable norm, which satisfies	
	$$
	\pa_n \Psi \in \Lcal^{i+ \frac{1}{2}}_{2, \ga + i + \frac{1}{2}}(\Ga) , \qquad f := \Delta\Psi \in \Lcal^{i - 1}_{2, \ga + i + \frac{1}{2}}(\Om).
	$$
	By Proposition \ref{prop:maz}, for this choice of~$f$ there exists a solution $v\in \Lcal^{i + 1}_{2, \ga + i + \frac{1}{2}}(\Om)$ of the boundary value problem~\eqref{eq:systemOmPoisson}. Then $u:=\Psi-v$ satisfies~\eqref{eq:systemOmDirichlet}. Uniqueness follows from Proposition~\ref{prop:maz}.
\end{proof}

\subsection{The change of variables}
\label{ss.cov}

Our goal in this section is to construct a diffeomorphism mapping the bounded domain with a corner~$\Om$ into a smoother unbounded domain~$\Scal$. This diffeomorphism, or change of variables, will be essential in our analysis of the Laplace equation on~$\Om$. 

It is worth stressing that the diffeomorphism we construct is not conformal. The reason is that we need to ensure that all the relevant norms of the diffeomorphism are bounded in terms of the norms of~$z$ controlled by the $E_{k,\beta}$-energy.

For convenience, let us denote the diffeomorphism by
$$
\dif ^{-1}: \Om \rightarrow \Scal.
$$ 
The domain $\Scal\subset\Rbb^2$ that we will define does not have any corners (in fact, it is of class $\Ccont^{1 + k, \lambda}$ for some $\lambda>0$). The `right' side of this domain, $\Scal \cap \Rbb_+^2$, is bounded, where we use the notation
\[
\Rbb^2_\pm:=\{ (\t x,\t y)\in\Rbb^2: \pm \t x>0\}\,.
\]
The left side of~$\Scal$ is an infinite strip; more precisely, $\Scal \cap \Rbb_-^2 = \Pi_\nu^-$, where
\be\label{eq:Pi-}
\Pi_\nu := \left\{ (\t x , \t y)\in\Rbb^2 \, : \,\, \t x \in \Rbb, \,\,  |\t y| < 2\nu  \right\}, \qquad \Pi_\nu^- := \Pi_\nu \cap \Rbb_-^2
\ee
respectively denote the horizontal strip of width~$2\nu$ and its intersection with the negative half-plane.

The point here is to construct the diffeomorphism in a small neighborhood of the curvilinear corner, that is, to construct a diffeomorphism (with the right dependence on the norms of~$z$)
$$
\dif ^{-1}:\Om\cap B  \rightarrow \Pi_\nu^-
$$
where $B $ is a ball centered at the origin and of small radius~$\delta>0$. This is what we will do next. One can then extend this diffeomorphism as a map $\Om\to \Scal$ preserving this condition in a standard way (and one can even assume that $\dif^{-1}$ is the identity outside $B$, or that the whole domain~$\Scal$ is contained in~$\Pi_\nu$). A detailed discussion of this point is given e.g.~in~\cite{MunRam}.

Without loss of generality, by rotating the axes if necessary, for small~$\de$ we may assume that $\Om \cap B $ is contained in the right half-plane (i.e., $\Om \cap B \subset \Rbb^2_+$) and that $\nu_\pm = \pm\nu$. Let $\Ga^\pm$ denote the arcs forming the corner in the neighborhood of the tip, i.e. 
$$
\Ga\cap B  = \Ga^+\cup \Ga^-, \qquad \Ga^\pm := \{(x,\ka_\pm(x))\,:\, x\in[0,\de)\}.
$$
For small $\de>0$, we can assume that $\ka_-(x) < \ka_+(x)$ for $x\in I_\de := (0,\de)$ and that $\ka_\pm(0) = 0$. The assumptions on the regularity of the interface, inside~$B$, are simply that
$$
\ka_\pm(x) = \pm x\tan\nu + \Lcal^{k + 2}_{2,\beta + k}(I_\de)\,.
$$

It will be convenient to analyze $\Ga^\pm$ in polar coordinates, $(r,\vartheta)$. We then denote by $(r_\pm(x),\vartheta_\pm(x))$ the polar coordinates of the point $(x,\ka_\pm(x))$ on~$\Ga^\pm$. For $x\in I_\de$, we then have
$$
r_\pm(x) = \sqrt{x^2 + \ka_\pm(x)^2}, \qquad r'_\pm(x) = \frac{x}{r_\pm} \left(1 + \frac{\ka_\pm(x)}{x}\, \ka'_\pm(x)\right), \qquad r'_\pm(x) \approx \frac{x}{r_\pm}  
$$
on the respective arcs, so we can parametrize the arcs in terms of the radial distance along them, provided~$\de$ is small. With some abuse of notation, we will denote this parametrization of~$\Ga^\pm$ by $(r,\vartheta_\pm(r))$.

It is not difficult to see that
$$
\vartheta_\pm(r) = \pm\nu + \Lcal^{2 + k}_{2,\beta + k + 1}(I_{\de})\,,
$$
so, by Lemma~\ref{lem:sobolev}, for $i \leq  1 + k$ one has
\be\label{eq:vartheta_pm}
\left(r\frac{d}{dr}\right)^i (\vartheta_\pm(r) \mp \nu)\in \Lcal^1_{2,\beta}(I_\de), \quad \left(r\frac{d}{dr}\right)^i (\vartheta_\pm(r) \mp \nu) = O(r^{\lambda_\beta}).
\ee
We recall that $\lambda_\beta := \frac{1}{2} - \beta \in [0, 1)$ given $\beta \in \left(-\frac12, \frac12\right]$. 

We can now define the diffeomorphism in polar coordinates as 
\be\label{formulahtilde}
\t\dif^{-1}(r,\vartheta) := \left(\log r , \;  2\nu\left  (-\frac{1}{2} + \frac{\vartheta -\vartheta_-(r)}{\vartheta_+(r)-\vartheta_-(r)}\right  )\right ),
\ee
so that $\dif^{-1}: \Om \cap B  \rightarrow \Pi_\nu$ is simply the above expression in Cartesian coordinates:
\[
\dif^{-1}(x,y):=\t\dif^{-1}(r(x,y),\vartheta(x,y))\,,
\]
which we can now extend to a diffeomorphism $\Om\to\Scal$.

Under the change of variables $\dif $, the Laplace equations transforms as 
\be\label{eq:poissonH}
-\Delta u = f \quad \Leftrightarrow \quad - \mathrm{div}(\Abb \na (u \circ \dif )) = f\circ \dif  |\det D\dif |,
\ee
where the matrix-valued function~$\Abb$ is
$$
\Abb \circ \dif ^{-1} := \frac{1}{|\det D\dif ^{-1}|} \, \, D\dif ^{-1} (D\dif ^{-1})^\top .
$$
For future reference, let us also record the transformation rule for the normal derivative in terms of unit normal $\v n$ and tangent $\v t$ vectors:
\be\label{eq:normalH}
\pa_n u = g \quad \Leftrightarrow \quad  \Abb \na (u \circ \dif ) \cdot \v n =  g \circ \dif  |D\dif  \cdot \v t|.
\ee

To derive bounds on the unbounded part of~$\Scal$, we need to know the asymptotic behavior of $\Abb$ on $\Pi_\nu^- = \dif ^{-1}(\Om \cap B )$. We shall show by direct calculation that $\Abb=\Ibb + \Lbb$, where $\Ibb$ is the identity matrix and where the error term $\Lbb(\t x,\t y)$ decays exponentially as $\t x\rightarrow -\infty$. To this end, we compute
$$
D\t \dif^{-1}(r,\vartheta) = \left(
\begin{array}{cc}
	1/r & 0\\
	\pa_r h_2^{-1}(r,\vartheta) & \frac{2\nu}{\vartheta_+(r) - \vartheta_-(r)}
\end{array}
\right),
\quad
D\t\dif^{-1}\circ \dif (\t x, \t y) = \left(
\begin{array}{cc}
	e^{ -\t x} & 0\\
	X(\t x, \t y) & {1}/{\Theta(\t x)}
\end{array}
\right),
$$
where we have used that $r = e^{ \t x}$ and we have introduced the notation 
$$
\Theta_\pm(\t x) := \frac{1}{2\nu}\vartheta_\pm(e^{ \t x}), \qquad \Theta(\t x) := \Theta_+(\t x) - \Theta_-(\t x), \qquad X(\t x, \t y) := (\pa_r \t \dif^{-1}_2)\circ {\dif (\t x, \t y)}.
$$
Taking into account the factors coming from the change to polar coordinates, we then obtain
\be\label{eq:det}
|\det D\dif ^{-1}\circ \dif | =  e^{ -2\t x}/{\Theta(\t x)}.
\ee

Using these formulas, an elementary computation yields the formula
$$
\aligned
\Abb(\t x, \t y) &=\Theta(\t x)\left(
\begin{array}{cc}
	1 & X(\t x, \t y) e^{\t x}\\		      
	X(\t x, \t y) e^{\t x} &  X(\t x, \t y)^2 e^{2\t x} + \Theta(\t x)^{-2}
\end{array}
\right) ,
\endaligned
$$
which we can write as
\be\label{eq:Lbb}
\Abb =  \Ibb +  \Lbb, \qquad \Lbb  :=  \left(\begin{array}{cc}
	\Theta(\t x) - 1 & Y(\t x, \t y)\\		      
	Y(\t x, \t y) &  \frac{Y(\t x, \t y)^2}{\Theta(\t x)} + \left(\frac{1}{\Theta(\t x)} - 1\right)
\end{array}
\right)
\ee
with $Y(\t x, \t y) := \Theta(\t x) X(\t x, \t y) e^{ \t x}$. 

Equivalently, one can write this expression as
\be\label{eq:Y}
Y(\t x, \t y)
= (\t y-\nu)\Theta_-'(\t x) - (\t y + \nu)\Theta_+'(\t x).
\ee
By~\eqref{eq:vartheta_pm}, for $0 \leq i\leq k + 2$ we have 
\be\label{eq:regTheta}
\gathered
\frac{d^i}{d\t x^i}\left(\Theta_\pm(\t x) \mp 1/2 \right) \in \Lcal_{2,-\lambda_\beta}(\Rbb_-, \exp \t x)\endgathered
\ee
or, to put is differently, $\Theta_\pm(\t x) \mp 1/2\in H^{k + 2}_{-\lambda_\beta}(\Rbb_-, \exp \t x)$. Note that, when $\beta = \frac12$, we have $\lambda_\beta = 0$ and therefore $H^{k + 2}_{-\lambda_\beta}(\Rbb_-, \exp \t x) \equiv H^{k + 2}(\Rbb_-)$.
In particular, when $\lambda_\beta > 0$, the coefficients of $\Lbb$ decay as $e^{\lambda_\beta \t x}$ as $\t x\rightarrow -\infty$. On the other hand, when $\lambda_\beta = 0$, the decay is not necessarily exponential, however a contradiction argument shows that the coefficients must decay at least as fast as $1/|t|^{1/2}$.

\subsection{Proof of the results}\label{ss.variational}

Here we will prove the results we stated in Subsection~\ref{ss.statements}.
The proofs will take place mostly on the domain~$\Scal$ introduced in the previous subsection, so we start by defining suitable weighted Sobolev spaces on~$\Scal$.

For $\ga \in\Rbb$ and a nonnegative integer~$i$, let us define
$$
F\in H^{i, \ga}(\Scal, \t m) \quad \Leftrightarrow \quad  \t m^{ \ga}F\in H^i(\Scal),
$$
where $H^i$ is the usual (unweighted) Sobolev space of order~$i$ and  $\t m$ is the exponential weight on~$\Rbb^2$  defined, with some abuse of notation, as 
$$
\t m(\t x, \t y)\equiv \t m(\t x) := e^{\t x}.
$$ 
We will drop the weight $\t m$ notationally, as we will only consider this kind of weights on~$\Scal$ and~$\pa\Scal$.

Note that, as we are considering only smooth exponential weights, the spaces $H^{i, \ga}(\Scal) $ and $H^i_\ga(\Scal)$ are equivalent, where
$$
F \in H^i_\ga(\Scal) \qquad \Leftrightarrow \qquad \sum_{|\al| \leq i}\int_\Scal \t m^{2\ga}|\pa_{\t x ,\t y}^\al  F(\t x, \t y)|^2 \, d\t x \, d\t y < \infty.
$$
In particular, it is not difficult to see that, with this weight,
\be\label{eq:trafo}
\aligned
f \in \Lcal^{i + 1}_{2,\ga + i + 1}(\Om) \qquad &\Leftrightarrow \qquad f\circ \dif  \in H^{i + 1, \ga + 1}(\Scal),\\
\psi\in \Lcal^{i + 1/2}_{2,\ga + i + 1/2}(\Ga) \qquad &\Leftrightarrow \qquad \psi \circ \dif \in H^{i + 1/2, \ga + 1/2}(\pa \Scal), \\
\psi\in \Lcal^{i}_{2,\ga + i}(\Ga) \qquad &\Leftrightarrow \qquad \psi \circ \dif \in H^{i, \ga + 1/2}(\pa \Scal).
\endaligned
\ee

We are now ready to present the proofs of the results in Subsection~\ref{ss.statements}:

\begin{proof}[Proof of Lemma \ref{lem:extension}] 
Let us start by showing how to construct a suitable extension~$\Psi$ is the function~$\psi$. Away from the corner point, the construction is standard (with optimal bounds in terms of the Sobolev regularity of~$z$ given in~\cite{ChSh17}), so we only need to consider functions supported near the origin, say inside $\frac12 B$ where $B$ is a small ball as in Subsection~\ref{ss.cov}.

As it is compactly supported, we can regard the transformed function on the boundary as being defined on the whole strip: 
$$
g = \psi \circ h_1 \in H^{i + \frac{1}{2},\ga + \frac{1}{2}}(\pa\Pi_\nu, \t m).
$$
We claim that there exists $G \in H^{i + 1,\ga +\frac{1}{2}}(\Pi_\nu)$ such that $G|_{\pa\Pi_\nu} = g$. Then $\Psi = (\chi G) \circ \dif ^{-1}$ is the required extension of $\psi$ in the neighborhood of the corner tip, where $\chi$ is a suitable cutoff such that $\chi|_{\pa\Pi_\nu}=1$ on the support of~$g$ and $\text{supp}(\chi)\subset \Pi_\nu^-$.

Now notice that 
$$
\t g :=  m^{\ga}g\in H^{i+ \frac{1}{2},0}(\pa \Pi_\nu)\equiv H^{i+ \frac{1}{2}}(\pa \Pi_\nu)
$$
is in an unweighted Sobolev space, so the usual extension theorem ensures the existence of a suitably bounded extension $\t G\in H^{i + 1}(\pa \Pi_\nu)$ of~$\t g$. The corresponding extension of~$g$ can then be taken as
	$$
	G := \t m^{ - \ga } \t G \in H^{i+ 1,\ga+ \frac{1}{2}}(\Pi_\nu)
	$$
	is the required extension of $g$. 
	
	That the boundary trace operator has the mapping properties specified in the statement can be proved using the same idea.\end{proof}

To prove Lemma \ref{prop:maz}, we first need to consider solutions of the Poisson equation on $\Scal$:

\begin{lemma}\label{lem:laplace-strip} Assume $\ga\in\Rbb$, with $|\ga| < \frac{\pi}{2\nu}$, and let $f\in H^{k - 1,\ga} (\Scal)$. Then there exists a unique solution $u \in H^{k + 1,\ga}(\Scal)$ of the boundary value problem
\be\label{eq:laplaceGpoisson}
-\Delta u = f,\qquad u|_{\pa\Scal} = 0,
\ee
that satisfies the estimate 
	$
	\|u \|_{H^{k + 1,\ga}(\Scal)} \, \lesssim \, \|f\|_{H^{k - 1,\ga}(\Scal)}.
	$
\end{lemma}

\begin{proof} 
We shall use a duality argument to prove the existence of weak solutions and prove higher regularity for the solutions using elliptic estimates.

Therefore, let us first show the existence of a weak solution $u\in H^{1,\ga}_0(\Scal)$ (defined as the closure of $\Ccont^\infty_c(\Scal)$ in $H^{1,\ga}(\Scal)$). For this, we consider the bilinear form $a: H^{1,\ga}_0(\Scal)\times H^{1,-\ga}_0(\Scal)\rightarrow \Rbb$ given by
	$$
	\aligned
	&a(u, v) := \int_\Scal \na u\cdot \na v.
	\endaligned
	$$
	
	For $(u,v)\in H^{1,\ga}_0(\Scal)\times H^{1,-\ga}_0(\Scal)$, we note we can effectively get rid of the exponential weights using that
	$$
	(\t u ,\t v) := (\t m^\ga u ,\t m^{-\ga}v)\in H^1_0(\Scal)\times H^1_0(\Scal).	
	$$
	Therefore, we define a quadratic form $\t a: H^{1}_0(\Scal)\times H^{1}_0(\Scal)\rightarrow \Rbb$ as
	\[
	\t a(\t u, \t v) := a(u,v)\,.
	\]
	To show this form is coercive, note that one can construct a weight depending only on~$\t x$, which we still denote by $\t m(\t x,\t y)\equiv \t m(\t x)$, which is equivalent to our exponential weight $e^{\t x}$ on~$\Scal$ and which satisfies
	$$
	|(\log \t m)'(\t x)| \leq 1, \quad \t x < 0; \qquad  \t m'(\t x) = 0, \quad \t x \geq 0.
	$$
	With this equivalent choice of weight, we compute
	$$
	\aligned
	\t a(\t u, \t u) &= \int_\Scal\na (\t m^{-\ga} \t u)\cdot \na (\t m^{\ga} \t u) \\
	&= \int_\Scal |\na \t u|^2 - \ga^2 \int_{\Pi_\nu^-} \frac{\t m_x^2}{\t m^2}\, |\t u |^2 \geq \int_{\Pi_\nu^-} (|\na \t u|^2 - \ga^2 |\t u |^2) + \int_{\Scal \cap \Rbb^2_+} |\na \t u|^2.
	\endaligned
	$$
	 
	As we can assume that $\Scal\subset\Pi_\nu$ and the Poincar\'e constant of the strip of width $2\nu$ is
	\[
	\int_{\Pi_\nu}|\nabla w|^2\geq \frac{\pi^2}{4\nu^2}\int_{\Pi_\nu} |w|^2\qquad \text{for all } w\in H^1_0(\Pi_\nu)\,,
	\]
	we then obtain
	$$
	\aligned
	\t a(\t u, \t u) &= \int_\Scal\na (\t m^{-\ga} \t u)\cdot \na (\t m^{\ga} \t u)  \geq \left( 1- \frac{4\nu^2\ga^2}{\pi^2}\right)\int_{\Scal} |\na \t u|^2  .
	\endaligned
	$$
	(Note that if we do not have  $\Scal\subset\Pi_\nu$, the fine tuning of the weight~$\t m$ is essential; otherwise the interval of allowed~$\ga$ would be strictly smaller, which  would be a problem if we want to consider angles $2\nu$ close to~$\pi/2$ in our local wellposedness theorem.)
	
	Now that we have coercitivity for the form~$\t a$, the existence of weak solutions for the equation
	\[
	\t a(\t u,w)=\int_{\Scal}  F w=\int_{\Scal}  f (\t m^\ga w) \qquad \text{for all }w\in H^1_0(\Scal)
	\] 
	follows from the Lax--Milgram theorem, with $F:= \t m^{\ga} f$. This easily implies that $u$ is a weak solution to $-\Delta  u = f$. Elliptic estimates for standard Sobolev spaces~\cite{ChSh17} then show the higher regularity of~$\t u$ in~$H^{k+1}(\Scal)$, which translates into higher regularity for~$u$ in~$H^{k+1,\gamma}(\Scal)$.
\end{proof}

\begin{proof}[Proof of Proposition \ref{prop:maz}] 
For concreteness, consider the case $i = 0$, as the general case is analogous. Then, given $f\in \Lcal_{2,\ga}(\Om)$, using the small ball~$B$ centered at the origin of Subsection~\ref{ss.cov}, we take a smooth, radially symmetric cutoff~$\chi$ identically equal to~$1$ on~$B $ and to~$0$ outside~$2B$. We decompose 
$$
f = \chi f + (1-\chi) f =: f_1 + f_2.
$$ 

We first look for the solution $u_1$ of a boundary problem of the form~\eqref{eq:systemOmPoisson} with source $f_1$:
\[
-\Delta u_1=f_1,\qquad u_1|_{\Ga}=0\,.
\]
Applying the change of variables $\dif : \Scal \rightarrow \Om$, this is equivalent to finding a solution $U_1$ of the boundary problem
\be\label{E.defU1Delta}
(-\Delta + L) U_1 =  F_1, \qquad F_1 := f_1\circ \dif | \det D\dif | \in \Lcal_{2,\ga - 1}(\Scal).
\ee
Here $L$ is the (not necessarily elliptic) second order differential operator
\[
L U := -\mathrm{div}(\Lbb \na U),
\]
with $\Lbb$ given by~\eqref{eq:Lbb}. To pass to this equation, we have used  \eqref{eq:poissonH}, \eqref{eq:det} and \eqref{eq:trafo}. 

By the regularity properties \eqref{eq:regTheta}, it is easy to see that $L$ defines a continuous operator 
$$
L : H^{2,\ga-1} (\Scal) \rightarrow \Lcal_{2,\ga-1}(\Scal)
$$
with operator norm bounded in terms of the $E_{k,\beta}$-norm of~$z$. 

As the coefficients of~$\Lbb$ and their derivatives decay exponentially as $\t x \rightarrow -\infty$, let us take another smooth cutoff~$\eta$ such that $\eta(\t x, \t y) = 1$ for $\t x < \t x_0$ and $\eta(\t x, \t y) = 0$ for $\t x > \t x_0/2$ where $\t x_0 < 0$ is negative enough that the $C^1$-norm of $\eta \Lbb $ is smaller than a certain small constant~$\delta$. Then the coercitivity argument used to prove Lemma~\ref{lem:laplace-strip} still applies if one replaces $-\Delta$ by $-\Delta + \eta L$, yielding the solution $U_* \in H^{2,\ga - 1}_0(\Scal)$ to the equation
$$
(-\Delta + \eta L) U_* = F_1.
$$

To construct the function~$U_1$ as above, we write $U_1 =: U' + U_*$, where $U'$ must satisfy 
$$
(-\Delta + L) U' = -(1-\eta) U'
$$
with zero Dirichlet boundary condition.
Going back to $\Om$, this is equivalent to solving a problem of the kind~\eqref{eq:systemOmPoisson}, that is
\be\label{E.Deltau1}
-\Delta u'= f',\qquad u'|_{\Ga}=0
\ee
in~$\Om$, where the source term $f'$ is now identically~$0$ in a neighbourhood of the corner point, say in a small ball~$B'$ centered at~0. Furthermore, it also remains to find the solution $u_2$  of a problem of the form~\eqref{eq:systemOmPoisson} with a source term $f_2 $, which has the same structure. Therefore, we can use the same argument in both cases.

In order to solve~\eqref{E.Deltau1}, we just need to use that $\Om$ satisfies weighted Poincar\'e inequality  
\be\label{eq:poincare-weighted}
\int_{\Om} \frac{|w(x,y)|^2}{x^2+y^2} \,dx \, dy\lesssim \int_{\Om} |\na w(x,y)|^2 dx \,dy
\ee
for $w\in \Ccont^\infty_0(\Om)$ (and therefore for the closure of this space in~$\Lcal^1_{2,0}(\Om)$, cf.~\eqref{eq:sobolev-2d}). Before proving this, note that this inequality does not suffice to solve the problem for~$u_1$, as the weights in this inequality do not match what we need. However, the fact that $f'$ vanishes in a neighborhood~$B'$ of the origin, which can be controlled in terms of the $E_{k,\beta}$-norm of~$z$, makes the presence of those weights irrelevant.

To prove~\eqref{eq:poincare-weighted}, we can assume that $\Om$ is contained in some conic sector~$S_{\nu_0}$, described in polar coordinates as
\[
S_{\nu_0}:=\{r>0,\; |\vartheta|<2\nu_0\}.
\]
We can take~$\nu_0$ to be any number larger than~$\nu$ (in particular, smaller than $\frac\pi2$). The reason is that, although indeed~$\Om$ does not need to be contained in~$S_{\nu_0}$, one can take a small neighborhood of the origin, which one can assume to be~$B$, and a diffeomorphism $h':\Om\to \Om'$ with the desired bounds, which is the identity on $\Om\cap B$ (that is, it does not change the domain near the corner) but ensures that $\Om'\subset S_{\nu_0}$.

 Proving~\eqref{eq:poincare-weighted} is now trivial: any $w\in C^\infty_c(S_{\nu_0})$ satisfies this inequality with a constant that depends on~$\nu_0$. Indeed, with the change of variables $s:=\log r$, this is just the standard Poincare inequality on a strip. Denoting by $w(r,\vartheta)$ and $w(s,\vartheta)$ the expression of the function~$w$ in these coordinates with some abuse of notation,  we just note that this follows from the identities
 \begin{align*}
 \int_{-\nu_0}^{\nu_0}\int_0^\infty (w_r^2 + r^{-2}w_\vartheta^2)\, r\, dr\, d\vartheta & =  	 \int_{-\nu_0}^{\nu_0}\int_{-\infty}^\infty {w_s^2 + w_\vartheta^2}  \,ds\, d\vartheta ,\\
\int_{-\nu_0}^{\nu_0}\int_0^\infty r^{-2} w^2\,  r\, dr\, d\vartheta & =  	 \int_{-\nu_0}^{\nu_0}\int_{-\infty}^\infty {w^2} \, ds\, d\vartheta .
 \end{align*}
 The existence and regularity for~\eqref{E.Deltau1} is now standard and follows as in the proof of Lemma~\ref{lem:laplace-strip}. Note that the weak solution~$u_1$ is only guaranteed to exist $\Lcal^1_{2,0}(\Om)$, and similarly with regularity higher estimates, but this norm of~$u_1$ is equivalent to its $\Lcal^2_{2,1}(\Om)$-norm due to the fact that it vanishes on~$B'$.
 \end{proof}

\begin{proof}[Proof of Corollary~\ref{C.reg-v}]
Writing $u=: b_0+ b_1z+ \t u$, we infer that $\t u$ satisfies the boundary problem~\eqref{eq:systemOmDirichlet} with boundary datum~$\t\psi$. The existence of~$\t u$ then follows from Proposition~\ref{prop:dirichletInt}. The fact that is bounded in~$C^{1,\lambda}$ is elementary away from the corner; near the corner, it follows directly from the regularity of the function~$\t u \circ h^{-1}$ on~$\Scal$ and the asymptotic behavior of the derivatives of the transformation $\dif^{-1}:\Om\to\Scal$, cf.~\eqref{formulahtilde}.
\end{proof}

\subsection{The harmonic extension of certain second order polynomial}
\label{ss.special_sol}

To conclude, in this subsection we present a concrete calculation that we use to capture the leading order asymptotic behavior of~$\si$ near the cusp. As we have set $z_*=0$, the particular choice of the boundary datum corresponds to the leading terms in Equation~\eqref{eq:psi-asymptotic}. We shall state it for concreteness in the case $k=2$; the case of higher~$k$ is analogous.


	\begin{lemma}\label{lem:special_sol} Let $\left|\beta - \frac{5}{2}\right| < \frac{\pi}{2\nu}$ and let 
	$$
	\psi = \Re\left(c_0 + c_1 z\right) + c_\pm |z|^2 , \quad z\in\Ga
	$$
	for some $c_0, c_1\in \Cbb$ and $c_\pm\in \Rbb$. Then, the normal derivative of its harmonic extension $u$  to~$\Om$ is of the form
	$$
	\pa_n u =  -\Im \left(c_1 z_{s} \right) + \frac{2|z|}{\sin 4\nu}\left(c_\pm \cos 4\nu - c_\mp\right)  +  \Lcal^{2 + \frac{1}{2}}_{2,\beta+ \frac{1}{2}}(m).
	$$
	\end{lemma}

	\begin{proof}
	It is elementary that the contribution of the summand $\Re(c_0 + c_1 z)$ to $\pa_n u$ is $-\Im \left(c_1 z_{s}\right)$, so let us focus on the remaining term, which we write as  
	$$
	\psi_1 =: c_\pm \chi(r) r^2 + f.
	$$ 	
	Here $\chi$ is a radially symmetric cutoff that is equal to~$1$ on some small neighborhood of the origin. Note that $f$ is smooth and supported away from the origin. In particular, by Proposition \ref{prop:dirichletInt}, the contribution of~$f$ to $\pa_n u$ is in $\Lcal^{2 + \frac{1}{2}}_{2,\beta + \frac{1}{2}}(\Ga)$. 
	
	The nontrivial part is to estimate the remaining term, that is, the contribution of~$c_\pm \chi(r) r^2$. By means of the change of variables $\dif $ introduced in Subsection~\ref{ss.cov}, the problem reduces to estimate the solution~$U$ to the boundary value problem in~$\Scal$ given by
	$$
	(-\Delta + L)U = 0, \qquad U|_{\pa\Scal}= c_\pm \chi(e^{\t x})e^{2\t x}.
	$$
	
	We start by considering the solution~$\zeta_0$ to the analogous problem on the whole strip~$\Pi_\nu$, which one can equivalently write as
	$$
	\Delta \zeta_0 = 0, \qquad \zeta_0(\t x, \pm \nu) = c_\pm e^{2\t x}. 
	$$
	Setting $\zeta_0(\t x, \t y):= e^{2 \t x}p(\t y)$, we arrive at the ODE
	$$
	p'' + 4 p = 0, \qquad p(\pm \nu) = c_\pm,
	$$
	which is uniquely solvable because $2\nu \neq {l\pi}/{2}$ with $l\in \Zbb$ (in which case there are also non-trivial solutions to the homogeneous equation). Specifically, the solution is
	$$
	p(\t y) = a_+ e^{2i\t y} + a_- e^{-2i\t y}, \qquad a_\pm := \frac{c_\pm e^{2i\nu}- c_\mp e^{-2i\nu}}{e^{4i\nu} - e^{-4i\nu}}.
	$$
	
	We now write $U = \zeta_1 + \chi \zeta_0$. The function $\zeta_1$ then satisfies
	$$
	\aligned
	(-\Delta + L)\zeta_1 &= -\zeta_0\Delta \chi - \nabla\zeta_0 \cdot\nabla\chi + L(\chi \zeta_0), \qquad \zeta_1 |_{\pa\Scal} = 0\\
	&=:\chi L\zeta_0 + F ,
	\endaligned
	$$
	where $F$ is identically $0$ for all sufficiently negative $\t x$. It is not difficult to see that the r.h.s. belongs to $H^{2,-(\lambda_\beta + 2)}(\Scal)$, so following the proof of Proposition \ref{prop:maz} we obtain the existence of a unique solution $\zeta_1 \in H^{4,-(\lambda_\beta + 2)}(\Scal)$, where we recall  that $\lambda_\beta := \frac{1}{2} - \beta$. By composing with the change of variables and using the asymptotics for the derivatives of~$\dif$, Equation~\eqref{eq:trafo},  we obtain that  $ \zeta_1 \circ \dif ^{-1} \in \Lcal^{4}_{2, \beta + \frac{1}{2}}(\Om)$. It therefore contributes a term in $\Lcal^{2 + \frac{1}{2}}_{2,\beta + \frac{1}{2}}(\Ga)$ to the normal derivative of $u$. 
	
	It only remains to compute the contribution of $\zeta_0$ to the normal derivative. As	
	$$
	p'(\pm \nu) = \pm\frac{2}{\sin 4\nu}\left(c_\pm \cos 4\nu - c_\mp\right),
	$$ 
	the normal derivative on~$\pa\Pi_\nu$ is $\pa_n \zeta_0(\t x, \pm\nu) = \frac{2}{\sin 4\nu}\left(c_\pm \cos 4\nu - c_\mp\right)$. Back on the original domain~$\Om$, by~\eqref{eq:normalH} the normal derivative in the neighborhood of the corner tip reads 
	$$
	\pa_n (\zeta_0 \circ \dif ^{-1}) \circ \dif  = \frac{1}{|D \dif  \v e_1|}(\Abb \nabla \zeta_0) \cdot \v n, \ = \frac{1}{|D \dif  \v e_1|}\left[\pa_n \zeta_0 + (\Lbb \nabla\zeta_0) \cdot \v n\right].
	$$ 
	As $|D \dif  \v e_1| = e^{\t x}$, we get
	$$
	\pa_n (\zeta_0 \circ \dif ^{-1})= \frac{2r}{\sin 4\nu}\left(c_\pm \cos 4\nu - c_\mp\right) + \Lcal^{2 + \frac{1}{2}}_{2,\beta + \frac{1}{2}}(\Ga).
	$$ 
	The lemma then follows.
	\end{proof}

\section*{Acknowledgements}

The authors are deeply indebted to Erik Wahl\'en for pointing out how the ODE system~\eqref{E.ODEintro} blows up, to Siddhant Agrawal and Sijue Wu for very valuable discussions, and to an anonymous referee for very useful suggestions and corrections. This project has received funding from the European Research Council (ERC) under the European Union's Horizon 2020 research and innovation program through the grant agreements 788250 (D.C.) and 633152 (A.E.). This work is also supported by the grant PID2022-136795NB-I00 of the Spanish Science Agency (A.E.) and by the ICMAT-Severo Ochoa grant CEX2019-000904-S.


\begin{appendices}

\section{Some auxiliary estimates}
\label{A.A}

In this Appendix we collect some auxiliary estimates that we will need throughout the paper. Unless  mentioned otherwise, the weight~$m$ is always as in the energy estimates (that is, smooth, positive on the bounded open interval under consideration ---typically $I:=(-\pi,\pi)$---, and vanishing at both endpoints with nonzero derivative).

\subsection{Hardy inequalities}

Let us start by recalling the general form of the $L^2$-based Hardy inequalities (see e.g. \cite{KufOpic}):

\begin{theorem}[Hardy's inequality]\label{thm:hardy}
	Let $-\infty \leq a < b \leq \infty$ and let $m_1, m_2$ be measurable, positive functions on $(a,b)$. Then, we have 
	\be\label{eq:hardyI}
	\int_a^b\Big(\int^x_a f(s)ds\Big)^2m_1(x)^2dx \,\leq \, C \int_a^b|f(x)|^2 m_2(x)^2dx
	\ee
	for some~$C$ and all~$f$ if and only if 
	$$
	\sup_{a<x<b}\Big(\int_x^b m_1(s)^2ds\Big)\Big(\int^x_a m_2(s)^{-2}ds\Big) \, < \, \infty.
	$$
	Similarly, the dual inequality 
	\be\label{eq:hardyII}
\int_a^b\Big(\int^b_x f(s)ds\Big)^2m_1(x)^2dx  \,\leq \, C \int_a^b|f(x)|^2 m_2(x)^2dx
	\ee
	holds for some~$C$ and  all~$f$ if and only if
	$$
	\sup_{a<x<b}\Big(\int_a^x m_1(s)^2ds\Big)\Big(\int^b_x m_2(s)^{-2}ds\Big) \, < \, \infty.
	$$	
\end{theorem}

In the proofs we will use Hardy's inequality with the following formulation, directly in terms of weighted Sobolev spaces:

\begin{lemma}\label{lem:hardy}
	Let $p>1$ and let $I_\de:= (0,\de)$ with $\de < \infty$.
	\begin{enumerate}
		\item\label{itm:hardy1} Let $\ga + 1/2 < 1$. Then,  
		$$
		f\mapsto \int_0^x f(t)dt \,:\, \Lcal_{2,\ga}(I_\de) \longrightarrow \Lcal_{2,\beta}(I_\de) 
		$$
		is continuous for all $\beta \geq \ga -1$. In particular, we have
		\be\label{eq:HardyA}
		f\mapsto \frac{1}{x}\int_0^x f(s)ds \,:\, \Lcal_{2,\ga}(I_\de) \longrightarrow \Lcal_{2,\ga}(I_\de). 
		\ee
		
		\item\label{itm:hardy2} Let $\ga + 1/2 >0$. Then, 
		$$
		f\mapsto \int_x^\de f(t)dt \,:\, \Lcal_{2,\beta}(I_\de) \longrightarrow \Lcal_{2,\ga}(I_\de)
		$$
		is continuous for all $\beta \leq \ga + 1$. In particular, we have
		\be\label{eq:HardyB}
		f\mapsto \int_0^x \frac{f(s)}{s}\, ds \,:\, \Lcal_{2,\ga}(I_\de) \longrightarrow \Lcal_{2,\ga}(I_\de). 
		\ee		

	\end{enumerate}
\end{lemma}

\subsection{Commutators involving derivatives and the Hilbert transform}\label{ss.commutator}

The following lemma shows how we handle commutators involving the weight, the derivative and the Hilbert transform:

\begin{lemma}\label{lem:trHilbert}
	Let $|\ga| < \frac{1}{2}$ and let $j\leq k$ be positive integers. The Hilbert transform,
	\[
	Hf(x) := \frac{1}{2\pi}\PV \int^\pi_{-\pi} \,\cot\left(\frac{x-y}{2}\right)\,f(y)\, dy,
	\]
	is a bounded operator $\Lcal^k_{2,\ga + j}(m)\to H^k_{\ga + j}(m)$. Furthermore:
		\begin{enumerate}
		\item[i)] As long as the all derivatives of~$f$ involved are integrable, the Hilbert transform and the corresponding derivatives commute. E.g. if $f\in \Lcal^1_{2,\ga}(m)$, then $Hf \in H^1_{\ga}(m)$ and 
		$$
		\pa_x H f = H\pa_x f.
		$$
		\item[ii)] If some derivative of $f$ is not integrable, we can ``correct'' the operator including factors of~$m$ at the expense of obtaining a well behaved commutator term. E.g., if $f\in \Lcal^1_{2,  \ga  + 1}(m)$, then $Hf\in \Lcal^1_{2,  \ga  + 1}(m)$ and 
		$$
		\pa_x Hf = \frac{1}{m} H (m\pa_x f) + K f,
		$$
		where $K: \Lcal_{2,\ga}(m)\rightarrow \Lcal^1_{2,  \ga + 1}(m)$ is bounded.
	\end{enumerate}	
\end{lemma}

\begin{proof}
	Let us consider the case $k = 1$, as generalizing the argument to higher~$k$ is straightforward. Let us fix a point $x\in I:=(-\pi,\pi)$. It is worth keeping in mind that $m$~is essentially distance to the boundary of $I$, so it is easy to see that we can assume the principal value is defined via integration over $I\backslash I_\ep(x)$, with $I_\ep(x) := \{y: |y-x|<\ep m(x)\}$ and $\ep\to0^+$.
		
	One can readily see that
	$$
	\aligned
	\pa_x\int^\pi_{x + \ep m(x)} \cot\left(\frac{x-y}{2}\right)f(y)dy 
	&=-\int^\pi_{x + \ep m(x)} f(y)\pa_y\cot\left(\frac{x-y}{2}\right)dy + f(x + \ep m(x))\cot\left(\frac{\ep m(x)}{2}\right)\\
	&=-f(\pi)\cot\left(\frac{x-\pi}{2}\right) + \int^\pi_{x + \ep m(x)}  \pa_y f(y)\cot\left(\frac{x-y}{2}\right)dy,
	\endaligned
	$$
	and similarly for $[-\pi, x-\ep m(x))$.	
 	Since by assumption $f|_{\pa I} = 0$ because $f\in \Lcal^1_{2,\ga}$ with $|\ga|<\frac12$, we conclude by letting $\ep\rightarrow 0^+$ that
 	$$
 	\pa_x Hf = H\pa_x f \in \Lcal_{2,\ga}(m).
 	$$ 	
 	
Let us now suppose $f\in\Lcal^1_{2,\ga + 1}(m)$. Then $mf \in \Lcal^1_{2,\ga}(m)$ and we can write
 	$$
 	Hf = \frac{1}{m} H(mf) + \frac{1}{m}[m, H]f. 
 	$$
 	By the first part of the proof, we then have
 	$$
 	\aligned
 	\pa_x Hf &= \frac{1}{m}H\pa_x(mf) - \frac{m_x}{m} Hf + \frac{1}{m}\pa_x[m, H] f\\
 	&= \frac{1}{m} H (m\pa_x f) - \frac{1}{m}[m_x, H]f + \frac{1}{m}\pa_x[m, H] f,
 	\endaligned
 	$$
 	so we conclude $\pa_x Hf \in\Lcal_{2,\ga + 1}(m)$ using Lemma~\ref{lem:derFphi} below.
 \end{proof}

\begin{lemma}\label{lem:derFphi}
Let $\ga\in \left(-\frac{1}{2}, \frac{1}{2}\right)$, $\beta\in \left(-\frac{1}{2}, \frac{1}{2}\right]$ and let $k\geq 1$. We define $[g, H] f := g H f - H (g f)$.
 
Then, given $f\in \Lcal_{2, \ga}(m)$ and $g\in H^{k + 1}_{\beta + k}(m)$, we have
$$
\pa^i [g, H] f \in \Lcal_{2, \ga - \lambda_\beta +i}(m) , \qquad \lambda_\beta := \frac{1}{2} - \beta
$$
whenever $1 \leq i \leq k$ satisfies $\ga - \lambda_\beta + i > -\frac{1}{2}$. The estimate is also valid if we replace $\lambda_\beta$ by any $\lambda \geq 0$ provided that the derivatives of~$g$ satisfy the pointwise bound $g^{(j)} = O(m^{\lambda - j})$ for $1\leq j\leq k + 1$; or by any $\lambda \in (-1,0)$ such that $\ga - \lambda < \frac{1}{2}$. 
\end{lemma}

\begin{proof} 
	Again, take $k=1$ without any loss of generality and fix some point $x$. As the kernel of the periodic  Hilbert transform is
	\be\label{eq:cot}
	\frac{1}{2}\cot\left(\frac{x - y}{2}\right) =: \frac{1}{x - y} + h\left(\frac{x-y}{2}\right),
	\ee
	where $h$ is a smooth function and the weight is only singular at the endpoints, it is therefore enough to prove the claim in a small neighborhood of the endpoints and with the periodic kernel replaced by the real-line kernel, $1/(x-y)$.
	
	For notational simplicity, as in the proof of Lemma~\ref{lem:BRbasic}, let us consider the translated interval $I := (0, 2\pi)$ instead of $(-\pi, \pi)$. We take $x \in I_{\de}:=(0,\de)$, for some small $\de>0$. We show that, as a function of~$x$,
	$$
	\pa_x\int_I F(g)(x,y) f(y) dy \in \Lcal_{2, (\ga - \lambda_\beta) + 1}(I_\de),
	$$
	where we have set
	\be\label{eq:F}
	F(g)(x, y) := \frac{g(x)- g(y)}{x - y}
	\ee
	First note that, because we are only considering differences $g(x) - g(u)$, we may assume that actually $g\in \Lcal^2_{2,\beta + 1}(m)$ without loss of generality. Therefore, by Lemma \ref{lem:sobolev}, we have, $g(x) = O(m(x)^{\lambda_\beta})$ and $g'(x) = O( m(x)^{\lambda_\beta - 1})$, where recall that $m(x) \approx x$ near the origin.

	As in \eqref{eq:partition}-\eqref{eq:Is}, we split the interval as 
	\be\label{eq:interval}
	I = I_{2\de} \cup (I\setminus I_{2\de}) \qquad I_{2\de} =: I_l(x) \cup I_c(x) \cup I_r(x).
	\ee
	Consider $y\in I_{2\de}$ first. In this region we have $m(y) \approx y$. We can write
	\be\label{eq:Fj2B}
	\pa_x F(g)(x,y) =  \frac{ g'(x)}{x-y}\,  - \frac{g(x) - g(y)}{(x - y)^2},
	\ee
	which for $y\in I_l(x) :=\{0 < y \leq x/2\}$ implies  
	$$
	|\pa_x F(g)(x,y)| \,\lesssim\, \frac{1}{x}\,|g'(x)| + \frac{1}{x^2}\,|g(x)| + \frac{1}{x^2}\,|g(y)| \, \lesssim \, x^{\lambda_\beta - 2}.
	$$
	Therefore
	$$
	\int^\de_{0}m(x)^{2(\ga - \lambda_\beta + 1)}\Big(\int_{0}^{x/2} \pa_x F(g)(x,y) f(y)dy \Big)^2 dx \,\lesssim\, \int^\de_{0}m(x)^{2(\ga - 1)}\Big(\int_0^x |f(y)| dy \Big)^2 dx,
	$$
	which is bounded by the $\Lcal_{2, \ga}$-norm of $f$ by Hardy's inequality \eqref{eq:HardyA}. 
	
	On the other hand, for $y\in I_r(x) :=\{2x < y \leq 2\de\}$, formula \eqref{eq:Fj2B}  implies
	$$
	|\pa_x F(g)(x,y)| \,\lesssim\, \frac{1}{y}\,|g'(x)| + \frac{1}{y^2}\,|g(x)| + \frac{1}{y^2}\,|g(y)| \,\lesssim \,  \frac{x^{\lambda_\beta - 1}}{y} +  y^{\lambda_\beta - 2}
	$$
	Using the bound $g = O(x^\lambda_\beta)$, 
	$$
	\aligned
	\int^\de_{0}m(x)^{2(\ga - \lambda_\beta + 1)}\left(\int_{2x}^{2\de} \pa_x F(g)(x,y) f(y)dy \right)^2 &dx \,\lesssim\, \int^\de_{0}m(x)^{2 \ga}\left(\int_{2x}^{2\de} \frac{|f(y)| }{y}dy \right)^2 dx \\
	&+ \int^\de_{0}m(x)^{2 (\ga - \lambda_\beta + 1)}\left(\int_{2x}^{2\de} |f(y)|y^{\lambda_\beta - 2}dy \right)^2 dx.
	\endaligned
	$$
	Both are bounded by the $\Lcal_{2, \ga}$-norm of $f$ by Hardy's inequality \eqref{eq:HardyB}. 
	
	When $y\in I_c(x) :=\{ x/2 < y \leq 2x\}$, we  write 
	\be\label{eq:Fj2A}
	\pa_x F(g)(x,y) = \frac{1}{(x - y)^{2}}\int_x^{y} g''(\tau)\, \frac{(y - \tau)}{2}d\tau.
	\ee
	If we only control the weighted $L^2$-norm of $g ''$, as in the case of non-negative $\lambda\geq 0$ other than $\lambda_\beta$ discussed in the statement, we can employ the inequality
	$$
	|\pa_x F(g)(x,y)| \,\lesssim\,  m(x)^{\lambda_\beta - 1} \, \frac{x^{-1/2}}{|x - y|^{1/2}} \, \|g '' \|^2_{2, \beta + 1},
	$$
	where we have used that $m(x)\approx m(y)$ uniformly for all $y\in I_c(x)$. The corresponding integral is easily seen to be bounded in $\Lcal_{2, \ga - \lambda_\beta + 1}(I_\de)$ by Theorem \ref{thm:rieszintegral}  below. 
	If, on the other hand, we have $g''(x) = O(x^{\lambda_\beta - 2})$, then
	$$
	|\pa_x F(g)(x,y)| \,\lesssim\,  m(x)^{\lambda_\beta - 2},
	$$
	in which case the claim follows arguing as on the interval $I_l(x)$.
	
	Finally, when $y\in I\setminus I_{2\de}$, we have $|x-y| > \de$, and so formula \eqref{eq:Fj2B} implies
	$$
	|\pa_x F(g)(x,y)| \,\lesssim\,  \frac{1}{\de^2}[m(x)^{\lambda_\beta - 1} + O(1)]. 
	$$
	The claim then follows since $m(x)^{2(\ga - \lambda_\beta + 1)}$ is integrable.
	
	To conclude, let us discuss the case $-1 < \lambda < 0$. A first observation is that the condition $\ga - \lambda < \frac{1}{2}$ when $\lambda < 0$ ensures that $g f \in \Lcal_{2, \ga - \lambda}(m)\subset L^p(I)$ for some $p>1$. As the function $g'$ is generally not integrable at the origin,  we  need to modify the part corresponding to the region $I_l(x)$. In this case, we can estimate \eqref{eq:Fj2B} for $y \in I_l(x)$ as  
	
	$$
	|\pa_x F(g)(x,y)| \, \lesssim \, m(x)^{\lambda - 1} + \frac{m(y)^\lambda}{x} .
	$$
	The claim now follows from Hardy's inequality \eqref{eq:HardyA} as before, provided that $\ga - \lambda + \frac{1}{2} < 1$.

	For higher $k$, we need to show $\pa^j [g, H] f\in\Lcal_{2,\ga - \lambda + k}(m)$ for all $j\leq k$. One can proceed just as above, using the higher-order version of \eqref{eq:Fj2A} when $u\in I_c(x)$ and considering direct derivatives of \eqref{eq:F} otherwise. Note that if $(\ga - \lambda) + j_0 > -1/2$ for some $j_0 < k$, then $\pa^j [g, H] f\in\Lcal_{2,\ga - \lambda + j}(m)$ for $j_0 \leq j\leq k$. We omit further details, which are largely straightforward.
\end{proof}

\subsection{Estimates involving the fractional Laplacian}
\label{SS.La}

Let us start by recalling a standard result on the boundedness of the Riesz potential on weighted Lebesgue spaces with power weights, which can be found e.g.~in~\cite{SW}:

\begin{theorem}\label{thm:rieszintegral} Let $\ga \in\left(0,\frac{1}{2}\right) $, so that both $\ga$ and $\ga - \frac{1}{2}$ give rise to Muckenhaupt weights. Then
$$
\int_{-\infty}^\infty \Big(\int_{-\infty}^\infty \frac{f(u)}{|x-u|^{1/2}}du\Big)^2 |x|^{2(\ga -\frac12)}dx \, \lesssim \, \int_{-\infty}^\infty |f(x)|^2 |x|^{2\ga} dx.
$$
\end{theorem}

Armed with this result, we are ready to present several lemmas involving fractional derivatives that we will use repeatedly:

\begin{lemma}\label{lem:halfDer}
Assume $\ga \in\left(0,\frac{1}{2}\right) $. Then, we have
\be\label{ineq:halfd2}
\|\Lambda^{1/2}f\|_{2,\ga - \frac{1}{2}}\,\lesssim \, \|f'\|_{2,\ga}.
\ee
\end{lemma}

\begin{proof}
It is enough to show that, as a function of~$x$,
$$
 \PV \int_I \frac{f(x) - f(y)}{|x - y|^{3/2}}\,dy = \lim_{\epsilon\rightarrow 0} \int_{|y-x|>\epsilon m(x)} \frac{f(x) - f(y)}{|x - y|^{3/2}}\,dy\in \Lcal_{2,\ga-\frac{1}{2}}(I). 
$$
So let $f\in H^1_{\ga}(m)\cap \Ccont^1(\overline{I})$. Integration by parts leads to
$$
\aligned
2\int^{\pi}_{x + \ep m(x)} \frac{f(x) - f(y)}{|x - y|^{3/2}}\,dy &= \int^{\pi}_{x + \ep m(x)} \big(f(y) - f(x)\big)\pa_y (y-x)^{-1/2}\,dy \\
& = \frac{f(\pi) - f(x)}{(\pi - x)^{1/2}} - \frac{f(x+\ep m(x)) - f(x)}{(\ep m(x))^{1/2} } -\int^{\pi}_{x + \ep m(x)}\frac{f'(y)}{|x - y|^{1/2}}\,dy.
\endaligned
$$
Note that Holder's inequality yields the bound
$$
\aligned
\left|\frac{f(x + \ep m(x)) - f(x)}{\sqrt{\ep m(x)}}\right| &\, \lesssim \, \frac{1}{\sqrt{\epsilon m(x)}}\int_x^{x+ \ep m(x)}|f'(s)|ds \\
&\, \lesssim \,  \, \left(\int_x^{x + \ep m(x)}m(s)^{2\ga}|f'(s)|ds\right)^{1/2} \left(\frac{1}{\ep m(x)}\int_x^{x + \ep m(x)}m(s)^{-2\ga}ds\right)^{1/2}\\
&\, \lesssim \,   m(x)^{-\ga}\|f'\|_{\Lcal_{2,\ga}((x, x + \ep m(x)))},
\endaligned
$$
which vanishes as $\ep\rightarrow 0$. 

Using the analogous formulas for the interval $[-\pi, x-\ep m(x))$ , we conclude
$$
\PV  \int_I \frac{f(x) - f(y)}{|x - y|^{3/2}}\,dy  = \frac{1}{(\pi - x)^{1/2}}\int^\pi_x f'(y)dy - \frac{1}{(\pi + x)^{1/2}}\int_{-\pi}^x f'(y)dy -\int_{-\pi}^{\pi}\frac{f'(y)}{|x-y|^{1/2}}\, dy.
$$
The first two integrals are bounded in $\Lcal_{2,\ga + \frac{1}{2}}(I)$ by Hardy's inequality, while the last one follows from Theorem~\ref{thm:rieszintegral}. The claim for general $f\in H^1_\ga(I)$, then follows by density, see e.g. \cite{Kufner}. 
\end{proof}

\begin{lemma}\label{lem:comm_gb}
Let $\ga \in \left(-\frac{1}{2},\frac{1}{2}\right)$ and $\lambda \leq 1$. If $f \in \Lcal_{2,\ga-1/2}(m)$, $ g = O(1)$, and  $g' = O(m^{-\lambda})$, then
$$
[\Lambda^{1/2}, g]f\in  \Lcal_{2, \ga}(m).
$$
\end{lemma}

\begin{proof} 
As before, we can consider the shifted interval $I := (0,2\pi)$. We take $x\in I_\de := (0,\de)$ for some small $\de>0$ and we divide $I$ in intervals depending on~$x$ as in \eqref{eq:interval}. It is enough to prove that
$$
 \PV \int_{I} \frac{g(x) - g(y)}{|x -y|^{3/2}} \, f(y) d y \in \Lcal_{2,  \ga}(I_\de).
$$

We clearly have
$$
\aligned
&\left|\int_{I_l(x)} \frac{g(x) - g(y)}{|x -y|^{3/2}} \, f(y) d y\right|   \, \lesssim \,  x^{- 3/2} \int_{0}^x |f(y)| dy, \\
&\left|\int_{I_r(x)} \frac{g(x) - g(y)}{|x -y|^{3/2}} f(y) d y\right| \, \lesssim \, \int_{x}^{2\de} \frac{|f(y)| }{y^{3/2}}\, dy,
\endaligned
$$
which are bounded in $\Lcal_{2,  \ga}(I_\de)$ by Hardy's inequality. 

When $y\in I_c(x)$ and  $\ga \in\left(-\frac{1}{2}, 0\right)$, we have
$$
\left|\, \PV\int_{I_c(x)} \frac{g(x) - g(y)}{|x -y|^{3/2}} \, f(y) d y\right| \, \lesssim \, \int_{I_c(x)} \frac{|f(y)|y^{-1}}{|x - y|^{1/2}} \, d y,
$$
where we have used $g' =O(m^{-\lambda} ) = O(m^{-1})$ since, by assumption $\lambda \leq 1$. The claim then follows from Theorem \ref{thm:rieszintegral}. 

In the case $\ga \in \left(0, \frac{1}{2}\right)$, we just need to replace $x^\ga y^{-1}$ by $x^{\ga - 1/2}y^{-1/2}$.
\end{proof}

\begin{lemma}\label{lem:comm_weight}
Let $\ga \in \left(-\frac{1}{2},\frac{1}{2}\right)$. If $f\in \Lcal_{2,\ga}(I)$, $ g = O(m^{2\ga})$ and $g'=O(m^{2\ga})$, then
$$
 [\Lambda^{1/2}, g]f\in  \Lcal_{2, -\ga + \frac{1}{2}}(I).
$$
\end{lemma}

\begin{proof} As before, let $I := (0,2\pi)$, take $x\in I_\de := (0,\de)$ for some small $\de>0$, and divide $I$ as in \eqref{eq:interval}. Consider the case $y\in I_l(x)$ first. Then 
$$
x^{-\ga + \frac{1}{2}}\left|\int_{I_l(x)} \frac{g(x) - g(y)}{|x - y|^{3/2}} \, f(y) d y\right| \, \lesssim \,  x^{\ga - 1}\int_{0}^{x} |f(y)| d y + x^{-\ga-1}\int_{0}^{x} |y|^{2\ga}|f(y)| d y,
$$
where both terms are square integrable over $(0, \de)$ by Hardy's inequality. 

When $y\in I_r(x)$, we have
$$
\aligned
x^{-\ga + \frac{1}{2}}\left|\int_{I_r(x)} \frac{g(x) - g(y)}{|x -y|^{3/2}} f(y) d y\right| \, \lesssim \, x^{\ga}&\int_{x}^{\de} \frac{|f(y)|}{y}\, dy + x^{-\ga}\int_{x}^{\de} \frac{y^{2\ga}|f(y)|}{y}\, dy.
\endaligned
$$
Again, both are square integrable over $(0, \de)$ by Hardy's inequality.

Finally, when $y\in I_c(x)$. Suppose $\ga \in\left(-\frac{1}{2}, 0\right)$, so $\ga$ and $\ga + 1/2$ satisfy the Muckenhaupt $A_2$-condition. Then
$$
x^{- \ga + \frac{1}{2}}\left|\, p.v\int_{I_c(x)} \frac{g(x) - g(y)}{|x -y|^{3/2}} \, f(y) d y\right| \, \lesssim \, x^{\ga}\int_{I_c(x)} \frac{1}{|x - y|^{1/2}} \, |f(y)| y^{-1/2}d y,
$$
where we have used that $g' = O(x^{2\ga - 1})$. This is square integrable over $(0, \de)$ by Theorem~\ref{thm:rieszintegral}. The same is true if $\ga\in\left(0, \frac{1}{2}\right)$, and in this case we do not need the extra factor $y^{-1/2}$ in the integral. 

Finally, away from~$x$, i.e. if $y\in I\setminus I_{2\de}$, the result is obvious because $|x - y| > \de$.
\end{proof}

\subsection{Variable-step convolution operators}\label{ss.convolution}

Here we discuss some properties of the variable-step convolution operators that we use to regularize our equations in the proof of Theorem~\ref{thm:main}. Thus we shall consider here the functions and operators $\phi_{\de\eta(x)}(y)$, $B_\de$, $B_\de^*$ and $A_\de$ defined in~\eqref{E.defphide}--\eqref{E.defAde}.

So let $f\in L^1_{\text{loc}}(I)$. For $\de > 0$, it follows from the formula
$$
B_\de f(x) := (\phi_{\de \eta(x)} \ast f)(x) := \int  \frac{1}{\de \eta(x)}\phi\left(\frac{x-y}{\de \eta(x)}\right) f(y)dy
$$
that, given $x\in I$, the interval of integration $B_{\de\eta(x)}(x) = (x-\de \eta(x), x + \de \eta(x))$ is contained in $I$ and has positive distance to $\pa I$ for all sufficiently small $\de > 0$. In particular, the integral is well-defined and $m(x) \approx m(y)$ uniformly for $y\in B_{\de\eta(x)}(x)$. 

In the case of the adjoint,
$$
B_{\de}^*g(y) := \int \phi_{\de \eta(x)}(x - y)g(x)dx, 
$$
for fixed $y\in I$, the integral runs over $[x^-(y), x^+(y)]\subseteq I$, where $x^\pm(y)$ are the solutions of $x \mp \de \eta(x) = y$, respectively. These are well defined for all $\de<1 / \sup |\eta'| $). It is not difficult to see that  $x^\pm(y) \rightarrow \pa I$ as $y\rightarrow \pa I$ (and they have the same limit). 

We shall next prove two technical results about these operators that play an important role in the proof of Theorem~\ref{thm:main}:


\begin{lemma}\label{lem:B_delta} 
For any $j\geq0$, all~$\ga\in\Rbb $ and all $\de>0$, the operator $B_\de : \Lcal^j_{2, \ga + j}(I) \to  \Lcal^j_{2, \ga + j}(I)$ is continuous, Furthermore, it preserves growth rates near~$\pa I$ in the sense that, for any $f\in \Lcal^j_{2, \ga + j}(I)$,
	$$
	B_\de f\in \Ccont^\infty(I) \cap \Lcal^{j'}_{2,\ga' + j'}(I)
	$$
	for all $\ga'\in\Rbb$ and all $j'\in\Nbb $, and it approximates~$f$ as
	$$
	\lim_{\de\rightarrow 0}\|B_{\de} f -  f\|_{\Lcal^j_{2, \ga + j}(I)} = 0.
	$$
	If $f\in \Lcal^1_{2,\ga + 1}(I)$, then
	$$
	(B_\de f)' = B_\de(f') + K_\de f,
	$$
	where $K_\de: \Lcal_{2,\ga}(I) \rightarrow \Lcal_{2,\ga + 1}(I)$ is bounded, and one has the quantitative estimate
	\be\label{eq: BdeII}
	\|B_{\de} f -  f\|_{\Lcal_{2, \ga}(I)} \ \lesssim \ \de^{1/2} \, \|f\|_{\Lcal^{1}_{2,\ga + 1}(I)}.
	\ee
	Analogous estimates hold for the case of higher~$j$. Furthermore, $B_\de^*$ and~$A_\de$ enjoy the same properties.
\end{lemma}

\begin{proof} 
	We consider the case $j=0$, that is, $f\in\Lcal_{2,\ga}(I)$. 	The estimates for higher $j$ follow from very similar arguments.
	
 To  show that $B_\de (f)\in\Lcal_{2, \ga}(m)$, we write 
	$$
	\aligned
	|B_\de f(x)|^2 &\, \lesssim \, \int_{B_{\de\eta(x)}(x)} \phi_{\de\eta(x)}(x-y) m(y)^{-2\ga} dy \int_{B_{\de\eta(x)}(x)} \phi_{\de\eta(x)}(x-y) |f(y)|^2 m(y)^{2\ga} dy\\
	&\,\lesssim \, m(x)^{-2\ga}\int_{B_\de(x)} \phi_{\de\eta(x)}(x-y) |f(y)|^2 m(y)^{2\ga} dy, 
	\endaligned
	$$
	where we have used that $m(x) \approx m(y)$ uniformly for $|y-x|\leq \de\eta(x)$. Multiplying both sides by $m(x)^{2\ga}$ and integrating, and then using that $ B^*_\de (1) \lesssim  1$, we conclude that $B_\de$ is bounded in~$\Lcal_{2, \ga}(m)$. Since $\Ccont_c(I)$ is dense in $\Lcal_{2, \ga}(I)$ (see e.g.~\cite{Kufner}), we have $B_\de(f) \rightarrow f$ in $\Lcal_{2, \ga}(I)$ as $\de\to 0$. Although we shall not need this fact, for $f\in\Ccont(\overline{I})$, we actually have $\sup |B_\de f - f| \rightarrow 0$ as $\de \rightarrow 0$. It is not difficult to see that $B_\de f$ are smooth in $I$ and that taking $j$ derivatives on $\phi$ leads to factors bounded as $O((\de m)^{-j})$. All these facts hold for $B_\de^*$ and $A_\de$ as well.
	
	Let us now consider the commutator with a derivative and $B_\de$ (or $B_\de^*$). With $f\in \Lcal^1_{2,\ga + 1}(I)$, we write
	$$
	\aligned
	(B_\de f)'(x) &= \pa_x \int  \phi\left(v\right) f\left(x+\de v\eta(x)\right)dv \\
	&=\int  \phi\left(v\right) f'\left(x+\de v\eta(x)\right)(1+\de v \eta'(x))dv \\  
	&= \int \phi\left(v\right) f'\left(x+\de v\eta(x)\right)dv - \frac{\eta'(x)}{\eta(x)}\int \pa_v (v\phi\left(v\right))  f\left(x+\de v\eta(x)\right) dv  \\ 
	&=: B_\de(f') +  K_\de f,
	\endaligned
	$$ 
	where $K_\de$ is bounded in $\Lcal_{2,\ga }(I)\to \Lcal_{2,\ga + 1}(I)$ by the same argument as above. Similarly, 
	$$
	\aligned
	B_\de^* (g)'(y) 
	&=  \int  \phi\left(v\right) g\left(x(y, v)\right)\pa_y x \, dv \\
	&= \int  \phi\left(v\right) \left(g'(x(y, v))(\pa_y x)^2 + g(x(y, v))\pa_y^2 x\right) dv \\
	&= \int  \phi\left(v\right) g'(x(y, v))\pa_y x dv + \int \left(\pa_v\left(\phi\left(v\right)\frac{\pa_y x}{\pa_v x}(\pa_y x - 1)\right) + \phi(v)\pa_y^2 x\right) g(x(y, v)) dv \\
	&=: B_\de^* (g') + \t K_\de(g),
	\endaligned
	$$
	where note that $\frac{\pa_y x}{\pa_v x}(\pa_y x - 1) = \frac{\eta'(x)}{\eta(x)}v\pa_y x$. The estimate for the commutator of $A_\de$ with the derivative follows directly from these results and the formula
	\be\label{E.commAdepa}
	[A_\de, \pa] = [B_\de^*, \pa]B_\de  + B_\de^*[B_\de, \pa].
	\ee
	The result then follows.
	\end{proof}

Finally, in this lemma we consider commutators between variable-step convolutions and the Hilbert transform:

\begin{lemma}\label{lem:BdeH} Let $\ga\in\left(-\frac{1}{2},\frac{1}{2}\right)$. Then $[ B_\de, H]f$ is a smooth function and the commutator $[ B_\de, H]$ is bounded on $\Lcal^k_{2,\ga + k}(I)$ for any $k\geq 0$. The same result holds for the commutators of $H$ with~$B_\de^*$ and~$A_\de$. Moreover, if~$g$ is a function bounded as $g' = O(m^{\lambda})$ for some $\lambda \in\Rbb$, then for $f\in\Lcal_{2,\ga}(m)$ and any $\ga\in\Rbb$, we have
$$
[g, \, B_\de] f \in  \Lcal^{1}_{2, \ga - \lambda}(m), \qquad [g, \, B_\de^*] f \in  \Lcal^{1}_{2, \ga - \lambda}(m).
$$	
\end{lemma}

\begin{proof}
	Suppose that $f\in\Lcal_{2,\ga}(I)$ is smooth. Then we can write
	$$
	\aligned
	B_\de (Hf)(x) &=  \int  \phi\left(v\right) Hf\left(x+\de v\eta(x)\right)dv \\
	&= \frac{1}{\pi}\int  \phi\left(v\right)\PV  \int \cot\left(\frac{x' - y'}{2}\right)f(y')dy' dv, \qquad x':= x+\de v\eta(x),\\
	H B_\de f(x) &= \frac{1}{\pi}\PV \int \cot\left(\frac{x - y}{2}\right)\int \phi(v) f(y + \de v\eta(y))dv \,dy \\
	&= \frac{1}{\pi}\int \phi(v) \, \PV \int \cot\left(\frac{x - y}{2}\right)f(y')dy \,dv, \qquad y':= y+\de v\eta(y).
	\endaligned
	$$
	Therefore,
	$$
	[B_\de,H]f(x) = \frac{1}{\pi}\int  \phi\left(v\right)\, \int \left[\cot\left(\frac{x'(x,v) - y'(y, v)}{2}\right)\frac{\pa y'}{\pa y} - \cot\left(\frac{x - y}{2}\right)\right]f(y'(y, v)) dy \,dv. 
	$$
	Note that the terms in brackets is actually smooth:
	$$
	\frac{1}{x' - y'}\frac{\pa y'}{\pa y} - \frac{1}{x - y} = \frac{\de v \frac{1}{(x - y)} \left[\eta'(y) -  \frac{\eta(x) - \eta(y)}{x-y} \right] }{1 + \de v \frac{\eta(x) - \eta(y)}{x-y} }.
	$$
	This is a smooth function of $x$, uniformly bounded in terms of~$\eta$ and its derivatives. By a density argument, the formula is true for general $f$ as well. 
	
The argument for the commutator with $B_\de^*$ is similar, using that
	$$
	\aligned
	B_\de^* Hf(y) 
	&= \int \frac{1}{\de\eta(y')}\phi\left(\frac{y' - y}{\de \eta(y')}\right)f(y')dy'\\
	&=  \int  \phi\left(v\right) Hf\left(y'(y, v)\right)\pa_y y' \, dv,  \qquad y' - \de v \eta(y') = y, \quad \pa_v y' = \de \eta(y') \pa_y y', \\
	&= \frac{1}{\pi}\int  \phi\left(v\right)\PV  \int \cot\left(\frac{y'(y, v) - s'}{2}\right)\pa_y y'(y,v) f(s')ds' dv, \\
	H B_\de f(y) &= \frac{1}{\pi}\PV \int \cot\left(\frac{y - s}{2}\right)\int \phi(v) f(s'(s, v))\pa_s s' dv ds, \qquad s= s'-\de v\eta(s'), \\
	&= \frac{1}{\pi}\int \phi(v) \, \PV \int \cot\left(\frac{y - s}{2}\right)f(s')ds' dv,
	\endaligned
	$$
	and the claim for the commutator with~$A_\de$ follows from a formula analogous to~\eqref{E.commAdepa}. The details are straightforward.	
	
	For the commutator with $g$, we use the formula 
	$$
	\pa_x  [ g, \, B_\de](f)(x) = \int_{B_{\de\eta(x)}(x)} \pa_x\phi_{\de\eta(x)}(x- y)(g(x) - g(y)) f(y)dy + g'(x)B_\de(f)(x).
	$$
	As we are integrating over a set where $m(x) \approx m(y)$, the claim trivially follows from the fact that
	$$
	(x - y)\pa_x\phi_{\de\eta(x)}(x- y) = O(1), \qquad \frac{|g(x) - g(y)|}{|x - y|} = O(m(x)^{\lambda}).
	$$
	The lemma is then proven.
\end{proof}

\section{Blowup for the ODE system}
\label{A.ODE}

Our objective in this Appendix is to prove that the ODE~\eqref{E.ODEintro} blows up. The authors are indebted to Erik Wahl\'en for pointing this fact to them.

Let us introduce the variable $\beta:=\nu_++\nu_-$ (which will only appear in this Appendix, and which should not be mistaken for the constant~$\beta$ we use in the local wellposedness theorem). In terms of this variable, the system~\eqref{E.ODEintro} can be rewritten as
\begin{gather*}
	\nu'=R\cos(\beta+\psi)\sin2\nu\,,\qquad \beta' = 2R\sin(\beta+\psi)\cos 2\nu\,,\\
	R'= R^2\frac{\cos(\beta+\psi)}{\cos 2\nu}\,,\qquad \psi' =-R\frac{\sin(\beta+\psi)}{\cos2\nu}\,.
\end{gather*}
Here we are using the notation $b_1=:R e^{i\psi}$ for the modulus and argument of the complex number~$b_1$, and we recall that the corner angle is $2\nu := \pi+\nu_+-\nu_-$ throughout the paper. The initial data for the ODE at $t=0$ will be denoted by $(\nu^0,\beta^0,R^0,\psi^0)$.

First note that the evolution of the factor $\cos(\beta+\psi)$ appearing in the equation for~$R$ can be readily computed:
\[
\frac d{dt}\cos(\beta+\psi)= \frac R{\cos2\nu} (1-2\cos^22\nu ) \sin^2(\beta+\psi) \,.
\]
This quantity is positive as long as $2\nu\in (\frac\pi4, \frac\pi2)$. Since $\nu'>0$ as long as~$\nu$ is in this interval and $\cos(\beta+\psi)$ remains positive, a simple bootstrap shows that it suffices to pick $\cos(\beta^0+\psi^0)>0$ and $2\nu^0 > \frac\pi4$ to ensure that $\nu$ and $\cos(\beta+\psi)$ are increasing functions of~$t>0$. Therefore, for positive time,
\[
R'=R^2\frac{\cos(\beta+\psi)}{\cos 2\nu}\geq c_0 R^2\,,\qquad c_0:= {\cos(\beta^0+\psi^0)}>0\,.
\]
Whenever $R^0>0$, the solution then blows up as $R(t)\gtrsim (T-t)^{-1}$ for some finite time~$T>0$. Furthermore, one then has $2\nu \rightarrow \frac\pi2$ as $t\rightarrow T$ because $R^2/\tan2\nu$ is constant along the flow. To see this, just note that a straightforward computation yields 
$$
\frac d{dt}\left(\frac{R^2}{\tan 2\nu}\right) = 0.
$$

Note that, in the water wave system, this ODE blowup corresponds to the assertion that, if the the assumptions of our local existence theorem hold up to time~$T$, then as $t\to T$ the angle of the corner tends to~$\frac\pi2$ and the gradient of the velocity at the corner becomes unbounded.

\end{appendices}


\end{document}